\documentclass[opre]{informs4}
\makeatletter
\def\theARTICLETOP{}%

\def\setoddRH{\hbox to \textwidth{\fs.7.8.\tabcolsep0pt
  \begin{tabular*}{\textwidth}[b]{l@{\extracolsep\fill}r}
  {\theRRHFirstLine}& \raisebox{0pt}[0pt][0pt]{\fs.10.10.\thepage}\\[-4pt]
  \rlap{\VRHDW{0.5pt}{0pt}{\textwidth}}&\\
  \end{tabular*}}}
\def\setevenRH{\hbox to \textwidth{\fs.7.8.\tabcolsep0pt
  \begin{tabular*}{\textwidth}[b]{l@{\extracolsep\fill}r}
  \raisebox{0pt}[0pt][0pt]{\fs.10.10.\thepage}& {\theLRHFirstLine}\\[-4pt]
  \rlap{\VRHDW{0.5pt}{0pt}{\textwidth}}&\\
  \end{tabular*}}}

\def\theARTICLETITLE{%
  \HOOKtop
  \begin{center}
    \fs.20.24.\rm\bfseries
    \HD{24}{0}%
    \theTITLE
    \HD{0}{15}%
  \end{center}%
}

\def\AUTHOR#1{\begin{center}\AUTHORfont\HD{15}{0}#1\HD{0}{6}\end{center}}
\def\AFF#1{\begin{center}\AFFfont{#1}\vskip1.6pt\end{center}}
\makeatother

\RequirePackage{tgtermes}
\RequirePackage{newtxtext}
\RequirePackage{newtxmath}
\RequirePackage{bm}
\RequirePackage{endnotes}
\RequirePackage{booktabs}
\usepackage{threeparttable}
\usepackage{array}

\OneAndAHalfSpacedXI 

\usepackage{algorithm}
\usepackage{algpseudocode}
\usepackage{tikz}
\usepackage{graphicx,psfrag,epsf}
\usepackage{wrapfig}
\usepackage{caption}
\usepackage{subcaption}
\usepackage[unicode=true,
 bookmarks=false,
 breaklinks=false,pdfborder={0 0 1},colorlinks=true,citecolor = green!50!black]{hyperref}

\usepackage{natbib}
 \bibpunct[, ]{(}{)}{,}{a}{}{,}%
 \def\bibfont{\small}%

\newcommand{\norm}[1]{\left\lVert #1 \right\rVert}

\newcommand{\indep}{\perp\!\!\!\perp}
\newcommand{\proofend}{\hfill$\blacksquare$}

\EquationsNumberedThrough    

\TheoremsNumberedThrough     
\ECRepeatTheorems  %

\MANUSCRIPTNO{OPRE-2026-03-2857.R1}

\makeatletter
\renewcommand{\theARTICLEABSTRACT}{%
  \HOOKb
  \vspace*{18pt}
  \noindent
  \begin{minipage}[t]{\textwidth}\parindent1em
    \ABSfont
    \noindent\theABSTRACT\endgraf
    \vskip5pt
    \theFUNDING
    \theKEYWORDS
    \theSUBJECTCLASS
    \theAREAOFREVIEW
    \theMSCCLASS
    \theORMSCLASS
    \if@BLINDREV\else\theHISTORY\fi
    \noindent\hrulefill
  \end{minipage}%
  \vspace*{0pt}%
}
\makeatother

\begin{document}
\RUNAUTHOR{Belloni, Chen and Harding}

\RUNTITLE{Adversarial Estimation of Binary Configuration Probabilities Under Sparse Dependence} 

\TITLE{Adversarial Estimation of Binary Configuration Probabilities under Sparse Dependence}

\ARTICLEAUTHORS{%
\AUTHOR{Alexandre Belloni}
\AFF{The Booth School of Business, University of Chicago, \EMAIL{Alex.Belloni@chicagobooth.edu}}

\AUTHOR{Yan Chen}
\AFF{The Fuqua School of Business, Duke University, \EMAIL{yc555@duke.edu}}

\AUTHOR{Matthew Harding}
\AFF{Department of Economics, University of California, Irvine, 
\EMAIL{matthew.harding@uci.edu}}
}

\SUBJECTCLASS{
Statistics: Localized nonparametric adversarial estimation;
Probability: Multivariate binary distributions.
}

\ABSTRACT{We consider the problem of estimating 
probabilities of binary configurations, which is common in operations management applications, including product bundling, advertising, etc. Existing approaches typically model each binary component as a category and apply multinomial models to estimate the choice probabilities. While computationally convenient, these methods do not exploit independence structures in the joint distribution and may therefore be statistically inefficient when the total number of items is large. Using a representation from \citet{bahadur1959representation}, we relate the sparsity of the generalized correlation coefficients to the (in)dependence structure of the binary components. We formulate the problem as estimating a high-dimensional vector of generalized correlation coefficients, together with low or moderate-dimensional nuisance parameters corresponding to the marginal probabilities. We develop a regularized adversarial estimator that attains the optimal rate under standard regularity conditions while remaining computationally feasible. The framework naturally extends to settings with covariates. We apply the proposed estimators to estimating average treatment effects with multiple binary treatments and show substantial finite-sample improvements over non-adaptive methods. Numerical studies corroborate the theoretical results.}


\maketitle

\section{Introduction}\label{sec-intro}
Many empirical applications require estimating the joint distribution of binary configurations, including purchase basket \citep[e.g.][]{chung2003general,gabel2022product}, advertising exposures \citep[e.g.][]{aouad2021display,li2025online}, clinical outcomes
\citep[e.g.][]{gueorguieva2001correlated,agresti2000strategies}, and multiple treatment assignments \citep[e.g.][]{imbens2015causal,arkhangelsky2021double,fitzmaurice2008longitudinal}. In multi-purchase applications, each configuration indicates whether a product is selected, so the joint distribution determines the probability of every possible bundle and captures cross-product purchase dependence. In causal inference settings, the configurations indicate assignment of multiple binary treatments, and their conditional joint distribution yields the generalized propensity score for each treatment combination. Despite their different interpretations, both settings share the same statistical primitive, a
(conditional) joint distribution over a binary configuration.

In many applications, the joint distribution exhibits (partial) independence across products or treatments. For example, in retail demand models, purchase decisions are sometimes approximately independent across product categories \citep[e.g.][]{donnelly2021counterfactual}. In marketing, only a limited set of interaction effects is typically modeled across media channels or competing ads, implying independence across groups of exposures \citep[e.g.][]{zantedeschi2017measuring, danaher2008effect}. In clinical trials, certain combinations of drugs or therapies are designed or assumed to act independently \citep[e.g.][]{pomeroy2022drug, plana2022independent}. To estimate probabilities of the binary configurations for $M$ items, most existing methods do not exploit independence structures among products or treatments and can be inefficient when $M$ is large. Methods that exploit such structure typically rely on strong modeling assumptions, such as sparse multivariate logit models \citep[e.g.][]{hruschka2026analyzing}, factorization models \citep[e.g.][]{gopalan2014bayesian}, low-dimensional latent feature representation \citep[e.g.][]{ruiz2020shopper}, shared confounders \citep[e.g.][]{ranganath2018multiple} or independence of pre-specified categories \citep[e.g.][]{donnelly2021counterfactual}.

In contrast, our approach exploits sparsity in the dependence structure to construct estimators with improved finite-sample performance and stronger theoretical guarantees. We only assume that the underlying dependence structure is sparse, without requiring knowledge of its support. Without additional structural assumptions, the distribution of an $M$-dimensional binary vector involves $2^M - 1$ parameters \citep{bahadur1959representation}, consisting of $M$ marginal probabilities and $2^M - M - 1$ \emph{generalized correlation coefficients}. Independence among some components could reduce the parameter number substantially. So in our setup, sparsity naturally arises from the intrinsic sparse generalized correlation structure across binary components. In Section~\ref{sec:connection:bundle}, we provide examples that connect our formulation to classical bundle and multipurchase choice models, including bundle-level multinomial logit \citep[e.g.][]{jasin2024assortment,bai2024assortment,
tulabandhula2023multipurchase,abdallah2026multipurchase} and Markov choice models \citep[e.g.][]{elhousni2025markovian,blanchet2016markov,feldman2017revenue,
simsek2018expectation,desir2020constrained}. Conditional on a fixed assortment, our framework represents the induced joint distribution of purchase incidences through marginal purchase probabilities and generalized correlation coefficients, without imposing a specific utility-shock distribution or transition process. 

To estimate the marginal probabilities and generalized correlation coefficients in this representation, a natural starting point is the maximum likelihood estimator (MLE). However, the log-likelihood function is highly non-concave in the marginal probabilities, making optimization computationally infeasible, and the MLE tends to overfit when $2^M$ is comparable to or exceeds $N$. Moreover, the MLE fails to adapt to the underlying sparse generalized correlation structure among the binary components. Notably, this non-concavity concerns only the $M$ marginal probabilities, for which consistent estimators are readily available (e.g., sample averages). Conditional on these marginal probabilities, the remaining $2^M - M - 1$ generalized correlation coefficients yield a concave log-likelihood. This naturally motivates a two-stage plug-in approach: first estimate the marginal probabilities, then plug them in to estimate the generalized correlation coefficients via a convex program with an $\ell_1$ penalty to exploit sparsity \citep[e.g.][]{tibshirani1996regression,tibshirani1997Lasso,zou2006adaptive,belloni2014high}. We treat the $M$ marginal probabilities as nuisance parameters and the $2^M - M - 1$ generalized correlation coefficients as the parameters of interest. This formulation leads to a high-dimensional estimand in the presence of low- to moderate-dimensional nuisance parameters. In contrast, the debiased machine learning literature \citep[e.g.,][]{chernozhukov2017double} typically considers low-dimensional target parameters with high-dimensional nuisance components. We show that a plug-in approach induces misspecification in the design matrix and fails to achieve rate optimality. While from a statistical performance we want to adapt to the intrinsic dimension, from a computational perspective, the dimension of the problem is $p \asymp 2^M$. 

To address these challenges, we propose two {regularized adversarial estimators}. The first is an $\ell_1$-regularized estimator using a modified log-likelihood, namely the  minimum log-likelihood over a set of marginal probabilities that are consistent with the data. This conservative approach allows for a lower regularization bias which makes its rate of convergence faster than the plug-in approach. Particularly, the term ``adversarial'' follows established usage in statistics and econometrics \citep[e.g.,][]{chernozhukov2026adversarial,kaji2023adversarial}, which characterizes the max-min structure of our estimator, in which the marginal probability is chosen ``adversarially" from the data-consistent set. However, this estimator is computationally intractable due to the inner minimization over the marginals. We therefore introduce a novel {first-order estimator}, which approximates the log-likelihood via a first-order expansion in the marginal probabilities, yielding an objective that is concave in the marginals. We construct the adversarial set for the marginal probabilities as a hyper-rectangle, so that the inner minimization is attained at one of its vertices. Since the objective remains concave in the generalized correlation coefficients, both the inner minimization and outer maximization are computationally tractable. This estimator matches the performance of the adversarial estimator and is rate-optimal under mild conditions, achieving the same rate as in the oracle case where the marginal probabilities are known. 

\paragraph{Contributions}
First, we propose a novel method for estimating the joint probabilities of binary configurations that adapts to the independence structure among $M$ items using the Bahadur representation \citep{bahadur1959representation} without explicit structural assumptions. Second, we develop an adversarial estimation approach for high-dimensional estimand in the presence of low-dimensional nuisance parameters. When the support of the generalized correlation coefficients contains at most $s$ nonzero terms, without assuming their locations, we show that the oracle estimator achieves a rate of $\sqrt{sM/N}$ (the standard Lasso rate) when the marginal probabilities are known. When the marginals are unknown, the plug-in estimator attains a rate of $M\sqrt{s/N}$, while the first-order estimator achieves $\sqrt{sM/N}+\sqrt{M\log M/N}$. Consequently, the first-order estimator is rate-optimal when $s\gtrsim\log M$. Third, we extend the framework to settings with low-dimensional covariates, which require estimating multiple nuisance functions. We establish a similar convergence rate for the \textit{localized first-order estimator}, which outperforms the \textit{localized plug-in estimator}, as supported by our simulation studies. Finally, we apply our framework to causal inference with many binary treatments, where estimating generalized propensity scores \citep[e.g.,][]{rosenbaum1983central,imbens2000role} is challenging. Exploiting independence structures, our estimators are consistent and efficient, with substantial finite-sample gains when the number of treatment combinations is large. Simulations show that our approach outperforms the multinomial logit model \citep[e.g.][]{imai2004causal,li2019propensity,yang2016propensity} commonly used for multi-level treatments.

\paragraph{Related Literature.} 
Some existing works also study the estimation of joint probabilities of binary configurations based on \citet{bahadur1959representation}. For example, \citet{sinha2010multivariate} fits a Bahadur representation with higher-order terms set to zero, while \citet{yuan2021community} approximates the joint distribution of within-community connectivities using a truncated Bahadur representation. Unlike these works, we consider the full Bahadur representation without assuming that higher-order terms vanish or that the support of the nonzero terms is known; instead, we impose only a sparsity assumption. The additional structure of binary configurations has also motivated alternative modeling approaches \citep[e.g.][]{raginsky2013recursive,banerjee2008model,urschel2017learning,rachkovskij2012similarity,rachkovskij2001representation,rachkovskij2013building,cox1972analysis,rao2004efficiency,gyllenberg1997classification,freund1991unsupervised,ichimura1998maximum}. Specifically, our formulation complements the literature on the conditional independence for undirected graphical models \citep{loh2012structure,morrison2017beyond,van2014new,fan2017high}, including classical Ising model \citep{cipra1987introduction,newell1953theory}. The Ising model, restricted to pairwise dependencies, may still exhibit nonzero higher-order correlations, precluding sparsity in our representation. As discussed in Remark~\ref{rmk:ass1:comparison:ising:graphic}, conditional independence in binary graphical models and sparsity of the Bahadur generalized correlation coefficients \citep{bahadur1959representation} are distinct; Examples~\ref{ex1:sparsity} and~\ref{ex2:sparsity} further show that neither implies the other. Our localized estimator relates to a rich literature on the local maximum likelihood estimation \citep{fan1998local, newey1994kernel, stone1977consistent, tibshirani1987local, lewbel2007local, hansen2008uniform}. Our method is also related to recent developments on adversarial approach, including regularized M-estimation \citep{hirshberg2021augmented}, adversarial Riesz representation \citep{chernozhukov2026adversarial}, and distributionally robust optimization \citep{blanchet2019quantifying,blanchet2019robust}.

\section{Problem Formulation}\label{Sec:Setup}
We have $N$ i.i.d. observations $\{(X_i,Y_i)\}_{i=1}^N$ where $X_i\in\mathcal{X}$ for some compact set $\mathcal{X}\subset\mathbb{R}^d$ such that $d\geq1$ and $Y_i \in\{0,1\}^M$. For example, in the bundle-purchase setting, $X_i$ collects contextual information relevant to customer $i$'s purchase decision, such as product attributes and customer demographics, while $Y_i$ is the corresponding incidence vector indicating which products from the offered assortment $[M]$ are selected. A customer may select any subset of the assortment, including no products, with no restriction on bundle size, and this setting is also considered by \citep[e.g.][]{ruiz2020shopper,donnelly2021counterfactual}. We begin with characterizing $\mathbb{P}(Y_i = y \mid X_i = x)$ by extending the unconditional representation of $\mathbb{P}(Y_i = y)$ by  \citet{bahadur1959representation}. For a binary vector $y \in \{0,1\}^M$ and marginal probabilities $\alpha=(\alpha_{01},\ldots,\alpha_{0M})'\in (0,1)^M$, we define the $j$-th standardized component as 
\begin{equation}\label{eq:z}
\displaystyle z_j(y,\alpha):=(y_{j}-\alpha_{0j})/\sqrt{\alpha_{0j}(1-\alpha_{0j})}.
\end{equation}
For $j\in[M]:=\{1,2,\ldots,M\}$, we denote the $j$-th conditional marginal probability as $\alpha_{0j}(x) = \mathbb{P}(Y_{ij}=1 \mid X_i=x)$, which is the $j$-th component of vector $\alpha_0(x)\in(0,1)^M$. Note $\mathbb{E}[ z_j(Y_i,\alpha_0(x))\mid X_i=x]=0$ and $\mathrm{Var}[ z_j(Y_i,\alpha_0(x))\mid X_i=x]=1$. A $k$-tuple $ \ell = \{j_1,\ldots,j_k\} \subseteq [M]$ such that $k=|\ell|\geq 2$ represents an interaction subset. Since there are in total of $p=2^M-M-1$ such possible subsets, we index each component $l \in [p]$ by its corresponding configuration $\ell = \ell(l)$. Given covariate $x\in\mathcal X$, the \textit{generalized correlation coefficients} are $r_0(x):=\left(r_{01}(x),\ldots,r_{0p}(x)\right)$, where for any $l\in [p]$, 
\begin{equation}\label{def:r0}
r_{0l}(x):= \mathbb{E}\left[ \prod_{j\in\ell(l)}z_j(Y_{i},\alpha_0(X_i)) \bigg| X_i=x\right].
\end{equation}
So each entry $r_{0l}(x)$ corresponds to an $\ell(l)$ such that $|\ell(l)|\geq2$.  \citet{bahadur1959representation} characterizes the joint distribution of a multivariate binary vector in the absence of covariates, and its extension to the conditional distribution of all binary components given covariates is straightforward, following Propositions 1 and 2 in \citet{bahadur1959representation}:
\begin{theorem}\label{Thm:Bahadur}
$\mathbb{P}(Y_i = y \mid X_i=x )= f(y,\alpha_0(x),r_0(x))\prod_{j=1}^M \alpha_{0j}(x)^{y_j}(1-\alpha_{0j}(x))^{1-y_j}$ for any $y \in \{0,1\}^M$ and $x \in \mathcal{X}$, where $\displaystyle f(y,\alpha_0(x),r_0(x)):= 1 + \sum_{\ell(l)\subseteq [M] : \left|\ell(l)\right|\geq2} r_{0l}(x)\prod_{j\in\ell(l)} z_{j}(y,\alpha_0(x))$. 
\end{theorem}
Theorem \ref{Thm:Bahadur} characterizes the dependence between components through the generalized correlation coefficients $r_0(x)=(r_{0l}(x))_{l\in[p]}$ within function $f$. By defining a $p$-dimensional vector $W(\alpha,y)$ where for any $l\in[p]$, its $l$-th component is equal to $W_{l}(\alpha,y):=\prod_{j\in \ell(l)}z_j(y,\alpha)$, we can write the function $f$ as an affine function of the generalized correlation coefficients: $f(y,\alpha_0(x),r_0(x))= 1+W(\alpha_0(x),y)'r_0(x)$.
If all the binary components are independent, $r_{0}(x)=0$ and $f(y,\alpha_0(x),0)=1$, which recovers the conditional likelihood of $y$ given $x$ as $\prod_{j=1}^M \alpha_{0j}(x)^{y_j}(1-\alpha_{0j}(x))^{1-y_j}$ under full independence. According to \eqref{def:r0}, for any $l\in[p]$, if within $\ell(l)$ there is at least one binary component independent from the others conditioning on the covariates, then $r_{0l}(x)=0$. This connects independence structure with zero value of a  generalized correlation coefficient. Thus in applications of interest where independence holds among some of the binary components, we can expect the vector $r_0(x) \in \mathbb{R}^p$ to be sparse. Theorem~\ref{Thm:Bahadur} implies that estimating the binary configuration probabilities requires estimating $\alpha_0(x)$ and $r_0(x)$. Before presenting our theoretical results, we first relate our probability representation to established models of multipurchase choice. We then begin the theoretical analysis with the \textit{no-covariate case}, where the target is the unconditional probability $\mathbb{P}(Y_i=y)$. Building on these results, we extend the proposed estimators to the \textit{with-covariate case}, where the target is the conditional probability $\mathbb{P}(Y_i=y\mid X_i=x)$.

\subsection{Connection to Multipurchase Choice Models}\label{sec:connection:bundle}
Multipurchase models typically induce purchase probabilities over utility-maximizing bundles. Conditional on a fixed assortment $[M]$, our framework directly represents the joint binary distribution through marginal purchase probabilities $\alpha_0$ and generalized correlation coefficients $r_0$, where $\alpha_{0j}=\mathbb{P}(Y_j=1)$ is the marginal purchase probability of product $j$, while the components of $r_0$ are standardized dependence coefficients of the joint bundle distribution. This avoids imposing a particular utility-shock distribution or transition process and permits the estimator to exploit unknown sparsity in the joint purchase probabilities. 
The following examples illustrate how structural interpretations of these parameters arise under two particular bundle-level choice models: bundle-level MNL model and Markov choice model.

\begin{example}[Bundle-level MNL Model]\label{ex:bundle:MNL}
Let $[M]$ denote the offered assortment. For simplicity, we focus on the setting without covariates. Following Theorem~\ref{Thm:Bahadur}, the probability of purchasing the bundle characterized by $y\in\{0,1\}^M$ is $q(y)=f(y;\alpha_0,r_0)\prod_{j=1}^M\alpha_{0j}^{y_j}(1-\alpha_{0j})^{1-y_j}$. Assume first that $q(y)>0$ for every $y$. Define the utility of bundle $y$ as $V_{\alpha,r}(y)=\sum_{j=1}^My_j\log\frac{\alpha_{0j}}{1-\alpha_{0j}}+\log f(y;\alpha_0,r_0)$. If the bundles are treated as alternatives and their utility shocks are i.i.d. Gumbel, then the saturated bundle-level MNL representation yields $q(y)=\frac{\exp(V_{\alpha,r}(y))}{\sum_{z\in\{0,1\}^M}\exp(V_{\alpha,r}(z))}$. If $q(y)=0$, the representation requires the extended utility $V_{\alpha_0,r_0}(y)=-\infty$ (equivalently, exclusion of that
bundle). We next illustrate how generalized correlation coefficients characterize local substitution patterns within a random-utility framework. Suppose a consumer chooses a bundle $Y\in\{0,1\}^M$ and the deterministic utility of bundle $y$ at price $p\in\mathbb{R}_{+}^M$ is $V(y;p)=V_0(y)-\beta p^\top y$ where $\beta>0$. Recall that $\alpha_{0k}=\mathbb{P}(Y_k=1)$. Under the bundle-logit model (with i.i.d. Gumbel utility shocks), for any distinct $j,k\in[M]$, $\frac{\partial\alpha_{0k}}{\partial p_j}=-\beta\mathrm{Cov}(Y_j,Y_k)=-\beta r_{0,jk}\sqrt{\alpha_{0j}(1-\alpha_{0j})\alpha_{0k}(1-\alpha_{0k})}$. Hence, the cross-price elasticity is $\varepsilon_{kj}(p):=\frac{\partial\log\alpha_{0k}}{\partial\log p_j}=-\beta p_j r_{0,jk}\sqrt{\frac{\alpha_{0j}(1-\alpha_{0j})(1-\alpha_{0k})}{\alpha_{0k}}}$. Thus $r_{0,jk}<0$ implies products $j,k$ are local price substitutes; $r_{0,jk}>0$ implies a negative cross-price effect and hence local price complementarity; $r_{0,jk}=0$ implies no first-order cross-price effect. The own-price elasticity is $\varepsilon_{jj}(p)=-\beta p_j(1-\alpha_{0j})<0$.

Furthermore, the structural bundle-MNL specification satisfies independence of
irrelevant alternatives (IIA) at the bundle level. Let
$S' \subseteq S \subseteq [M]$, with feasible bundle sets
$\mathcal{Y}(S') \subseteq \mathcal{Y}(S)$, and define $q(y;S)
=
\frac{\exp\{V(y;p)\}}
{\sum_{w \in \mathcal{Y}(S)} \exp\{V(w;p)\}}$.
Suppose that the utility of each retained bundle remains unchanged when the
assortment contracts from $S$ to $S'$. Then $\frac{q(y;S)}{q(z;S)}
=
\exp\{V(y;p)-V(z;p)\}
=
\frac{q(y;S')}{q(z;S')}$ for any
$y,z \in \mathcal{Y}(S')$.
Thus, assortment contraction does not affect the relative odds of any two
retained bundles. This bundle-level IIA property does not, however, imply an
analogous property for individual-product marginal purchase probabilities,
because products are components of overlapping bundle alternatives rather than
mutually exclusive alternatives. Under the same utility-invariance condition, removing product $j$ from the
full assortment yields the distributions conditional on
$Y_j=0$. Consequently, for any $k \neq j$,
$\mathbb{P}(Y_k=1 \mid Y_j=0)-\mathbb{P}(Y_k=1)
=
-\frac{\operatorname{Cov}(Y_j,Y_k)}{1-\alpha_{0j}}
=
-\frac{
r_{0,jk}
\sqrt{
\alpha_{0j}(1-\alpha_{0j})
\alpha_{0k}(1-\alpha_{0k})
}
}{
1-\alpha_{0j}
}$, where the marginal probabilities and correlations are evaluated under the
full assortment. Hence, $r_{0,jk}<0$ implies that removing product $j$ increases
demand for product $k$, whereas $r_{0,jk}>0$ implies that it decreases demand
for product $k$. As a special case, suppose that each customer purchases at most one product. Then, for every $k \neq j$, $Y_jY_k=0, \operatorname{Cov}(Y_j,Y_k)
=-\alpha_{0j}\alpha_{0k}$, 
which implies $\mathbb{P}(Y_k=1 \mid Y_j=0)
=
\frac{\alpha_{0k}}{1-\alpha_{0j}}$.
Thus, the probability mass associated with product $j$ is redistributed among
the remaining products in proportion to their original choice probabilities.
Moreover, under $V(y;p)=V_0(y)-\beta p^\top y$,
the cross-price elasticity satisfies $\varepsilon_{kj}
:=
\frac{\partial \alpha_{0k}}{\partial p_j}
\frac{p_j}{\alpha_{0k}}
=
-\beta p_j
\frac{\operatorname{Cov}(Y_j,Y_k)}{\alpha_{0k}}
=
\beta p_j\alpha_{0j},\ k \neq j$, recovering the
proportional-substitution property of the classical product-level MNL model. 
\proofend
\end{example}

\begin{example}[Markov Chain Choice Model]\label{ex:markov:choice}
Suppose that $\mathbb{P}(Y=y)=\mathbb{P}(Y_1=y_1)\prod_{j=2}^M\mathbb{P}(Y_j=y_j|Y_{j-1}=y_{j-1})$ and define $q_{j,01}=\mathbb{P}(Y_j=1|Y_{j-1}=0)$, $q_{j,11}=\mathbb{P}(Y_j=1|Y_{j-1}=1)$. Let $\sigma_j^2=\alpha_{0j}(1-\alpha_{0j})$ be the variance of $Y_j$, then $q_{j,01}=\alpha_{0j}-r_{0,\{j-1,j\}}\sigma_j\sqrt{\frac{\alpha_{0,j-1}}{1-\alpha_{0,j-1}}}$, $q_{j,11}=\alpha_{0j}+r_{0,\{j-1,j\}}\sigma_j\sqrt{\frac{1-\alpha_{0,j-1}}{\alpha_{0,j-1}}}$. So the Markov chain can be equivalently parametrized by $(\alpha_{01},\ldots,\alpha_{0M},r_{0,12},r_{0,23},\ldots,r_{0,\{M-1,M\}})$. Particularly, $r_{0,ij}=\prod_{m=i+1}^jr_{0,\{m-1,m\}}$. Given products $i,j\in[M]$, $\mathbb{P}(Y_j=1|Y_i=1)-\mathbb{P}(Y_j=1|Y_i=0)=r_{0,ij}\frac{\sigma_j}{\sigma_i}$. So $r_{0,ij}$ can be viewed as a standardized end-to-end demand-propagation coefficient in the Markov choice model defined. Particularly, suppose a promotion increases $\alpha_{0i}$ while leaving downstream transition kernels unchanged. Then $\frac{d\alpha_{0j}}{d\alpha_{0i}}=r_{0,ij}\frac{\sigma_j}{\sigma_i}$. If $m_j$ is product $j$'s contribution margin, then the system-wide marginal value of increasing purchases of product $i$ is $m_i+\sum_{j>i}m_jr_{0,ij}\frac{\sigma_j}{\sigma_i}$. A promotion on a low-margin product can raise total profit if it increases subsequent purchases of high-margin products. Conversely, a promotion that increases focal-product sales may reduce expected total profit by crowding out profitable downstream
demand.  \proofend 
\end{example}

\paragraph{Notation.} Throughout the paper, $\|\cdot\|_1$, $\norm{\cdot}_2$ (or $\norm{\cdot}$), $\norm{\cdot}_\infty$, and $\norm{\cdot}_0$ denote the $\ell_1$-, $\ell_2$-, $\ell_\infty$-, and $\ell_0$-norms, respectively. For a positive integer $n$, $[n]:=\{1,\ldots,n\}$.
For a matrix $A\in\mathbb{R}^{m\times n}$, $\norm{A}_{\infty,\infty}:=\max_{i\in[m],j\in[n]}|A_{ij}|$, $\norm{A}_{k,k}:=(\sum_{i=1}^m\sum_{j=1}^n|A_{ij}|^k)^{1/k}$, $\norm{A}_0$ denotes the number of nonzero entries of $A$, $\norm{A}_{\mathrm{op}}$ denotes the operator norm of $A$. $\mathbf{1}\{\cdot\}$ denotes the indicator. $\propto$ denotes proportionality. For any vector or matrix $V$, $V'$ denotes the transpose of $V$. Given $\delta\in\mathbb{R}^p$ and $T\subseteq[p]$, define $\delta_T$ as $\delta$ restricted to $T$. For random variables $Z$ and $W$, $Z\indep W$ denotes independence. For any finite-dimensional real vector $\nu$, $\mathrm{dim}(\nu)$ denotes its dimension. $\mathcal{N}(\mu,\sigma^2)$ denotes the Gaussian distribution with mean $\mu$ and variance $\sigma^2$. For deterministic sequences ${a_n}$ and ${b_n}$, we write $a_n = \mathrm{O}(b_n)$ if $|a_n| \le C |b_n|$ for some constant $C>0$ and all sufficiently large $n$, and $a_n = \mathrm{o}(b_n)$ if $a_n/b_n \to 0$ as $n \to \infty$. For stochastic sequences, $X_n = \mathrm{O}_{\mathrm{p}}(a_n)$ means $X_n/a_n$ is bounded in probability, and $X_n = \mathrm{o}_{\mathrm{p}}(a_n)$ means $X_n/a_n \xrightarrow{p} 0$, $X_n=\Omega_p(a_n)$ means $X_n/a_n$ is bounded away from zero with non-vanishing probability. We write $a \gtrsim b$ (respectively, $a \lesssim b$) if, for some universal constant $C>0$ and all sufficiently large $N$, $a \geq Cb$ (respectively, $a \leq Cb$).

\section{Technical Assumptions}
We start with assumptions for the with-covariate case. The no-covariate case can be viewed as a special case where the covariate space $\mathcal{X}=\{x\}$. In this case, $\alpha_0(x)=\alpha_0$, $r_0(x)=r_0$, and $\nabla_x r_0(x)=0$, so that the sparsity condition reduces to $\|r_0\|_0\le s$, corresponding to condition (i) of Assumption~\ref{ass:iid-min-den}. We focus on estimating $r_0(x)$ and $\alpha_0(x)$ at a given covariate $x$ in the interior of $\mathcal{X}$. 

\begin{assumption}[Sparsity and regularity conditions]\label{ass:iid-min-den}
(i) $\displaystyle\sup_{x \in \mathcal{X}} \|r_0(x)\|_0+\|\nabla r_0(x)\|_0 \leq s$. (ii) For some positive constant $\bar{L}$, ${ \displaystyle \max_{x\in\mathcal{X}, y\in \{0,1\}^M}} 1+W(\alpha_0(x),y)^\prime r_0(x) \leq \bar L$ (iii) ${ \displaystyle \min_{i\in [N]}} 1+W(\alpha_0(X_i),Y_i)'r_0(X_i) > \xi>0$ with probability $1-\epsilon$. (iv) $\alpha_{0j}(x)$ is bounded away from $0$ and $1$ (i.e. $\exists\ \eta\in(0,1)$, such that $\eta<\alpha_{0j}(x)<1-\eta$) for all $j\in[M]$, $\forall x\in\mathcal{X}$. (v) $X$ has density $\rho$  such that $\rho(\cdot)\geq\rho_0>0$ on $\mathcal{X}$ and $\rho$ is continuously differentiable in the interior of $\mathcal X$. 
\end{assumption}
Condition (i) of Assumption~\ref{ass:iid-min-den} bounds the degree of sparsity on $r_0(x) \in \mathbb{R}^p$ and $\nabla r_0(x) \in \mathbb{R}^{p\times d}$. When all binary components are independent from each other (for all values of $x\in\mathcal{X}$), we have $s=0$. The analysis also directly extends to a local measure of sparsity (i.e., $s_x$). For instance, when only pairwise interactions are present and higher-order interactions are absent, we have $\|r_0(x)\|_0+\|\nabla r_0(x)\|_0\leq(d+1)\binom{M}{2}\ll 2^M$.
Conditions (ii) and (iii) impose upper and lower bounds on the deviations from full independence. For example, this condition holds in empirical settings where there exist $q_1, q_2>0$ such that the conditional probabilities are bounded within multiplicative factors $q_1$ and $q_2$ of the independent probabilities, i.e., $q_1\prod_{j=1}^M\alpha_{0j}(x)^{y_j}(1-\alpha_{0j}(x))^{1-y_j}\le \mathbb{P}(Y_i=y\mid X_i=x) \le q_2\prod_{j=1}^M\alpha_{0j}(x)^{y_j}(1-\alpha_{0j}(x))^{1-y_j}$. Assumption \ref{ass:iid-min-den} allows for situations where some combinations of treatments have zero probability (so no $Y_i$ equal to that value exists in the data). 

From a modeling perspective, sparsity of $r_0(x)$ can be interpreted as localized dependence conditional on observed characteristics. For example, for customers shopping groceries, pasta, sauce, and cheese may be commonly purchased together, while the purchase of drinks are conditionally independent. Arbitrary dependence within this group can generate three nonzero pairwise coefficients and one three-way coefficient, whereas all coefficients involving other products vanish. Similarly, doctors have the option to provide (simultaneously) multiple treatments to a patient. Although several treatments may share an eligibility rule, others are assigned independently conditional on the patient records. Sparsity then allows rich dependence within a few operationally connected groups while ruling out most cross-group dependence.

\begin{remark}\label{rmk:ass1:comparison:ising:graphic}
The sparsity assumption in Assumption~\ref{ass:iid-min-den} differs from that in sparse Ising or binary graphical models, which impose sparsity on canonical interaction parameters. The two notions are fundamentally distinct: Even in a tree-structured Ising model, most pairwise correlations are typically nonzero, and Bahadur generalized correlation coefficients can be dense despite graph sparsity. Conversely, sparsity in the Bahadur generalized correlation coefficients does not imply a sparse conditional independence graph either. In particular, a distribution with only a single nonzero generalized correlation coefficient can still induce dense graphical conditional dependencies. Examples~\ref{ex1:sparsity} and~\ref{ex2:sparsity} below illustrate our argument. In both examples, $Y_i=(Y_{i1},Y_{i2},Y_{i3},Y_{i4})\in\{0,1\}^4$ denotes the purchase decisions of the $i$-th customer for an assortment of four products. Suppose $\alpha_{0j}=\frac{1}{2}$ for each $j\in[4]$, and $Z_{ij}:=z_{j}(Y_{ij},\alpha_{0j})=2Y_{ij}-1\in\{-1,+1\}$. We use $r_{0,\ell}$ to denote the generalized correlation coefficient associated with the coordinate corresponding to bundle $\ell\subseteq[4]$.
\proofend
\end{remark}
\begin{example}\label{ex1:sparsity}
 Suppose $\mathbb{P}(Z_i=(z_1,z_2,z_3,z_4))\propto\exp\{\theta(z_1z_2+z_2z_3+z_3z_4)\}$ with $\theta\neq0$. This corresponds to a pairwise markov random field of the four-node Ising chain: ``$1-2-3-4$'' with only three edges $(1,2), (2,3), (3,4)$, and each node depends only on its neighbors, so $Z_{i1}\indep(Z_{i3},Z_{i4})|Z_{i2}$, $Z_{i4}\indep (Z_{i1},Z_{i2})|Z_{i3}$, etc. 
This is a sparse conditional independence graph. However, all the pairwise generalized correlation coefficients are nonzero, since for some $\rho\neq0$, $r_{0,12}=r_{0,23}=r_{0,34}=\rho$, $r_{0,13}=r_{0,24}=\rho^2$, $r_{0,14}=\rho^3$. Additionally, $r_{0,1234}=\rho^2$. So graphical model sparsity doesn't imply the sparsity assumption in Assumption~\ref{ass:iid-min-den}.  \proofend
\end{example}
\begin{example}\label{ex2:sparsity}
Suppose $\mathbb{P}(Z_i=(z_1,z_2,z_3,z_4))\propto\frac{(1+\gamma z_1z_2z_3z_4)}{16}$ where $0<|\gamma|<1$. Then $r_{0,jk}=0$ for any pair $\{j,k\}$, $r_{0,jkl}=0$ for any triple $\{j,k,l\}$, but $r_{0,1234}=\gamma\neq0$. So $r_0$ is sparse with the only nonzero element $r_{0,1234}$. Since $\mathbb{P}(Z_{i1}=z_1,Z_{i2}=z_2|Z_{i3}=z_3,Z_{i4}=z_4)\propto 1+\gamma z_1z_2z_3z_4$, it can be checked directly that $\mathbb{P}(Z_{i1}=z_1,Z_{i2}=z_2|Z_{i3}=z_3,Z_{i4}=z_4)\neq\mathbb{P}(Z_{i1}=z_1|Z_{i3}=z_3,Z_{i4}=z_4)\mathbb{P}(Z_{i2}=z_2|Z_{i3}=z_3,Z_{i4}=z_4)$ unless $\gamma=0$. Thus $\mathbb{E}[Z_{i1}Z_{i2}|Z_{i3},Z_{i4}]\neq0$. By symmetry, $\mathbb{E}[Z_{ik}Z_{ij}|Z_{il},Z_{im}]\neq0$ for all distinct $(k,j,l,m)$. So the binary graphical model has all the edges and is not sparse. \proofend
\end{example}
The preceding examples illustrate that Bahadur and Ising sparsity encode different and generally non-nested forms of dependence. We focus on Bahadur sparsity because our objective is to estimate binary configuration probabilities directly: the Bahadur representation expresses these probabilities explicitly in terms of marginal probabilities and generalized dependence coefficients. By contrast, the Ising model relies on recovering conditional-dependence graphs. Although sparsity in its canonical parameters yields a sparse graph, computing marginal or configuration probabilities still requires the partition function, which is a normalizing term obtained by summing over all $2^M$ configurations and depending on the full set of canonical parameters. Thus, Ising sparsity does not generally provide a comparably direct sparse representation of the probabilities that are our primary estimands.

Next, Assumption~\ref{ass:moment} imposes moment conditions on $W(\alpha_0(x),Y)$ and $\nabla_{\alpha}W(\alpha_0(x),Y)'r_0(x)$:
\begin{assumption}[Moment conditions]\label{ass:moment}
Let $\widetilde W(\alpha_0(x),Y):=\frac{[\! W(\alpha_0(x),Y) ;\!\nabla_\alpha W(\alpha_0(x),Y)'r_0(x) ]}{1+W(\alpha_0(x),Y)'r_0(x)}$. There exist absolute constants $\bar{\psi}_2, \bar{K}_m>0$ such that 
${\displaystyle \min_{x\in \mathcal{X}, l\in [p+M]}} \mathbb{E}[ \widetilde W_{l}^2(\alpha_0(x),Y)\mid X=x]^{1/2}\geq\bar{\psi}_2$, ${\displaystyle\max_{x\in \mathcal{X}}\max_{l\in [p+M]}}\mathbb{E}[|\widetilde W_{l}(\alpha_0(X),Y)|^m\mid X=x]^{1/m}\leq\bar{K}_m$, where $m\in\{2,3\}$. 
\end{assumption}
We further define $\bar{K}_{\max}:=\mathbb{E}[ {\max_{l\in [p+M], i \in [N]}}\widetilde W_{l}^2(\alpha_0(X_i),Y_i)]^{1/2}$, which may depend on $N$, $M$, $p$. A simple sufficient condition for Assumption~\ref{ass:moment} is that, uniformly over $x\in\mathcal{X}$ and $l\in[p+M]$, $\widetilde W_l(\alpha_0(x),Y)$ is conditionally sub-Gaussian with scale $\bar{K}$ for some absolute constant $\bar{K}$, and has a conditional second moment bounded away from zero. This holds, for example, when $1+W(\alpha_0(x),Y)'r_0(x)$ is uniformly bounded above and away from zero and the numerator coordinates of $\widetilde{W}(\alpha_0(x),Y)$ are uniformly conditionally sub-Gaussian and nondegenerate. In this case, $\bar{K}_m\lesssim \bar{K}\sqrt{m}$ and $\bar K_{\max}\lesssim \bar{K}\sqrt{\log(N(p+M))}$, which becomes $\mathrm{O}\left(\bar{K}\sqrt{M+\log N}\right)$ since $p\asymp 2^M$. The following assumption specifies the joint asymptotic regime governing $M$, $N$, $h$, $d$, $s$, where $h$ denotes the bandwidth of the kernel estimator introduced later for the with-covariate case.
\begin{assumption}\label{ass:dimension}
$\sqrt{\frac{M^3}{Nh^d}}\leq\delta_N$, $sdh^2\leq\delta_N\xi$, $Nh^{d+4}\geq\frac{\log(NMd)}{d^2}$, $d^2h<1$, $M^3\leq N\delta_N^{2+d/2}\left(\frac{\xi}{sd}\right)^{d/2}$, $\delta_N\leq\min\left\{\frac{1}{2},\frac{s}{d^3\xi},\frac{d^{d/(d+4)}M^{3/2}}{N^{\frac{2}{d+4}}[\log(NMd)]^{\frac{d}{2d+8}}}\right\}$, $\delta_N\rightarrow0$, $h\rightarrow0$, $Nh^d\rightarrow\infty$.
\end{assumption}
In the no-covariate case Assumption~\ref{ass:dimension} implies $\sqrt{M^3/N}\leq\delta_N$. In the with-covariate case, by direct calculation, Assumption \ref{ass:dimension} implies that $h$ satisfies $\left(M^3/(N\delta_N^2)\right)^{1/d}\leq h\leq\sqrt{\delta_N\xi/(sd)}$. This interval is nonempty since we assume $M^3\leq N\delta_N^{2+\frac{d}{2}}\left(\xi/(sd)\right)^{\frac{d}{2}}$. Assumption \ref{ass:smoothness-assumptions:sufficient} in the below imposes smoothness conditions: 
\begin{assumption}[Smoothness conditions]\label{ass:smoothness-assumptions:sufficient}
There exist deterministic sequences $L_N\geq1$ and 
$\rho_{\alpha,N}\rightarrow0$ such that as $N\rightarrow\infty$, $\sqrt{M}L_N\rho_{\alpha,N}=\mathrm{o}(\xi), L_N^2\rho_{\alpha,N}^2=\mathrm{o}(\xi)$, and with probability at least
$1-\delta_\alpha$, the following conditions hold for the no-covariate case: \\
(1) $\alpha_0\in\mathcal{A}(\hat{\alpha}), \sup_{\tilde{\alpha}\in\mathcal{A}(\hat{\alpha})}
\|\tilde{\alpha}-\alpha_0\|
\leq\rho_{\alpha,N}$. For every $i\in[N]$, $a\mapsto W(a,Y_i)$ is twice
continuously differentiable on an open neighborhood of
$\mathcal A(\hat\alpha)$. \\
(2) Uniformly over $\alpha\in\mathcal{A}(\hat{\alpha})$, 
$\|W(\alpha,Y_i)\|_\infty
+\max_{j\in[p], k\in[M]}|\partial_{\alpha_k}W_j(\alpha,Y_i)|
+\max_{j\in[p]}\|\nabla_{\alpha}^2W_j(\alpha,Y_i)\|_{\mathrm{op}}+\|\nabla_{\alpha}\{W(\alpha,Y_i)'r_0\}\|
+\|\nabla_{\alpha}^2\{W(\alpha,Y_i)'r_0\}\|_{\mathrm{op}}
\leq L_N$.\\
\noindent Further, with probability at least $1-\delta_\alpha$, the following conditions hold for the with-covariate case:\\
(1) $\alpha_{*x}(X):=\alpha_0(x)+\nabla_x\alpha_0(x)(X-x) \in \mathcal{A}_x(\hat\alpha)$, $\sup_{\substack{
\tilde{\alpha}\in\mathcal{A}_x(\hat{\alpha}), \|x'-x\|\leq h}}
\|\tilde{\alpha}(x')-\alpha_0(x')\|
\leq\rho_{\alpha,N}$. \\
(2) Uniformly over all $\alpha\in\mathcal{A}_x(\hat\alpha)$, we have $\|W(\alpha(X_i),Y_i)\|_\infty
+\max_{j\in[p], k\in[M]}|\partial_{\alpha_k}W_j(\alpha(X_i),Y_i)|
+\max_{j\in[p]}\|\nabla_{\alpha}^2W_j(\alpha(X_i),Y_i)\|_{\mathrm{op}}+\|\nabla_{\alpha}\{W(\alpha(X_i),Y_i)'r_0(X_i)\}\|
+\|\nabla_{\alpha}^2\{W(\alpha(X_i),Y_i)'r_0(X_i)\}\|_{\mathrm{op}}
\leq L_N$. \\
(3) The components of
$\alpha_0(\cdot)$ and $r_0(\cdot)$ are twice continuously
differentiable in a neighborhood $U_x$ of $x$ containing $\{z:\|z-x\|_2\leq h\}$. There exists a deterministic constant $C_x<\infty$ independent of $N,M,p,h$ such that $\sup_{z\in U_x}\left\{
 \max_{k\in[M]}\|\nabla_z^2\alpha_{0k}(z)\|_{\mathrm{op}}
 +\max_{j\in[p]}\|\nabla_z^2r_{0j}(z)\|_{\mathrm{op}}
 \right\}\leq C_x$.
\end{assumption}
The adversarial sets $\mathcal{A}(\hat{\alpha})$, $\mathcal{A}_x(\hat{\alpha})$ constructed in Section~\ref{sec:marginals} make condition (1) for both no-covariate and with-covariate cases hold with high probability. Condition (3) for the with-covariate case is a general smoothness assumption on $r_0(\cdot)$ and $\alpha_0(\cdot)$. For both the no-covariate and with-covariate settings, condition (2) provides a simple primitive condition for the more technical requirements used in the proofs. We defer these requirements, together with a detailed discussion, to Assumption~\ref{ass:smoothness-assumptions} in the Appendix. These conditions require sufficient smoothness in $\alpha_0(\cdot)$ and $r_0(\cdot)$ near the target covariate $x$ for local-linear approximation. They also require $\mathcal{A}_x(\hat{\alpha})$ to concentrate around $\alpha_0(\cdot)$, therefore to control the impact of estimation errors for the marginal probabilities on the Bahadur representation feature. In retail applications, this requires purchase dependence to vary smoothly with customer and product characteristics; in multi-level treatment applications, it analogously requires smooth changes in treatment-assignment dependence on the covariate information.

\section{Proposed Estimators}\label{sec:estimators}
We now turn to estimating $\alpha_0(\cdot)$ and $r_0(\cdot)$. We first focus on the no-covariate case. We begin with the \textit{maximum likelihood estimator} (MLE) as a natural baseline and discuss its limitations. Building on the MLE, we introduce a \textit{plug-in estimator} that improves upon it both statistically and computationally. However, the plug-in estimator is not rate-optimal. To address this, we consider two \textit{regularized adversarial estimators}. The first, termed the \textit{adversarial estimator}, is a natural extension of the plug-in approach but is computationally infeasible. We therefore propose the \textit{first-order estimator} as a second regularized adversarial estimator, which is computationally feasible and rate-optimal in common parameter regimes under mild conditions. Building on the results established for the no-covariate case, we then extend the proposed estimators to the with-covariate case, where we obtain localized counterparts of the plug-in, adversarial, and first-order estimators. Theoretical properties of these estimators parallel those in the no-covariate case, while the inclusion of covariates introduces additional technical challenges.

We first consider MLE for $(\alpha_0,r_0)$:
$$(\hat \alpha^{mle}, \hat r^{mle}) \in \argmax_{(\alpha,r)\in \mathbb{K}} \sum_{i=1}^N\left\{\log( 1+W(\alpha,Y_i)'r )+ \sum_{j=1}^M Y_{ij} \log(\alpha_j)+ (1-Y_{ij}) \log(1-\alpha_j)\right\}$$
where $\mathbb{K}= \left\{ (\alpha,r)\in(0,1)^M\times\mathbb R^p:  1+W(\alpha,y)'r \geq 0, \forall y\in\{0,1\}^M\right\}$ includes $(\alpha,r)$ leading to non-negative probability estimates for all possible configurations. However, the MLE faces practical limitations: First, the objective is non-concave in $\alpha$ due to the standardized outcomes \eqref{eq:z} in the formulation of vector $W(\alpha,Y_i)$. Second, the number of parameters grows exponentially with $M$, leading to overfitting even for moderate values of $M$. Third, the MLE fails to exploit the underlying independence structure, which could otherwise improve estimation efficiency.

Although the objective function for MLE is non-concave in $\alpha$, fixing $\alpha$, the objective function of MLE becomes a concave maximization problem in $r$. Since independence implies sparsity in the coefficient vector $r_0$, the $\ell_1$-regularized estimator is particularly suitable. Moreover, consistent estimators for $\alpha_0$ (e.g., sample averages) are readily available. A straightforward approach is therefore to first estimate $\alpha_0$ by $\hat{\alpha}=(\hat{\alpha}_1,\ldots,\hat{\alpha}_M)$, plug $\hat{\alpha}$ into the objective of MLE, and apply an $\ell_1$ penalty to $r$. The resulting estimator is referred to as the \emph{plug-in estimator} $\hat r^{\mathrm{PI}}$:
\begin{equation}\label{eq:PMLE-no-covariates}
     \hat r^{\mathrm{PI}}\! \in \!\arg\max_{r}  \displaystyle \frac{1}{N}\sum_{i=1}^N  \log( 1+ W(\hat \alpha,Y_i)'r ) - \lambda \|r\|_{1,w},
\end{equation}
where $\|r\|_{1,w}:= \sum_{j=1}^p w_j |r_j|$ is a weighted $\ell_1$-norm and $\alpha$ are the nuisance parameters. Additionally, the weights are defined as $w_j :=\{\frac{1}{N}\sum_{i=1}^NW_j^2(\hat{\alpha},Y_i)/(1+W(\hat{\alpha},Y_i)'\hat{r}^{(0)})^2\}^{1/2}$ for any $j\in[p]$, where $\hat{r}^{(0)}$ denotes an initial estimator for $r_0$. In particular, when $\alpha_0$ is known, i.e. $\hat{\alpha}=\alpha_0$, we refer to the resulting estimator $\hat{r}^{\mathrm{oracle}}$ as the {oracle (plug-in) estimator}. Theoretically, the ideal implementation would set $w_j =\{\frac{1}{N}\sum_{i=1}^NW_j^2(\alpha_0,Y_i)/(1+W(\alpha_0,Y_i)'r_0)^2\}^{1/2}$, which balances the heterogeneity across components. However, since $\alpha_0$ and $r_0$ are unknown, we instead use their estimates $\hat{\alpha}$ and $\hat{r}^{(0)}$. Although the weights depend on unknown quantities, stable estimates can be obtained via an iterative procedure with mild over-penalization. Note that $\mathrm{dim}(\alpha_0)=M$ and $\mathrm{dim}(r_0)=2^M-M-1=p$, so $\mathrm{dim}(r_0)\gg\mathrm{dim}(\alpha_0)$ as $M$ grows. When $p\gg M$, we can obtain estimators for $\alpha_0$ with a faster rate of convergence than for $r_0$. Specifically, for the rest of the paper, 
\begin{equation}\label{eq:SAA}
    \hat\alpha_j = \frac{1}{N}\sum_{i=1}^N Y_{ij}\qquad\forall j\in [M]
\end{equation}
are used in all no-covariate case estimators. Since $\alpha_{0j}\in(0,1)$, for sufficiently large $N$, $\hat{\alpha}_j(1-\hat{\alpha}_j)>0$ with high probability, so $z_j(y,\hat{\alpha}_j)$ is well defined. In practice if $\hat{\alpha}_j(1-\hat{\alpha}_j)$ becomes numerically close to zero, one can apply a mild truncation or add a small constant to the denominator to avoid division by zero. 
\begin{remark}\label{remark:non-negativity}
In the estimation of $r_0$ for the plug-in estimator, we do not impose the constraint on $\mathbb{K}$, because we are using $\hat\alpha$, and it is possible that $(\hat\alpha,r_0) \not\in \mathbb{K}$. By construction of the objective function (\ref{eq:PMLE-no-covariates}), since $r=0$ is always feasible, by optimality of $\hat{r}^{\mathrm{PI}}$ we have $1+W(\hat\alpha,Y_i)'\hat r^{\mathrm{PI}}> 0$ for every $Y_i$ in the data, so the probability estimates is strictly positive for any outcome vector in the data. We also show in the proofs that under mild conditions, $1+W(\hat\alpha,Y_i)'r_0 > 0$ for any $Y_i$ in the data. Similar arguments hold for the estimators $\hat{r}^A, \hat{r}^{\mathrm{FO}}$ defined later. This result is sufficient in many applications, because practitioners typically only need probability estimates for binary configurations that are actually observed in the data. For example, in A/B testing or recommendation systems, a firm may only evaluate the probability of a limited set of products or advertisements presented to users, rather than all possible combinations of items. Similarly, in retail analytics, firms typically analyze purchase patterns among product bundles that have appeared in historical transactions, without requiring probability estimates for bundles that have never been offered. Related discussion on possible extensions is included in Section~\ref{sec:discussion}. 
\proofend
\end{remark}
Although $\hat{r}^{\mathrm{PI}}$ is computationally tractable, plugging in $\hat{\alpha}$ induces mis-specification in the score function related to \eqref{eq:PMLE-no-covariates}. The resulting discrepancy is dominated by the $\ell_1$ estimation error $\|\hat{\alpha}-\alpha_0\|_1 = \mathrm{O}_{\mathrm{p}}(M/\sqrt{N})$, introducing additional sparsity-dependent factors and yielding a slower convergence rate than the typical Lasso rate. Detailed discussions are provided in Section~\ref{sec:theoretical}. 
Faced with the slower rate of convergence for the plug-in estimator, we propose an \emph{adversarial estimator} $\hat{r}^A$, defined as 
\begin{equation}\label{eq:PMLE-no-covariates-non-linear}
    \displaystyle \hat r^{A} \in \argmax_{r}  \displaystyle\min_{\alpha \in \mathcal{A}(\hat\alpha)} \frac{1}{N}\sum_{i=1}^N\log(1+ W(\alpha,Y_i)'r ) - \lambda \|r\|_{1,w}
\end{equation}
where $\mathcal{A}(\hat{\alpha})$ is an adversarial set such that $\alpha_0 \in \mathcal{A}(\hat{\alpha})$ with high probability. From a statistical perspective, the minimization of  $\alpha \in \mathcal{A}(\hat\alpha)$ acts as a regularization allowing us to better control the mis-specification due to unknown $\alpha_0$. This allows us to use a smaller penalty parameter $\lambda$, mitigating the bias induced by the overly large $\lambda$ required by the plug-in estimator due to overfitting on $\hat{\alpha}$. Since the minimum of concave functions is still concave, the objective function \eqref{eq:PMLE-no-covariates-non-linear} is concave in $r$. However, the inner problem of \eqref{eq:PMLE-no-covariates-non-linear} requires the minimization over $\alpha$ through the highly non-linear and non-concave mapping $W(\alpha,\cdot)$. This computational challenge motivates us to further propose the \emph{first-order estimator} $\hat{r}^{\mathrm{FO}}$, defined as
\begin{equation}\label{eq:PMLE-no-covariates-bilinear}   
    \hat r^{\mathrm{FO}}\in\argmax_{r}  \displaystyle\min_{\alpha \in \mathcal{A}(\hat\alpha)}\frac{1}{N}\sum_{i=1}^N \log\left(1+W^{\mathrm{FO}}(\alpha,Y_i)'r\right)-\lambda \|r\|_{1,w},
\end{equation}
where $W^{\mathrm{FO}}(\alpha,Y_i)=W(\hat\alpha,Y_i)+\nabla_{\alpha} W(\hat\alpha,Y_i)(\alpha-\hat\alpha)$. Estimator (\ref{eq:PMLE-no-covariates-bilinear}) also adopts an adversarial approach but uses a first-order approximation of $W(\alpha,Y_i)$ with respect to $\alpha$ around $\hat{\alpha}$. The estimator has a bilinear form and is concave in $\alpha$ for any fixed $r$. Hence, the minimization of this concave function is attained at the extreme points of $\mathcal{A}(\hat{\alpha})$. For appropriately structured sets $\mathcal{A}(\hat{\alpha})$, such as a hyper-rectangle $\mathcal{A}(\hat{\alpha}) = \prod_{j=1}^M [\hat{\alpha}_j - c_j, \hat{\alpha}_j + c_j]$, where $c_j$ denotes the confidence interval half-width and $\hat{\alpha}_j$ is defined by~\eqref{eq:SAA}, the problem remains computationally feasible. As a result, the discrepancy in the score function becomes negligible since it includes only the second-order bias in the marginal probabilities, in contrast to the plug-in estimator $\hat{r}^{\mathrm{PI}}$. As shown in Theorem~\ref{thm:main:FO}, the $\lambda$ for $\hat{r}^{\mathrm{FO}}$ essentially bounds the $\ell_\infty$-norm of the score as if the marginal probabilities were known.

The insights and estimators from the no-covariate case extend naturally to the with-covariates case. However, the nuisance parameters associated with the marginal probabilities are now functions, which complicates the construction of the adversarial set, as discussed later. Given initial estimates $\hat\alpha(\cdot)=\left(\hat\alpha_1(\cdot),\ldots,\hat\alpha_M(\cdot)\right)$, the localized counterparts of the plug-in estimator~\eqref{eq:PMLE-no-covariates} and the adversarial estimator~\eqref{eq:PMLE-no-covariates-non-linear} are defined in the Electronic Companion by \eqref{eq:PI:with-covariate} and \eqref{eq:adversarial:with-covariate}. Our analysis below focuses primarily on the localized counterpart of the first-order estimator~\eqref{eq:PMLE-no-covariates-bilinear}, which is defined as follows at any covariate value $x$:
$$\hat r^{\mathrm{FO}}(\cdot,x) \in \argmax_{a, b} \min_{\alpha \in \mathcal{A}_x(\hat\alpha)}  \frac{1}{N}\sum_{i=1}^N  K_h\left( \mbox{$X_i-x$}\right)  \log\left( 1+ W^{\mathrm{FO}}(\alpha(X_i),Y_i)'r_{a,b}(X_i,x)\right)- \lambda_x \|r_{a,b}\|_{w_x,h,1},$$
$$\mbox{where}\quad W^{\mathrm{FO}}(\alpha(X_i),Y_i):=W(\hat\alpha(X_i),Y_i)+\nabla_\alpha W(\hat\alpha(X_i),Y_i)(\alpha(X_i)-\hat\alpha(X_i)),$$
$K_h(X_i-x):=\frac{1}{h^d}K\left(\frac{X_i-x}{h}\right)$, $K:\mathbb{R}^d\rightarrow \mathbb{R}$ is the kernel function, and $h$ is the bandwidth. For $a\in \mathbb{R}^p$, $b \in \mathbb{R}^{p\times d}$, $r_{a,b}(X_i,x):= a + b(X_i-x)$, $\|r_{a,b}\|_{w_x,h,1} :=\sum_{j=1}^p w_{xj} |a_j| + h\sum_{j=1}^p\sum_{k=1}^d w_{xjk}|b_{jk}|$, and the weights are $w_{xj} :=\left\{\frac{1}{N}\sum_{i=1}^NK_h(X_i-x)W_j(\hat{\alpha}(X_i),Y_i)^2/(1+W(\hat{\alpha}(X_i),Y_i)'\hat{r}^{(0)}(X_i,x))^2\right\}^{1/2}$, $w_{xjk}:=\left\{\frac{1}{N}\sum_{i=1}^NK_h(X_i-x)W_j(\hat{\alpha}(X_i),Y_i)^2[(X_{ik}-x_k)/h]^2/(1+W(\hat{\alpha}(X_i),Y_i)'\hat{r}^{(0)}(X_i,x))^2\right\}^{1/2}$ for any $j\in[p]$, $k\in[d]$, where $\hat{r}^{(0)}(X_i,x)$ denotes an initial estimator for $r_0(X_i)$. 

The estimates for $\alpha_{0j}(x)$ can be obtained from different approaches. In this paper, we focus on the localized estimator \eqref{def:localized:alpha:estimator} introduced in Section~\ref{sec:theoretical}, where we also describe the construction of the data-dependent adversarial set $\mathcal{A}_x(\hat{\alpha})$. All three objectives for the estimators in the covariate case are concave functions in $r_{a,b}$. At covariate $X_i$, the generalized correlation coefficient vector is $r_0(X_i)$, where $r_0(X_i)=r_0(x)+\nabla r_0(x)(X_i-x)+\Delta$ and $\|\Delta\|_{\infty}=\mathrm{O}\left(\|X_i-x\|^2\right)$ under general regularity conditions. Assumption~\ref{ass:local-linear} requires $K(\cdot)$ to be supported on $\{u\in\mathbb{R}^d:\|u\|\le1\}$. Then, by the definition of $K_h(\cdot)$, for any $X_i$ such that $K_h(X_i-x)\neq0$, we have $\mathrm{O}\left(\|X_i-x\|^2\right)=\mathrm{O}(h^2)$. Hence, $r_{a,b}(X_i,x)$ effectively approximates $r_0(X_i)$ when $X_i$ is close enough to $x$, where $a$ approximates $r_0(x)$ and $b$ approximates its gradient $\nabla r_0(x)$. The minimization over $\alpha(\cdot) \in \mathcal{A}_x(\hat{\alpha})$ is again attained at an extreme point of the nuisance function set $\mathcal{A}_x(\hat{\alpha})$. 

However, care must be taken in specifying this set, as it can easily become too large. For instance, since the nuisance functions depend on $X_i$, one might be tempted to minimize over each $\alpha(X_i)$ separately, but this would create an overly adversarial setup that ignores the dependence across covariates $\{X_i\}_{i\in[N]}$. We define the set of functions $\mathcal{A}_x(\hat\alpha)$ such that the function $\alpha_{*x}(X)=\alpha_0(x)+\nabla_x\alpha_0(x)(X-x) \in \mathcal{A}_x(\hat\alpha)$ with high probability. Additionally, the penalty for the gradient estimation term $b$ is scaled by the bandwidth parameter $h$. Intuitively, since the estimation target is $r_0(x)$, the factor $h$ attenuates the effect of the estimation error in $\nabla r_0(x)$ through $b$. Note that $r_{a,b}(X_i,x)$ provides a first-order approximation to $r_0(X_i)$, and higher-order refinements are possible in principle. Such extensions, however, would require additional regularity and smoothness assumptions. Below we include a remark on higher-order approximations:
\begin{remark}\label{rmk:higher:order:approx}
One could alternatively define the localized estimator with higher order terms. For example, 
$r_{a,b,c}(X_i,x):= a + b(X_i-x) + c[X_i-x,X_i-x]$, where $a\in \mathbb{R}^p, b \in \mathbb{R}^{p\times d}, c \in \mathbb{R}^{p\times d\times d}$, 
$\|r_{a,b,c}\|_{w_x,h,1} :=\sum_{j=1}^p w_{xj} |a_j| + h\sum_{j=1}^p\sum_{k=1}^d w_{xjk} |b_{jk}| + h^2\sum_{j=1}^p\sum_{k=1}^d\sum_{m=1}^d w_{xjkm} |c_{jkm}|$, and $c[X_i-x,X_i-x]$ is a contraction of the $p\times d\times d$ tensor with two copies of $X_i-x$, producing a $p$-dimensional vector. This formulation increases the dimension of the parameter space, requires higher-order smoothness conditions and needs to account for the sparsity in the second-order derivative of $r_0(x)$. Meanwhile it reduces the approximation errors by improving the dependence on the bandwidth $h$. \proofend
\end{remark}

Our local-linear estimator is best suited for low-dimensional covariates because its effective sample size is $Nh^d$. For large $d$, if the conditional distribution depends on $X$ through a low-dimensional representation $\psi(X)\in\mathbb{R}^q$ where $q\ll d$, the estimator can be applied to $\psi(X)$, with $q$ replacing $d$ in the assumptions. Alternatively, under additional structural assumptions, the local-linear estimators could be replaced by regularized high-dimensional methods \citep[e.g.][]{haris2022generalized, dong2025semiparametric}. For example, each component of $\alpha_0(\cdot)$ and $r_0(\cdot)$ may be approximated using only a small subset of covariates \citep[e.g.][]{lafferty2008rodeo,bertin2008selection}, a sum of separate covariate effects \citep[e.g.][]{ravikumar2009sparse,meier2009high}, or a function of a single linear combination of the covariates \citep[e.g.][]{alquier2013sparse,radchenko2015high}.

We now define restricted eigenvalues that characterize the curvature over the restricted sets induced by the regularization, even when the design matrix is globally rank-deficient. 
For a constant $c>1$, let $\bar c:=2\frac{c+1}{c-1}$. The restricted eigenvalues for both no-covariate and with-covariate cases are
\begin{equation}\label{eq:RE}
\begin{array}{rl}
\kappa_{\bar{c}}^2 &\displaystyle := \min_{\|v_{T^c}\|_{w,1} \leq \bar{c}\|v_T\|_{w,1},v\neq0} \frac{1}{N}\sum_{i=1}^N \frac{\{W(\alpha_0,Y_i)'v\}^2}{\{1+W(\alpha_0,Y_i)'r_0\}^2\|v_T\|_{w,2}^2},\\ 
\kappa_{x,\bar c}^2 &\displaystyle := \min_{\|v_{T_x^c}\|_{w_x,h,1} \leq \bar{c}\|v_{T_x}\|_{w_x,h,1}, (v_a,v_b)\neq0} \frac{1}{N}\sum_{i=1}^N \frac{K_h(X_i-x) \{W(\alpha_0(X_i),Y_i)'(v_a+v_b(X_i-x))\}^2}{ \{1+W(\alpha_0(X_i),Y_i)'r_0(X_i)\}^2\|(v_a,v_b)_{T_x}\|_{w_x,h,2}^2},
\end{array}
\end{equation}
where $T$ is the support of $r_0$ in the no-covariate case and $T_x$ is the support of $(r_0(x),\nabla r_0(x))$ in the covariate-case. In the definition for $\kappa_{\bar{c}}^2$, $v\in\mathbb{R}^p$, $\|v_T\|_{w,k}= (\sum_{j\in T} w_j^k|v_j|^k)^{1/k}$. For $v=(v_a',v_b')'$ in the definition for $\kappa_{x,\bar{c}}^2$, $v_a\in\mathbb{R}^p$, $v_b\in\mathbb{R}^{p\times d}$. $\|(v_a,v_b)_{T_x}\|_{w_x,h,k}=\big(\sum_{j\in T_x}w_{xj}^k|v_{aj}|^k+h^k\sum_{(j,l)\in T_x}w_{xjl}^k|v_{b,jl}|^k\big)^{1/k}$. The restricted eigenvalues are standard measures of curvature for analyzing Lasso-type estimators in high-dimensional statistics \citep{bickel2009simultaneous,buhlmann2011statistics}. However, new technical challenges arise in our analysis compared to the classical Lasso setting. For example, we need to account for the non-separability between marginal probabilities and the binary components, as manifested in the structure of the design matrix.

\section{Theoretical Results}\label{sec:theoretical}
\subsection{Rates of Convergence without Covariates}\label{Sec:Rates:no-covariates}
We now turn to the main theoretical results for the no-covariate case. We first compare the penalty parameters $\lambda$ required by the oracle, plug-in, and first-order estimators, then present their convergence rates, and finally discuss the resulting statistical and computational trade-offs.

We begin by defining several quantities that determine the lower bounds on the penalty parameter $\lambda$ for the estimators. When $\alpha_0$ is known so that $\hat{\alpha}=\alpha_0$, define $\lambda^{\mathrm{oracle}}$ as follows to characterize the lower bound on the penalty parameter $\lambda$ for the oracle estimator:
\begin{equation}\label{eq:lambda:oracle}
\lambda^{\mathrm{oracle}}:=\frac{\Phi^{-1}(1-\delta/(2p))}{\sqrt{N}}\max_{j\in[p]} \frac{1}{w_j}\left\{\frac{1}{N}\sum_{i=1}^N \frac{W_j^2(\alpha_0,Y_i)}{(1+W(\alpha_0,Y_i)'r_0)^2}\right\}^{1/2},
\end{equation}
where $\Phi$ denotes the cumulative distribution function of a standard normal random variable. The lower bounds $\underline{\lambda}^{\mathrm{PI}}$ and $\underline{\lambda}^{\mathrm{FO}}$ on the penalty parameters for the plug-in and first-order  estimators can be directly related to the $\lambda^{\mathrm{oracle}}$ as follows
\begin{equation}\label{eq:lambda:estimators}
\underline{\lambda}^{\mathrm{PI}}:=\lambda^{\mathrm{oracle}}+\frac{K_{\alpha,1}(1+\beta_N)\sqrt{M}}{\xi}\|\hat\alpha-\alpha_0\|,\qquad \underline{\lambda}^{\mathrm{FO}}:=\lambda^{\mathrm{oracle}}+\frac{K_{\alpha,2}(1+\beta_N)}{\xi}\norm{\hat{\alpha}-\alpha_0}^2.
\end{equation} 
where $\beta_N := \max_{j\in[p]}\frac{1}{w_j}\frac{1}{N}\sum_{i=1}^N \frac{|W_j(\alpha_0,Y_i)|}{1+W(\alpha_0,Y_i)'r_0}$ and $K_{\alpha,1}$, $K_{\alpha,2}$ as in Assumption~\ref{ass:smoothness-assumptions}.
Self-normalized moderate-deviation theory requires the penalty parameter $\lambda$ to dominate the $\ell_\infty$-norm of the score $S(\alpha,r):=\frac{1}{N}\sum_{i=1}^N\frac{W(\alpha,Y_i)}{1+W(\alpha,Y_i)'r}$ (similarly, denote $S^{\mathrm{FO}}(\alpha,r):=\frac{1}{N}\sum_{i=1}^N\frac{W^{\mathrm{FO}}(\alpha,Y_i)}{1+W^{\mathrm{FO}}(\alpha,Y_i)'r}$). For the oracle estimator we have $\|S(\alpha_0,r_0)\|_\infty=\mathrm{O}_{\mathrm{P}}(\lambda^{\mathrm{oracle}})$. For the plug-in and first-order estimators, the corresponding score norms satisfy with high probability,
$$\begin{array}{rl}
     \displaystyle\|S(\hat \alpha,r_0)\|_\infty\!  \displaystyle\le \lambda^{\mathrm{oracle}}+\mathrm{O}_{\mathrm{P}}(\sqrt{M}\|\hat{\alpha}-\alpha_0\|),  \ \ \ 
     \displaystyle\|S^{\mathrm{FO}}(\alpha_0,r_0)\|_\infty\!\displaystyle\le\lambda^{\mathrm{oracle}}+\mathrm{O}_{\mathrm{P}}(\|\hat{\alpha}-\alpha_0\|^2),
\end{array}$$
where $\lambda^{\mathrm{oracle}}=\mathrm{O}_{\mathrm{P}}(\sqrt{M/N})$. The additional term for the plug-in estimator arises from the estimation of $\alpha_0$ and is $\mathrm{O}_{\mathrm{P}}(\sqrt{M}\|\hat{\alpha}-\alpha_0\|)=\mathrm{O}_{\mathrm{P}}(M/\sqrt{N})$, which dominates the oracle term. Consequently, $\underline{\lambda}^{\mathrm{PI}}=\mathrm{O}(M/\sqrt{N})$. By contrast, for the first-order estimator, the estimation error enters only through a second-order term, leaving $\lambda^{\mathrm{oracle}}$ as the dominant term so $\underline{\lambda}^{\mathrm{FO}}=\mathrm{O}(\sqrt{M/N})$.

Proposition~\ref{prop:PI:infeasible} demonstrates the rate of convergence for the oracle estimator $\hat{r}^{\mathrm{oracle}}$. 
\begin{proposition}[Oracle estimator]\label{prop:PI:infeasible}
Suppose Assumptions \ref{ass:iid-min-den} -- \ref{ass:dimension} and the inequality in Assumption~\ref{ass:smoothness-assumptions:sufficient} (2) hold at $\alpha_0$. Additionally, suppose (i) $\lambda=c\lambda^{\mathrm{oracle}}$ for some $c>1$, (ii) for some constant $C'$, $\kappa_{\bar{c}}^2$ is bounded away from zero and $\bar K_2 + K_{\alpha,1} + K_{\alpha,2}/\xi^2 \leq C^\prime$, where $K_{\alpha,1},K_{\alpha,2}$ are defined in Assumption~\ref{ass:smoothness-assumptions}. Then for $N$ sufficiently large and $C$ independent of $s,N,M$, with probability $1-C\delta-C\delta_\alpha-C\epsilon$ where $\frac{\log^3(p/\delta)}{N}=\mathrm{o}(1)$, 
$$\sqrt{\frac{1}{N}\sum_{i=1}^N \frac{\{W(\alpha_0,Y_i)'(\hat r^{\mathrm{oracle}} - r_0 )\}^2}{(1+W(\alpha_0,Y_i)'r_0)^2}} \leq  C\sqrt{\frac{sM+s\log(1/\delta)}{N}}.$$
\end{proposition}
The oracle estimator has essentially the same rate of convergence as the typical Lasso rate for the high-dimensional linear model $\mathrm{O}_{\mathrm{P}}(\sqrt{s\log p/N})$, as $M\asymp \log p$. Since the oracle estimator treats $\alpha_0$ as known, condition (1) of Assumption~\ref{ass:smoothness-assumptions:sufficient} is unnecessary, and the remaining smoothness requirements need only hold at $\alpha=\alpha_0$. Specifically, for the lower-level conditions underlying Assumption~\ref{ass:smoothness-assumptions:sufficient}, it suffices that conditions (3) and (6) of Assumption~\ref{ass:smoothness-assumptions} hold, while conditions (2) and (4) of Assumption~\ref{ass:smoothness-assumptions} need only hold at $\tilde{\alpha}=\alpha_0$, rather than uniformly over the adversarial set.
\begin{proposition}[Plug-in estimator]\label{prop:plug-in:no-covariate}
Suppose Assumptions \ref{ass:iid-min-den} -- \ref{ass:dimension} and Assumption~\ref{ass:smoothness-assumptions:sufficient} (1)-(2) hold. Suppose (i) for constants $C_2>C_1>1$, $C_1\underline{\lambda}^{\mathrm{PI}}\leq\lambda/c\leq C_2\underline{\lambda}^{\mathrm{PI}}$ with probability at least $1-\delta$ where $\frac{\log^3(p/\delta)}{N}=\mathrm{o}(1)$; (ii) $sM^{1/2}K_{\alpha,1}\|\hat\alpha-\alpha_0\|/ \{\xi\kappa_{\bar c}^2\min_{j\in[p]}w_j^2\}\leq \delta_N$; (iii) $\kappa_{\bar{c}}^2$ is bounded away from zero and $\bar K_2 + K_{\alpha,1} + K_{\alpha,2}/\xi^2 \leq C'$ for some constant $C'$, where in (ii)-(iii), $\delta_N$ is defined in Assumption~\ref{ass:dimension} and $K_{\alpha,1},K_{\alpha,2}$ are defined in Assumption~\ref{ass:smoothness-assumptions}. Then for $N$ sufficiently large and $C$ independent of $s, M, N$, with probability $1-C\delta-C\delta_\alpha-C\epsilon$, 
$$\sqrt{\frac{1}{N}\sum_{i=1}^N \frac{\{W(\alpha_0,Y_i)'(\hat r^{\mathrm{PI}} - r_0 )\}^2}{(1+W(\alpha_0,Y_i)'r_0)^2}}\leq C\sqrt{\frac{s(M+\log(1/\delta))}{N}} + C\sqrt{\frac{sM^2}{N}}+C\sqrt{\frac{sM\log(1/\delta)}{N}}.$$ 
\end{proposition}
Proposition~\ref{prop:plug-in:no-covariate} establishes the rate of convergence for the plug-in estimator. The plug-in estimator requires more technical assumptions than the oracle estimator because it needs to learn the marginal probabilities. Example~\ref{ex:plug-in} in Appendix~\ref{appendix:no-covariate} of the Electronic Companion demonstrates the sharpness of the rate established for $\hat{r}^{\mathrm{PI}}$ in Proposition~\ref{prop:plug-in:no-covariate}. Proposition~\ref{prop:plug-in:no-covariate} implies that the rate of convergence for $\hat{r}^{\mathrm{PI}}$ is $\mathrm{O}_{\mathrm{P}}(\sqrt{sM^2/N})=\mathrm{O}_{\mathrm{P}}(\sqrt{s/N} \log p)$ if $s\geq1$ and $\log(1/\delta)=\mathrm{O}(M)$, which is slower than that of $\hat{r}^{\mathrm{FO}}$ as shown by Theorem~\ref{thm:main:FO} below.
\begin{theorem}[First-order estimator]\label{thm:main:FO}
Suppose Assumptions \ref{ass:iid-min-den} - \ref{ass:smoothness-assumptions:sufficient} hold. Suppose (i) for a constant $C_1>1$, $\lambda/c\geq C_1\underline{\lambda}^{\mathrm{FO}}$ with probability at least $1-\delta$ where $\frac{\log^3(p/\delta)}{N}=\mathrm{o}(1)$; (ii) $s\|\hat\alpha-\alpha_0\|^2 K_{\alpha,2} \bar K_3 / \{\xi\kappa_{\bar c}^2 {\displaystyle \min_{j\in[p]}} w_j^2\}\leq \delta_N$ where $\delta_N$ is defined in Assumption~\ref{ass:dimension} and $K_{\alpha,2}$ is defined in Assumption~\ref{ass:smoothness-assumptions}. Then for $N$ sufficiently large, $C$ independent of $s,N,M$, with probability $1- C\delta - C\delta_\alpha-C\epsilon$,
$$\sqrt{\frac{1}{N}\sum_{i=1}^N \frac{\{W(\alpha_0,Y_i)'(\hat r^{\mathrm{FO}} - r_0 )\}^2}{(1+W(\alpha_0,Y_i)'r_0)^2}}\leq C\lambda\sqrt{s}\frac{(1+1/c)}{\kappa_{\bar{c}}}+C\widetilde{\mathcal{R}}^{1/2}(\hat\alpha),$$
where $\widetilde{\mathcal{R}}(\hat\alpha) \leq C\bar{K}_2\sqrt{\frac{\log (M/\delta)}{N}}{\displaystyle\max_{\alpha \in \mathcal{A}(\hat\alpha)}}\|\alpha-\alpha_0\|_1  + C\frac{K_{\alpha,2}}{\xi^2}{\displaystyle\max_{\alpha \in \mathcal{A}(\hat\alpha)}}\|\alpha - \alpha_0\|^2$.
\end{theorem}
For the first-order estimator, $W^{\mathrm{FO}}(\alpha_0,Y_i)$ is a first-order expansion of $W(\alpha,Y_i)$ around $\hat{\alpha}$, so it introduces only a second-order bias in the score. This bias is smaller than the bias induced by directly plugging $\hat{\alpha}$ into the score. 
By incorporating adversarial selection of $\alpha$, the first-order estimator regularizes the dependence of the score on $\hat{\alpha}$ and mitigates overfitting. As a result, a smaller penalty level $\lambda$ suffices compared with the plug-in estimator, leading to a lower estimation error. Moreover, the formulation of $\hat{r}^{\mathrm{FO}}$ circumvents the non-concavity of the maximum likelihood objective and other computational challenges inherent in adversarial estimation of the marginal probabilities. 

The following corollary specializes Theorem~\ref{thm:main:FO} to a representative case, providing a precise convergence rate under the specific choice of the set $\mathcal{A}(\hat{\alpha})$ defined in~\eqref{def:A:no-covariate}, which admits efficient computation. The details of the construction of $\mathcal{A}(\hat\alpha)$ are in Section \ref{sec:marginals}.
\begin{corollary}\label{cor:FO}
    Under the conditions of Theorem \ref{thm:main:FO}, suppose $\lambda/c\leq C_2\underline{\lambda}^{\mathrm{FO}}$ for some constant $C_2>C_1>1$, $\kappa_{\bar{c}}$ is bounded away from zero and $\bar K_2 + K_{\alpha,1} + K_{\alpha,2}/\xi^2 \leq C'$ for some constant $C'$, $\mathcal{A}(\hat{\alpha})$ is constructed as~\eqref{def:A:no-covariate} and $0<\delta\leq\delta_{\alpha}$. Then for $N$ sufficiently large and $C$ independent of $s,N,M$, with probability $1- C\delta - C\delta_\alpha-C\epsilon$, 
    $$\sqrt{\frac{1}{N}\sum_{i=1}^N \frac{\{W(\alpha_0,Y_i)'(\hat r^{\mathrm{FO}} - r_0 )\}^2}{(1+W(\alpha_0,Y_i)'r_0)^2}} \leq  C\sqrt{\frac{s(M+\log(1/\delta))}{N}} + C\sqrt{\frac{M\log(M/\delta)}{N}}.$$
\end{corollary}
Corollary \ref{cor:FO} shows that the impact of estimated marginal probabilities on the first-order estimator is captured by $\widetilde{\mathcal{R}}(\hat\alpha)=\mathrm{O}_{\mathrm{p}}\left(\frac{M\log M}{N}\right)$ defined as in Theorem~\ref{thm:main:FO}. In many regimes of interest (e.g., $s \gg \log M$), this term is expected to be small as the adversarial estimation is over a nuisance parameter with dimension $M \ll p$. Note that the $\log M$ factor arises from the $\ell_\infty$ construction of $\mathcal{A}(\hat{\alpha})$ in~\eqref{def:A:no-covariate}, which requires controlling the maximum deviation across the $M$ marginal probability estimators. This factor could be avoided by instead defining $\mathcal{A}(\hat{\alpha})$ using an $\ell_2$ bound on the marginal estimation error (i.e., Euclidean balls). However, from a computational standpoint, representing the adversarial set as a hyper-rectangle is considerably more tractable, because the minimum of the inner optimization problem in~\eqref{eq:PMLE-no-covariates-bilinear} is attained at one of its finitely many extreme points. Since $M \asymp \log p$, Corollary~\ref{cor:FO} implies that the first-order estimator attains a convergence rate of order $\sqrt{sM/N}+\sqrt{M\log(M)/N}$. The rate for the oracle estimator with $\alpha_0$ known is of order $\sqrt{sM/N}$. So the first-order estimator is rate optimal when $s\gtrsim\log(M)$. Table \ref{table:comparison} compares all the aforementioned estimators when $s\gtrsim\log(M)$, showing the trade off between rates of convergence and computational complexity. 

\begin{table}[htp!]
\centering
\small
\caption{Comparison of estimators when $s\gtrsim\log(M)$.}
\begin{tabular}{p{2.0cm}p{2.0cm}p{2.cm}p{2.8cm}p{4.8cm}}
\toprule
Estimator & Order of $\lambda$ & Rate & Optimize over $r$ & Evaluation over $\alpha$ \\
\midrule
MLE & $0$ & $\sqrt{p/N}$ & concave & non-concave maximization \\
Plug-in ($\hat{\alpha}$) & $\log p/\sqrt{N}$ & $\sqrt{s/N}\log p$ & concave & direct \\
Adversarial & $\sqrt{\log p/N}$ & $\sqrt{s\log p/N}$ & concave & non-convex minimization \\
First-order & $\sqrt{\log p/N}$ & $\sqrt{s\log p/N}$ & concave & tractable minimization \\
Oracle ($\alpha_0$) & $\sqrt{\log p/N}$ & $\sqrt{s\log p/N}$ & concave & direct \\
\bottomrule
\end{tabular}
\begin{tablenotes}
\small
\item[a] All estimators maximize a concave function in $r$, but differ in how the nuisance parameters $\alpha$ affect the penalty choice and evaluation over $\alpha$.
\end{tablenotes}
\label{table:comparison}
\end{table}

\subsection{Rates of Convergence with Covariates}\label{Sec:Rates:covariates}
The intuition from the no-covariate case extends to the with-covariate case. We therefore focus on the localized first-order estimator $\hat{r}^{\mathrm{FO}}(\cdot,x)$. The analysis builds upon the no-covariate case but must additionally control approximation errors from the local approximation of $r_0(X)$ and $\alpha_0(X)$. 
Theorem~\ref{thm:main:covariate} provides rates of convergence in the with-covariate case, which depends on both $\lambda_x$ and the estimation errors for the marginal probabilities: 
\begin{theorem}\label{thm:main:covariate}
Suppose Assumptions \ref{ass:iid-min-den} -- \ref{ass:smoothness-assumptions:sufficient} hold, with probability at least $1-\delta$ where $\delta\geq\max\left\{\exp(-C'\{\sqrt{\log M}\sqrt{Nh^d}+Nh^d\}),\frac{\log(M)}{\delta_NNh^d}\right\}$ and $\frac{\log^3(pd/\delta)}{Nh^d}=\mathrm{o}(1)$, we have $\lambda_x/c\geq C_1\underline{\lambda}^{\mathrm{FO}}(x)$ for some constant $C_1>1$, $\delta_N\geq\underline{\delta}_{N,x}$ where $\underline{\lambda}^{\mathrm{FO}}(x)$ and $\underline{\delta}_{N,x}$ are defined as \eqref{eq:lambda:FO:covariate} and \eqref{eq:delta:FO:covariate}. Let $r_{*x}(X):=r_0(x)+\nabla_xr_0(x)(X-x)$. Then with probability at least $1-C\delta-C\delta_\alpha-\frac{C(\bar{K}_{\max}/\bar{K}_2)^2}{Nh^d}-C\epsilon$, 
$$\sqrt{\frac{1}{N}\sum_{i=1}^N \frac{K_h(X_i-x)\{W(\alpha_0(X_i),Y_i)'(\hat r^{\mathrm{FO}}(X_i,x) - r_{*x}(X_i) )\}^2}{(1+W(\alpha_0(X_i),Y_i)'r_0(X_i))^2}} \leq C\frac{\lambda_x\sqrt{s}}{\kappa_{x,\bar{c}}} +  \widetilde{\mathcal{R}}^{1/2}(\hat\alpha),$$
where $\widetilde{\mathcal{R}}(\hat\alpha)\leq C\overline{R}_x$, and $C$ is a universal constant independent of $s,M,N,h,d$, and $\overline{R}_x$ is defined as \eqref{bound:R:x}.
\end{theorem}
The convergence rate established in Theorem~\ref{thm:main:covariate} depends on the estimation error of $\alpha_0(\cdot)$ in the adversarial set through the bound $\overline{R}_x$, making accurate estimation of $\alpha_0(\cdot)$ essential. Given the smoothness condition in Assumption~\ref{ass:smoothness-assumptions:sufficient} (3), local linear regression provides a natural approach. Specifically, for each $j\in[M]$, we define
\begin{equation}\label{def:localized:alpha:estimator}
\left(\tilde{\alpha}_j(x), \tilde{\beta}_j(x)\right)\in\arg\min_{\tilde a\in\mathbb{R},\tilde b\in\mathbb{R}^{1\times d}}\frac{1}{N}\sum_{i=1}^NK_h(X_i-x)\left\{Y_{ij} - \tilde a - \tilde b(X_i-x)\right\}^2.
\end{equation}
Since $\alpha_0(x)\in[\eta,1-\eta]$ by Assumption~\ref{ass:iid-min-den}, we truncate $\tilde{\alpha}_{j}(x)$ and set
\begin{equation}\label{eq:truncate-local-linear}
\hat{\alpha}_j(x):=\min\{\max\{\eta,  \tilde{\alpha}_j(x)\} , 1-\eta\},\ \ \forall j\in[M].
\end{equation}
To establish the theoretical properties of this local linear estimator, we impose Assumption~\ref{ass:local-linear}:
\begin{assumption}\label{ass:local-linear}
(i)$\int uu^TK(u)du=\mu_2\mathbf{I}_{d}$, where $\mu_2\neq0$ is a scalar and $\mathbf{I}_{d}$ is the $d\times d$ identity matrix. 
(ii)  $K(\cdot): \mathbb{R}^d\mapsto\mathbb{R}$ is supported on $\{u:\|u\|\leq1\}$. For any $\norm{z}_2\leq1$, $K(-z)=K(z)$ and $\underline{\kappa}\leq K(z)\leq\mathcal{K}$, where $\underline{\kappa}$, $\mathcal{K}$ are positive absolute constants.
\end{assumption}
Assumption~\ref{ass:local-linear} is standard in the local linear regression literature \citep[e.g.][]{ruppert1994multivariate,fan1993local,lu1996multivariate, fan2003nonlinear} and is satisfied by many commonly used kernels, including the uniform (box) kernel $K(u)=\mathbf{1}\{\|u\|\leq1\}$ and the quadratic-positive kernel $K(u)=\mu_0(1+\mu_1\|u\|^2)\mathbf{1}\{\|u\|\leq1\}$. Since $K(\cdot)$ is supported on the unit Euclidean ball by Assumption~\ref{ass:local-linear} (ii), for any $X_i$ such that $K_h(X_i-x)\neq0$, $\mathrm{O}\left(\|X_i-x\|^2\right)=\mathrm{O}(h^2)$. Recall that smoothness of the marginal probability implies $\alpha_0(X_i)=\alpha_0(x)+\nabla \alpha_0(x)(X_i-x)+\Delta_{x,X_i}$ and $\|\Delta_{x,X_i}\|_{\infty}=\mathrm{O}\left(\|X_i-x\|^2\right)$. Hence, $\alpha_{*x}(X_i)=\alpha_0(x)+\nabla_x \alpha_0(x)(X_i-x)$ effectively approximates $\alpha_0(X_i)$ when $X_i$ is close enough to $x$, where $\tilde{\alpha}_j(x)$ approximates $\alpha_{0j}(x)$ and $\tilde{\beta}_j(x)$ approximates its gradient $\nabla\alpha_{0j}(x)$. This local linear approximation affects the analysis but its impact is well controlled.

We next derive the convergence rate of the first-order estimator in the with-covariate case using the adversarial set $\mathcal{A}_x(\hat{\alpha})$ defined in \eqref{def:A:covariate}; details of its construction are provided in Section~\ref{sec:marginals}.
\begin{corollary}\label{cor:FO:covariate}
Suppose the conditions of Theorem \ref{thm:main:covariate} and Assumption \ref{ass:local-linear} hold, $\lambda_x/c\leq C_2\underline{\lambda}^{\mathrm{FO}}(x)$ for some constant $C_2>C_1>1$, $\kappa_{x,\bar{c}}$ is bounded away from zero and $\bar{K}_2 + K_{\alpha,1} + K_{\alpha,2}/\xi^2 \leq C'$ for some constant $C'$. Let the adversarial set $\mathcal{A}_x(\hat{\alpha})$ be constructed as \eqref{def:A:covariate}, define $\Delta_{x,i}^{\mathrm{FO}}:=\frac{W(\alpha_0(X_i),Y_i)'(\hat r^{\mathrm{FO}}(X_i,x) - r_{*x}(X_i) )}{1+W(\alpha_0(X_i),Y_i)'r_0(X_i)}$. Then with probability $1-C\delta-C\delta_\alpha-C\epsilon-C(\bar{K}_{\max}/\bar{K}_2)^2/(Nh^d)$ where $0<\delta\leq\delta_{\alpha}$, we have
\begin{equation}\label{eq:cor:cov:bound-1}
\sqrt{\frac{1}{N}\sum_{i=1}^N K_h(X_i-x)\left(\Delta_{x,i}^{\mathrm{FO}}\right)^2}\leq C\sqrt{\frac{sM+s\log(1/\delta)}{Nh^d}}+C(s+\sqrt{M})dh^2,
\end{equation}
where $C$ is a universal constant independent of $s,N,M,h,d$. The optimal bandwidth choice is $h^*=\min\left\{\sqrt{\frac{\delta_N\xi}{sd}},\max\left\{\left(\frac{M^3}{N\delta_N^2}\right)^{1/d},4^{-\frac{2}{d+4}}\left(\frac{sM+s\log(1/\delta)}{N(s+\sqrt{M})^2}\right)^{\frac{1}{d+4}}\right\}\right\}$, under which with probability $1-C\delta-C\delta_\alpha-C\epsilon-C(\bar{K}_{\max}/\bar{K}_2)^2/[N(h^*)^d]$ where $0<\delta\leq\delta_{\alpha}$, we have $\sqrt{\frac{1}{N}\sum_{i=1}^N K_h(X_i-x)\left(\Delta_{x,i}^{\mathrm{FO}}\right)^2}\leq\bar{C}\min\left\{B_{1},B_2\right\}$, where $B_1:=\delta_N\sqrt{\frac{s(M+\log(1/\delta))}{M^3}}+(\sqrt{M}+s)d\left(\frac{M^3}{N\delta_N^2}\right)^{\frac{2}{d}}$,
$B_2:=\sqrt{\frac{s(M+\log(1/\delta))}{N}}\left(\frac{sd}{\delta_N\xi}\right)^{d/4}\!\!+\frac{(\sqrt{M}+s)\delta_N\xi}{s}$, and $\bar{C}$ is a constant independent of $s,N,M,h,d$.
\end{corollary}
The first term on the right-hand side of \eqref{eq:cor:cov:bound-1} is $\mathrm{O}\left(\sqrt{\frac{s\log(p/\delta)}{Nh^d}}\right)$, which is the typical localized Lasso rate with known $\alpha_0(\cdot)$. The second term, which is $\mathrm{O}((s+\sqrt{M})dh^2)$, captures the approximation error induced by localization. Under the smoothness conditions, the local-linear approximations of $r_0(\cdot)$ and $\alpha_0(\cdot)$ yield second-order Taylor remainders of order $dh^2$. These errors accumulate over the active components of $r_0(\cdot)$ yielding $sdh^2$, and across the $M$ marginal-probability functions in $\alpha_0(\cdot)$ yielding $\sqrt{M}dh^2$.

Ignoring the bandwidth constraints, the first-order condition for minimizing the right-hand side of \eqref{eq:cor:cov:bound-1} yields $h_0=4^{-\frac{2}{d+4}}\left(\frac{sM+s\log(1/\delta)}{N(s+\sqrt{M})^2}\right)^{\frac{1}{d+4}}$. Recall that Assumption~\ref{ass:dimension} requires $h\in[h_L,h_U]$, where $h_L=\left(M^3/(N\delta_N^2)\right)^{1/d}$ and $h_U=\sqrt{\delta_N\xi/(sd)}$. Hence the optimal bandwidth $h^*=\min\{h_U,\max\{h_L,h_0\}\}$. The quantities $B_1$ and $B_2$ are the values of the rate bound evaluated at the two endpoints of this feasible interval: $B_1$ corresponds to $h=h_L$, whereas $B_2$ corresponds to $h=h_U$. Because $h^*$ minimizes the bound over $[h_L,h_U]$, its rate is bounded by $\bar C\min\{B_1,B_2\}$. If neither bandwidth constraint binds, then $h^*=h_0$ and for fixed $s$, $M$, $d$, $\delta$, the resulting prediction error has order $N^{-2/(d+4)}$, which is the classical pointwise convergence rate for estimating a twice-smooth function using local linear regression.

\subsection{Construction of the Adversarial Sets}\label{sec:marginals}
 This section details the construction of the adversarial sets for the no-covariate and with-covariate cases. The key idea is to construct $\mathcal{A}(\hat{\alpha})$ and $\mathcal{A}_x(\hat{\alpha})$ as hyper-rectangles, so that the minimum of the inner optimization problem with respect to $\alpha$ in the first order estimators is attained at one of their finitely many extreme points, thereby enabling efficient computation. Recall that for the no-covariate case, we consider \eqref{eq:SAA} as the initial estimator for the marginal probabilities. We construct the adversarial set in the no-covariate case as
\begin{equation}\label{def:A:no-covariate}
\mathcal{A}(\hat\alpha) :=\prod_{j=1}^M [\hat\alpha_j - cv_{\delta_\alpha} \hat s_j, \hat\alpha_j + cv_{\delta_\alpha} \hat s_j],  
\end{equation}
where $\hat s_j= \sqrt{\hat\alpha_j(1-\hat\alpha_j)}$ and $cv_{\delta_\alpha} := (1-\delta_\alpha)\mbox{-quantile of} \max_{j\in[M]} \left| \frac{1}{N}\sum_{i=1}^N \xi_i \frac{(Y_{ij} - \hat \alpha_j)}{\sqrt{\hat\alpha_j(1-\hat\alpha_j)}}\right|$ conditioning on $\{Y_i\}_{i\in[N]}$ for $\delta_\alpha \in (0,1)$, and $\{\xi_i\}_{i\in[N]}$ are i.i.d. $\mathcal{N}(0,1)$ random variables. 

For the with-covariate case, recall that we use local linear regression estimators \eqref{def:localized:alpha:estimator} and \eqref{eq:truncate-local-linear} for the marginal probabilities. We construct $\mathcal{A}_x(\hat\alpha)$ as a set of linear functions to contain the function $\alpha_{*x}(X_i)=\alpha_0(x)+\nabla_x \alpha_0(x)(X_i-x)$ with high probability. Specifically,
\begin{equation}\label{def:A:covariate}
\mathcal{A}_x(\hat\alpha) :=\prod_{j=1}^M \left\{ \!\alpha_j+\beta_j(X-x) :\!\!\! \begin{array}{l}\alpha_j \in [\hat \alpha_j(x) - cv_x \hat s_{j,0}(x), \hat\alpha_j(x) + cv_x \hat s_{j,0}(x)],  \\ 
\beta_{jk} \in [\tilde \beta_{jk}(x) - cv_x \hat s_{j,k}(x), \tilde \beta_{jk}(x) + cv_x \hat s_{j,k}(x)], k \in [d] 
\end{array}\!\!\right\},
\end{equation}
and conditional on $\{Y_i,X_i\}_{i\in[N]}$, $cv_x := (1-\delta_\alpha)\mbox{-quantile of } {\displaystyle\max_{ j \in  [M], \ell \in\{0\}\cup[d]}}\displaystyle\left|\sum_{i=1}^N \xi_i \frac{\hat \epsilon_{ij} \hat{\zeta}_{i\ell}}{\hat s_{j,\ell}(x)}\right|$, where $\{\xi_i\}_{i\in[N]}$ are i.i.d. $\mathcal{N}(0,1)$ random variables, $\hat\epsilon_{ij}:=Y_{ij} - \hat{\alpha}_j(x)-\tilde\beta_j(x)(X_i-x)$, $\hat{\zeta}_{i} := (Z'W_xZ)^{-1}Z_i K_h(X_i-x)$  with $Z_i=(1, (X_i-x)')'$, $Z=(Z_1,\ldots,Z_N)'$ and $W_{x}$ is a diagonal $N\times N$ matrix with its $i$-th diagonal entry equal to $W_{x,ii}:=K_h(X_i-x)$, $\hat s_{j,\ell}(x) := \sqrt{\sum_{i=1}^N \hat \epsilon_{ij}^2 \hat{\zeta}_{i\ell}^2}$, and $\hat{\zeta}_{i\ell}$ is the $\ell$-th component of vector $\hat{\zeta}_i$, $\forall \ell\in\{0\}\cup[d]$.

\section{Application: Multiple Binary Treatments}\label{section:causal}
In this section, we apply our theoretical results to the estimation of average treatment effects in the presence of $M$ binary treatments. At a high level, we show that exploiting sparsity or independence among treatments enables our estimator to estimate the generalized propensity scores of treatment combinations more efficiently.

We consider a standard potential outcome framework \citep{imbens2015causal}. For a treatment vector $t \in \{0,1\}^M$, the $j$th treatment is treated if and only if $t_j = 1$. We observe the triplets $Z_i=(O_i,T_i,X_i)$ across $i\in[N]$, where $O_i\in\mathbb{R}$ denotes the outcome of interest, $T_i \in \{0,1\}^M$ denotes the treatment assignments, and $X_i \in \mathbb{R}^d$ is the covariate. We are interested in estimating the \textit{average treatment effect} (ATE) between any two treatment levels $t,t'$ of interest, namely $\tau(t,t') = \mathbb{E}[ O(t) - O(t')]$, where $O(t)$ denotes the potential outcome of receiving the treatment $t\in \{0,1\}^M$. We consider the cross-fitted \textit{augmented inverse probability weighted} (AIPW) estimator \citep[e.g.][]{robins1994estimation,robins1995semiparametric,hahn1998role,scharfstein1999adjusting,kennedy2020towards,newey2018cross}. 
Specifically, we split the data into $K$ disjoint folds of size $N/K$ and we use $\mathcal{I}_k$ to denote the $k$-th fold, for $k\in[K]$. It is not required that all the folds have the same size, but we proceed with this merely for simplicity, without loss of generality. For any $t\in\{0,1\}^M$, we let $\hat{\mu}_{(t)}^{-k}(X)$ and $\hat e_t^{-k}(X)$ denote the estimators for $\mu_{(t)}(X)=\mathbb{E}[O|T=t,X]$ and $e_t(X)=\mathbb{P}(T=t|X)$ without using observations from fold $k$. Then we estimate the ATE by
$\hat \tau(t,t')=\!\frac{1}{K}\!\!\sum_{k=1}^K\frac{1}{|\mathcal{I}_k|}\!\sum_{i\in\mathcal{I}_k}\!\left(\hat\mu_{(t)}^{-k}(X_i)\!-\!\hat \mu_{(t')}^{-k}(X_i)\!+\!\frac{\mathbf{1}\{T_i=t\}}{\hat e_t^{-k}(X_i)}\!\!\left(O_i - \hat\mu_{(t)}^{-k}(X_i)\right) - \frac{1\{T_i=t'\}}{\hat e_{t'}^{-k}(X_i)}\!\!\left(O_i - \hat\mu_{(t')}^{-k}(X_i)\right)\right)$.
The generalized propensity scores are estimated based on (\ref{eq:truncate-local-linear}) and the first-order estimator $\hat{r}^{\mathrm{FO}}(\cdot,x)$ at the given covariate of interest $x$: $\hat e_t(x) = \Big(1+W(\hat\alpha(x),t)'\hat r^{\mathrm{FO}}(x,x)\Big)\prod_{j=1}^M\hat\alpha_j(x)^{t_j}(1-\hat\alpha_j(x))^{1-t_j}$. We assume that the sample contains observations with treatment $t$ as well as observations receiving treatment $t'$. We impose the following regularity conditions:
\begin{assumption}\label{ass:causal:extra}
Suppose (i) $\{O_i(a)\}_{a\in\{t,t'\}} \indep T_i \mid X_i$; (ii)  there is a constant $c \in (0,1)$ such that  $c\prod_{j=1}^M \alpha_{0j}(x)^{a_j}(1-\alpha_{0j}(x))^{1-a_j} < e_a(x) < 1- c, \forall x \in \mathcal{X}$, for any $a\in\{t,t'\}$. (iii) For any $a\in\{t,t'\}$, $\sup_{x\in\mathcal{X}}|\hat\mu_{(a)}(x)|\leq\bar{\mu},\max_{i\in[N]}|O_i(a)|\leq\bar{\mu}$ for some absolute constant $\bar{\mu}>0$, $\mathbb{E}[\{\hat \mu_{(a)}(X_i) - \mu_{(a)}(X_i)\}^2]\leq\epsilon_N^2N^{-1+2\gamma}$, $\max_{a\in\{t,t'\}}\sqrt{\frac{1}{N}\sum_{i=1}^N \{\hat \mu_{(a)}(X_i) - \mu_{(a)}(X_i) \}^2} \leq \epsilon_N N ^{-1/2+\gamma}$ with probability $1-\epsilon_N$ for some $0<\gamma\leq1/4$, where $\epsilon_N\to0$.
\end{assumption}
Assumption~\ref{ass:causal:extra} (i) imposes a standard unconfoundedness condition. Assumption~\ref{ass:causal:extra} (ii) imposes a dimension-adaptive overlap requirement: each target configuration must occur with probability at least a constant fraction of the probability it would have under conditionally independent assignment with the same marginal probabilities. Because this lower bound in (ii) decreases exponentially with $M$, the condition allows $e_a(x)$ to decrease with $M$ while ruling out configurations that are disproportionately rare relative to their marginals. This distinction becomes particularly important when the analysis is extended to the setting where many contrasts are considered, because the conventional fixed lower bound on $e_a(x)$ cannot hold uniformly over many binary configurations as $M$ grows. When the propensity score $e_a(x)$ is positive but the lower bound in Assumption~\ref{ass:causal:extra}~(ii) is violated, large inverse-probability
weights can make the ATE estimate sensitive to a few observations.
Truncating these weights limits their influence but may introduce
bias. Alternatively, restricting or reweighting the target population
to emphasize covariate regions with greater treatment overlap changes
the target estimand \citep[e.g.,][]{crump2009dealing,li2019propensity}.
If $e_a(x)=0$ on a covariate region of positive probability so treatment
$a$ is never observed in that region, the full-population ATE involving
this treatment is therefore generally not identified without additional
assumptions. Possible remedies include restricting the target population or treatment contrast, model-based extrapolation, or partial identification.
Assumption~\ref{ass:causal:extra} (iii) requires the outcome regression estimators to satisfy boundedness
and convergence-rate conditions.
\begin{theorem}\label{thm:causal-inference:AIPW}
Suppose Assumptions \ref{ass:iid-min-den} -- \ref{ass:causal:extra} hold. Suppose $\kappa_{x,\bar{c}}$ is bounded away from zero, $\bar{K}_2 + K_{\alpha,1} + K_{\alpha,2}/\xi^2\leq C'$ and $\mathcal{A}_x(\hat{\alpha})$ is constructed as before. Additionally, suppose for some $\psi_N$ and $0<\gamma\leq\frac{1}{4}$, $\psi_NN^\gamma<1$, $\psi_NN^\gamma\to0$, and for some constant $C$, 
\begin{equation}\label{consistency:ATE:condition}
\sqrt{s}\left(\sqrt{\frac{sM}{Nh^d}}\!+(s+\sqrt{M})dh^2\right)\leq\frac{C\psi_N}{\sup_{x\in\mathcal{X},a\in\{t,t'\}}\!\prod_{j=1}^M\!\alpha_{0j}(x)^{a_j}(1-\alpha_{0j}(x))^{1-a_j}}, 
\end{equation} 
\begin{equation}\label{consistency:ATE:2}
Mdh^2+\frac{M}{\sqrt{Nh^d}}\leq C\psi_N,
\end{equation}
where $M=\mathrm{o}(\min\{\log(1/\psi_N),\log(1/\epsilon_N)\})$, then $\sqrt{N}\hat \Sigma^{-1/2} (\hat \tau(t,t') - \tau(t,t')) \to\mathcal{N}(0,1)$, where $\hat{\Sigma}$ is the semiparametric efficiency variance bound. \proofend
\end{theorem}
Theorem~\ref{thm:causal-inference:AIPW} shows that treatment-induced sparsity in $r_0(x)$ allows the generalized propensity estimator to satisfy the nuisance-rate conditions required for semiparametrically efficient AIPW inference when $s\ll 2^M$ under the stated growing-dimensional regime. The semiparametric efficiency calculation itself follows standard arguments for AIPW estimators \citep[e.g.,][]{wager2020stats,imbens2015causal}, and our contribution lies in the generalized propensity estimator and its growing-dimensional rate guarantees. Although Theorem~\ref{thm:causal-inference:AIPW} is stated for a prespecified pair $(t,t')$, it applies to any fixed pair satisfying the stated conditions. This matters when inference is required for an arbitrary contrast among a large collection of candidate treatment combinations. We discussed above that Assumption~\ref{ass:causal:extra} (ii) accommodates this setting. Remark~\ref{rmk:causal:2} shows the compatibility of the conditions in Theorem~\ref{thm:causal-inference:AIPW} and provide two examples of $(N,M,d,s)$ satisfying them. 
\begin{remark}\label{rmk:causal:2}
We illustrate the compatibility of
Assumption~\ref{ass:dimension} and the conditions of
Theorem~\ref{thm:causal-inference:AIPW}.
Fix the covariate dimension $d\geq1$ and $\xi>0$. Suppose $0<\gamma\leq\frac14, \gamma<\frac{2}{d+4}$, and choose a constant $\eta$ satisfying
$\gamma<\eta<2/(d+4)$. Set $M=\left\lceil(\log N)^{3/4}\right\rceil, \psi_N=N^{-\eta}, \epsilon_N=N^{-\gamma/2}$, choose $h=\left(\frac{\log(NMd)}{Nd^2}\right)^{1/(d+4)}, \delta_N=\frac{M^{3/2}}{\sqrt{Nh^d}}$.
Two admissible values of sparsity parameter $s$ are $s=\lceil\log M\rceil$ and $s=\lceil\sqrt M\rceil$. For the rest of this remark, we show that the given choices of $(s,N,M,d,h)$ satisfy all the imposed conditions. 

First,
$\psi_NN^\gamma=N^{\gamma-\eta}\to0$ and $M=\mathrm{o}\!\left(
\min\{\log(1/\psi_N),\log(1/\epsilon_N)\}
\right)$. Write
$v_N:=(Nh^d)^{-1/2}$ and $b_N:=dh^2$.
The bandwidth choice gives $b_N=\sqrt{\log(NMd)}\,v_N, M(v_N+b_N)=\mathrm{o}(\psi_N)$. This verifies \eqref{consistency:ATE:2}.
For either value of $s$, $\sqrt{s}\left(
\sqrt{\frac{sM}{Nh^d}}+(s+\sqrt M)dh^2
\right)=
s\sqrt M\,v_N+\sqrt{s}(s+\sqrt M)b_N
=\mathrm{O}\!\left(M(v_N+b_N)\right)
=\mathrm{o}(\psi_N)$. Since the product of marginal probabilities in the denominator
of \eqref{consistency:ATE:condition} is at most one,
that condition also holds. Finally, the chosen $\delta_N$ equals the third upper bound
in Assumption~\ref{ass:dimension}, and $\frac{sdh^2}{\xi\delta_N}
=
\frac{s\sqrt{\log(NMd)}}{\xi M^{3/2}}
\rightarrow0$ for both sparsity regimes. Consequently, $M^3=Nh^d\delta_N^2
\leq
N\delta_N^{2+d/2}
\left(\frac{\xi}{sd}\right)^{d/2}$. The remaining requirements follow from $Nh^{d+4}=\frac{\log(NMd)}{d^2}$, $\delta_N\to0$, $\frac{\delta_Nd^3\xi}{s}\to0, d^2h\to0, Nh^d\to\infty$. The restriction $\gamma<2/(d+4)$ is equivalent to
$d<2/\gamma-4$. In particular, $\gamma=1/4$ requires
$d\leq3$, whereas smaller values of $\gamma$ permit larger
fixed covariate dimensions.
\proofend
\end{remark}
Our result is related to the doubly robust kernel estimators. The condition in standard semiparametric settings requires $\mathrm{o}_{\mathrm{P}}(N^{-1/4})$ consistency \citep[e.g.][]{chernozhukov2017double}. For non-parametric methods, \citet{kennedy2017non} estimates continuous treatment effect and derives $\sqrt{N}$-local rate of convergence with a kernel smoothing approach, by requiring that $h=\mathrm{O}(N^{-1/5})$, and the outcome estimation and propensity score estimation are both required to be $\mathrm{o}_{\mathrm{P}}(N^{-1/5})$. \citet{rothe2019properties} proposes a semiparametric two-step doubly robust estimators by requiring $\mathrm{o}_{\mathrm{P}}(N^{-1/6})$ consistency for both outcome and propensity score estimation, as long as their product is $\mathrm{o}_{\mathrm{P}}(N^{-1/2})$. 
Our estimator allows a slower rate of convergence for estimating the generalized propensity scores by imposing additional sparsity assumptions with a prudent choice of $N,M,d,h$. 

In the Electronic Companion, Section~\ref{sec:numerical:M:4} reports numerical results for $M=4$, while Section~\ref{appendix:ATE:numerical} presents results for a causal-inference setting that estimates the ATEs of all 15 noncontrol treatment configurations relative to the control $(0,0,0,0)$. Across values of $N$ and $s$ spanning sparse to dense settings, AIPW estimators using generalized propensity scores from our methods outperform those based on the MNL and Nadaraya–Watson estimators. Further details are provided there.

\section{Numerical Simulations}\label{sec:simulations}
We construct a synthetic product bundle purchase scenario with $M=15$, so $p=2^{15}-15-1=32{,}752$. Each observation $i$ represents a customer's shopping decision: $Y_{ij}=1$ indicates that category $j$ is purchased, and $Y_i\in\{0,1\}^{15}$ records the resulting bundle from fifteen categories: milk, bread, eggs, butter, cereal, coffee, tea, cookies, pasta, pasta sauce, cheese, yogurt, chips, salsa, and beer. Each customer has $d=3$ covariates $x_i=(x_{i1},x_{i2},x_{i3})\in[0,1]^3$, representing percentiles of household size, income, and the age of the shopper. We draw a latent vector $\xi_i\sim\mathcal{N}(0,C)$, where $C$ has unit diagonal entries $C_{11}=C_{22}=C_{33}=1$, and off-diagonal entries $C_{12}=C_{21}=0.25$, $C_{23}=C_{32}=0.30$, $C_{13}=C_{31}=-0.20$, corresponding to household size--income, income--age, and household size--age dependence. We set $x_{ik}=\Phi(\xi_{ik})$, where $\Phi$ is the standard normal CDF. We specify each product category's conditional purchase probability using the logistic specification
$\alpha_{0j}(x)=1/\{1+\exp[-(\theta_{j0}+\theta_j^{\top}x)]\}$.
Let $\bar{x}=(0.5,0.5,0.5)$ denote the reference profile at which each percentile-transformed covariate equals its marginal median. For each category, we choose the slope coefficients $\theta_j$ and calibrate the intercept $\theta_{j0}$ so that the purchase probabilities remain within $[0.2,0.8]$ over the covariate cube. The generalized correlation coefficient vector $r_0(x)$ is supported on $s=6$ bundles:
pasta--sauce $r_{0,\{9,10\}}(x)=0.210+0.016(x_2-0.5)$;
chips--salsa $r_{0,\{13,14\}}(x)=0.178-0.0128(x_3-0.5)$; 
milk--cereal $r_{0,\{1,5\}}(x)=0.136+0.016(x_1-0.5)$; 
chips--beer $r_{0,\{13,15\}}(x)=0.105-0.0128(x_3-0.5)$; 
coffee--tea substitutes $r_{0,\{6,7\}}(x)=-0.063-0.0064(x_3-0.5)$; 
and the third-order milk--bread--eggs breakfast basket $r_{0,\{1,2,3\}}(x)=0.059+0.0096(x_1-0.5)$. 
All remaining coordinates of $r_0(x)$ are zero. We verify numerically that $1+\sum_l r_{0l}(x)W_l(\alpha_0(x),y)>0$ 
for all $y\in\{0,1\}^{15}$ and all $x$ on a $7^3$ grid of the covariate cube.

The no-covariate case aims at the median customer,
$\alpha_0=\alpha_0(\bar x)$ and $r_0=r_0(\bar x)$, whose six nonzero entries are $(0.210,\,0.178,\,0.136,\,0.105,\,-0.063,\,0.059)$ following the order listed above. We draw $N=2{,}500$ bundles per replication and estimate $r_0$, where we use the \emph{sample average approximation} (SAA) estimator as a baseline, whose $l$-th component is equal to $\frac{1}{N}\sum_{i=1}^N\prod_{j\in\ell(l)}z_j(Y_i,\hat{\alpha}_j), \forall l\in[p]$. Per replication of the with-covariate case, we draw $N=8{,}000$ customers and estimate $r_0(x)$ at two profiles:
$x_{\mathrm{L}}=(0.70,0.50,0.40)$, representing a larger household with median income and a younger shopper, and
$x_{\mathrm{S}}=(0.40,0.65,0.60)$, representing a smaller household with higher income and an older shopper. The local-linear regression uses a product uniform kernel and bandwidth $h=0.32$. The local neighborhoods contain on average $2{,}199$ and $2{,}242$ observations with nonzero kernel weights at $x_{\mathrm{L}}$ and $x_{\mathrm{S}}$. The with-covariate case uses the \emph{Nadaraya-Watson} (NW) estimator as baseline, whose $l$-th component is equal to $\frac{\sum_{i=1}^NK_h(X_i-x)\prod_{j\in\ell(l)}z_{j}(Y_{i},\hat{\alpha}(X_i))}{\sum_{i=1}^NK_h(X_i-x)}$, $\forall l\in[p]$.

We compare the plug-in, first-order and oracle estimators. We consider two penalty-weighting schemes, corresponding to different choices of $\{w_k\}_{k\in[2^M-M-1]}$ in \eqref{eq:PMLE-no-covariates} and \eqref{eq:PMLE-no-covariates-bilinear} in the no-covariate case: (I) $w_{k}=\sqrt{\frac{1}{N}\sum_{i=1}^NW_{k}(\hat{\alpha},Y_i)^2}$, corresponding to the initialization $\hat{r}^{(0)}=0$; (II) $w_{k}=\sqrt{\frac{1}{N}\sum_{i=1}^N \frac{W_{k}^2(\hat{\alpha},Y_i)}{(1+W_i^\prime r_0)^2}}$, corresponding to the oracle initialization $\hat{r}^{(0)}=r_0$. This initialization is infeasible because $r_0$ is unknown, and we include it as a benchmark. For the with-covariate case: (I) $w_{k}=\sqrt{\sum_{i=1}^N\frac{K_h(X_i-x)}{N}W_{k}(\hat{\alpha}(X_i),Y_i)^2}$, $w_{k,1}=\sqrt{\sum_{i=1}^N\frac{K_h(X_i-x)}{N}W_{k}(\hat{\alpha}(X_i),Y_i)^2(X_{i}-x)^2}$, corresponding to $\hat{r}^{(0)}(X_i,x)=0$; (II) $w_{k}=\sqrt{\sum_{i=1}^N\frac{K_h(X_i-x)W_{k}(\hat{\alpha}(X_i),Y_i)^2}{N[1+W(\hat{\alpha}(X_i),Y_i)^\prime r_0(X_i)]^2}}$, $w_{k,1}=\sqrt{\sum_{i=1}^N\frac{K_h(X_i-x)W_{k}(\hat{\alpha}(X_i),Y_i)^2(X_{i}-x)^2}{N[1+W(\hat{\alpha}(X_i),Y_i)^\prime r_0(X_i)]^2}}$, corresponding to $\hat{r}^{(0)}(X_i,x)=r_0(X_i)$. For any $y\in\{0,1\}^M$, let $p_y$ be the true probability of $y$, and $\hat{p}_y=
(1+W(\hat{\alpha},y)^\prime \hat{r})\prod_{j=1}^M\hat{\alpha}_j^{y_j}(1-\hat{\alpha}_j)^{1-y_j}$. We compute three metrics, \textit{root mean squared error} (RMSE) $\mathbb{E}[\norm{\hat{r}-r_0}_2^2]^{1/2}$, $\mathbb{E}[\max_{y\in\{0,1\}^M}|\hat{p}_y-p_y|]$ and $\mathbb{E}[\sum_{y\in\{0,1\}^M}p_y|\hat{p}_y-p_y|]$. When $\hat{p}_y \notin [0,1]$, we truncate it to $[0,1]$. For the no-covariate case, $\lambda_{cv}$ denotes the $\lambda$ chosen via cross-validation, $\lambda_{\mathrm{theory}}$ denotes the theoretically derived penalty parameter, chosen at the scale prescribed by Theorem~\ref{thm:main:FO} for the first-order estimators and Proposition~\ref{prop:plug-in:no-covariate} for the plug-in estimators. Specifically, $\lambda_{\mathrm{theory}}^{\mathrm{PI}}=\frac{\Phi^{-1}(1-\delta/(2p))}{\sqrt{N}}+\frac{M}{\sqrt{N}}$, $\lambda_{\mathrm{theory}}^{\mathrm{FO}}=\frac{\Phi^{-1}(1-\delta/(2p))}{\sqrt{N}}+\frac{M}{N}$, $\lambda_{\mathrm{theory}}^{\mathrm{PI}_{\mathrm{oracle}}}=\frac{\Phi^{-1}(1-\delta/(2p))}{\sqrt{N}}$. Similarly, for the with-covariate case, $\lambda_{\mathrm{theory}}^{\mathrm{PI}} = \frac{\Phi^{-1}(1 - \delta / 2p(d + 1))}{\sqrt{Nh^d}} + \frac{M}{\sqrt{Nh^d}} + dMh^2$,
$\lambda_{\mathrm{theory}}^{\mathrm{FO}} = \frac{\Phi^{-1}(1 - \delta / 2p(d + 1))}{\sqrt{Nh^d}} + \frac{M}{Nh^d} + Md^2h^4$, $\lambda_{\mathrm{theory}}^{\mathrm{PI}_{\mathrm{oracle}}} = \frac{\Phi^{-1}(1 - \delta / 2p(d + 1))}{\sqrt{Nh^d}}$, where ``PI'', ``FO'' and ``oracle'' refer to plug-in, first-order and oracle (plug-in) estimators. 

\begin{table}[!htbp]
\centering
\small
\setlength{\tabcolsep}{4pt}
\renewcommand{\arraystretch}{0.9}
\begin{threeparttable}
\caption{Simulation results for $M=15$.}
\label{tab:M15-results}
\begin{tabular}{@{}lcccccc@{}}
\toprule
& \multicolumn{2}{c}{RMSE}
& \multicolumn{2}{c}{$\mathbb{E}\!\left[
\max_{y\in\{0,1\}^M}\left|\hat{p}_y-p_y\right|\right]$}
& \multicolumn{2}{c}{$\mathbb{E}\!\left[
\sum_{y\in\{0,1\}^M}p_y\left|\hat{p}_y-p_y\right|\right]$} \\
\cmidrule(lr){2-3}\cmidrule(lr){4-5}\cmidrule(lr){6-7}
Estimator
& $\lambda_{\mathrm{cv}}$ & $\lambda_{\mathrm{theory}}$
& $\lambda_{\mathrm{cv}}$ & $\lambda_{\mathrm{theory}}$
& $\lambda_{\mathrm{cv}}$ & $\lambda_{\mathrm{theory}}$ \\
\midrule

\multicolumn{7}{@{}l}{\quad \textit{Panel A:} \textit{no-covariate} $N=2{,}500$, $M=15$, $s=6$} \\
SAA
& \multicolumn{2}{c}{3.655}
& \multicolumn{2}{c}{0.00148}
& \multicolumn{2}{c}{0.00011} \\
PI ($w_{\mathrm{I}}$)
& 0.169 & 0.336 & 0.00010 & 0.00021 & 0.00001 & 0.00002 \\
PI ($w_{\mathrm{II}}$)
& 0.154 & 0.336 & 0.00010 & 0.00021 & 0.00001 & 0.00002 \\
FO ($w_{\mathrm{I}}$)
& 0.148 & 0.224 & 0.00010 & 0.00015 & 0.00001 & 0.00001 \\
FO ($w_{\mathrm{II}}$)
& 0.142 & 0.236 & 0.00010 & 0.00016 & 0.00001 & 0.00001 \\
PI (oracle $w_{\mathrm{I}}$)
& 0.138 & 0.211 & 0.00008 & 0.00013 & 0.00001 & 0.00001 \\
PI (oracle $w_{\mathrm{II}}$)
& 0.139 & 0.223 & 0.00008 & 0.00013 & 0.00001 & 0.00001 \\
\midrule

\multicolumn{7}{@{}l}{\quad \textit{Panel B:} \textit{with-covariate} $x_{\mathrm{L}}=(0.7,0.5,0.4)$, $N=8{,}000$, $d=3$} \\
NW
& \multicolumn{2}{c}{3.907}
& \multicolumn{2}{c}{0.00157}
& \multicolumn{2}{c}{0.00010} \\
PI ($w_{\mathrm{I}}$)
& 0.147 & 0.338 & 0.00009 & 0.00019 & 0.00001 & 0.00002 \\
PI ($w_{\mathrm{II}}$)
& 0.142 & 0.338 & 0.00009 & 0.00019 & 0.00001 & 0.00002 \\
FO ($w_{\mathrm{I}}$)
& 0.139 & 0.240 & 0.00009 & 0.00015 & 0.00001 & 0.00001 \\
FO ($w_{\mathrm{II}}$)
& 0.143 & 0.252 & 0.00010 & 0.00015 & 0.00001 & 0.00001 \\
PI (oracle $w_{\mathrm{I}}$)
& 0.141 & 0.219 & 0.00007 & 0.00012 & 0.00001 & 0.00001 \\
PI (oracle $w_{\mathrm{II}}$)
& 0.140 & 0.231 & 0.00007 & 0.00013 & 0.00001 & 0.00001 \\
\midrule

\multicolumn{7}{@{}l}{\quad \textit{Panel C:} \textit{with-covariate} $x_{\mathrm{S}}=(0.4,0.65,0.6)$, $N=8{,}000$, $d=3$} \\
NW
& \multicolumn{2}{c}{3.862}
& \multicolumn{2}{c}{0.00164}
& \multicolumn{2}{c}{0.00011} \\
PI ($w_{\mathrm{I}}$)
& 0.149 & 0.336 & 0.00010 & 0.00020 & 0.00001 & 0.00002 \\
PI ($w_{\mathrm{II}}$)
& 0.145 & 0.336 & 0.00010 & 0.00020 & 0.00001 & 0.00002 \\
FO ($w_{\mathrm{I}}$)
& 0.143 & 0.241 & 0.00010 & 0.00016 & 0.00001 & 0.00001 \\
FO ($w_{\mathrm{II}}$)
& 0.147 & 0.253 & 0.00011 & 0.00017 & 0.00001 & 0.00001 \\
PI (oracle $w_{\mathrm{I}}$)
& 0.142 & 0.220 & 0.00007 & 0.00013 & 0.00001 & 0.00001 \\
PI (oracle $w_{\mathrm{II}}$)
& 0.140 & 0.231 & 0.00008 & 0.00014 & 0.00001 & 0.00001 \\
\bottomrule
\end{tabular}

\begin{tablenotes}
\footnotesize
\item PI and FO denote the plug-in and first-order estimators.
$w_{\mathrm{I}}$ and $w_{\mathrm{II}}$ denote weighting schemes (I) and (II).
The oracle PI estimators use the true marginal probabilities $\alpha_0$.
\item SAA and NW denote the sample-average approximation and
Nadaraya--Watson estimators. They are unpenalized; consequently,
each reported value spans the $\lambda_{\mathrm{cv}}$ and
$\lambda_{\mathrm{theory}}$ columns.
\end{tablenotes}
\end{threeparttable}
\end{table}

Table~\ref{tab:M15-results} reports results from 100 replications: Panel A for the no-covariate case and Panels B and C for profiles $x_{\mathrm{L}}$ and $x_{\mathrm{S}}$ repectively. $w_{\mathrm{I}}$ and $w_{\mathrm{II}}$ identify weighting schemes (I) and (II). The first-order estimators yield lower RMSEs and more accurate probability estimates than the plug-in estimator. The oracle plug-in estimator performs best but is generally infeasible because $\alpha_0$ is unknown. By contrast, the SAA and NW baselines perform substantially worse than the plug-in, first-order, and oracle plug-in estimators.

\textbf{Computation of the first-order estimator for $M=15$.} In the no-covariate case, $\mathcal{A}(\hat{\alpha})$ contains $2^{15}=32,768$ extreme points and exact enumeration is feasible. In the with-covariate case, the adversarial set is a hyper-rectangle with $2^{M(1+d)}=2^{60}$ extreme points, making exhaustive enumeration infeasible. We therefore employ a two-step inexact cutting-set procedure following the heuristics of \citet{mutapcic2009cutting}: repeatedly solve a restricted robust problem, search for a more adverse realization, and add it to the restricted problem.  
\emph{Step 1 (initialize)}: We initialize an active set with two opposite corners of the hyper-rectangle: the all-upper corner, where every $\alpha_j$ and $\beta_{jk}$ is set to the upper endpoint of its confidence interval, and the all-lower corner, where every coefficient is set to its lower endpoint. We then add $100$ randomly sampled extreme points by independently selecting the lower or upper endpoint of each coefficient with equal probability. \emph{Step 2 (expand extreme-point set)}: We iteratively expand the collection of extreme points. First, we solve the estimation problem with the inner minimization restricted to the current collection, obtaining $\hat{r}$. Holding $\hat{r}$ fixed, we then search for a more adverse extreme point. Within each iteration, we cycle through the coefficients $\alpha_j$ and $\beta_{jk}$. For each coefficient, we hold all others fixed, compare the inner objective at its lower and upper endpoints, and retain the endpoint yielding the smaller value. We repeat until no single endpoint change decreases the objective and retain the best terminal point. If this point lowers the inner objective at $\hat{r}$ by more than the prescribed tolerance relative to the current collection, we add it and re-solve the estimation problem; otherwise, the procedure terminates. Each coordinate update requires one pass over the observations, so a complete iteration over all $M(1+d)$ coefficients costs $\mathrm{O}(NM(1+d))$, which is negligible relative to the outer estimation problem.

\section{Discussion}\label{sec:discussion}
We develop a method for estimating binary configuration probabilities. Existing approaches often ignore independence structures in the joint distribution and can therefore be statistically inefficient in high dimensions. Our approach is adaptive to the independence structure by leveraging sparsity in the generalized correlation coefficients \citep{bahadur1959representation}, yielding improved finite-sample performance both theoretically and numerically. We close the paper by a discussion on possible extensions. First, for the covariate-case, we establish a pointwise rate of convergence result for $r_0(x)$. So a natural extension is to establish uniform convergence results over all $x\in\mathcal{X}$. Second, it would also be interesting to incorporate prior information about independence directly into the generalized correlation coefficients and to develop localized sparsity constraints, such as sparsity levels $s_x$ that vary across covariate strata. We expect our method to remain valid provided that the covariate dimension is low. It is also interesting to incorporate high-dimensional covariates, which would require new modeling strategies and estimators. Further, our approach offers a fundamentally different way to model independence from the graphical conditional independence. \citet{yuan2021community} shows in simulations that their community detection method based on a truncated Bahadur representation outperforms approaches that assume conditional independence. This suggests that our method may also be useful in similar settings, such as community detection, to capture dependence among edges. Lastly, recall from Remark~\ref{remark:non-negativity} that since $(\hat \alpha, r_0)$ may not lie in $\mathbb{K}$, we do not impose this constraint in the optimization. For the plug-in estimator, enforcing $(\hat{\alpha},r)\in\mathbb{K}$ requires adding $2^M$ linear constraints. For the first-order estimator, this requirement would be even more restrictive, since it would require $(\alpha,r)\in\mathbb{K}$ for every $\alpha\in\mathcal{A}(\hat{\alpha})$, which can be computationally prohibitive. Moreover, this introduces a tradeoff: the feasible set may become overly restrictive, and since it is possible that $(\hat \alpha, r_0)\notin\mathbb{K}$, imposing such constraints can lead to large estimation bias. The discussion in Remark~\ref{remark:non-negativity} indicates that our result is sufficient for a broad range of practical applications. For the causal inference application, $\min_{i\in[N]}1+W(\hat{\alpha}(X_i),t)'\hat{r}^{\mathrm{FO}}(X_i,X_i)>0$ with high probability, so the AIPW estimator is generally not impacted either. In the rare cases where positivity of the probability estimates fails, a simple truncation step can be applied. There could exist alternative projection-based approaches that could enforce this constraint, but analyzing such methods would require a different type of analysis. Since the current approach is already sufficiently general for many practically relevant settings, so we leave such extensions to future work.

\begin{APPENDICES}
\OneAndAHalfSpacedXI
\section*{Appendix}\label{sec:appendices}
We now state Assumption~\ref{ass:smoothness-assumptions} which is implied by Assumption~\ref{ass:smoothness-assumptions:sufficient}, with the choices $K_{\alpha,1}=2L_N$, $\bar K_w=L_N$, $K_{\alpha,2}=2L_N^2$, $\bar K_{x,2}=C_x$ below:
\begin{assumption}\label{ass:smoothness-assumptions}
With probability $1-\delta_\alpha$, conditions (1)--(6)
below hold simultaneously for every $i\in[N]$ and for all vectors
$u,v$ of compatible dimensions: \\(1) $\alpha_0\in\mathcal{A}(\hat{\alpha})$, $\quad$ (2) ${\displaystyle\max_{j\in [p], \tilde \alpha \in \mathcal{A}(\hat\alpha)}}|v'\nabla_\alpha W_j(\tilde \alpha,Y_i)|  +|v'\nabla_\alpha W(\tilde \alpha,Y_i)'r_0| \leq M^{1/2}K_{\alpha,1}\|v\|$, \\
(3) $|u'\nabla_{\alpha}W(\alpha_0,Y_i)'r_0|\leq\bar{K}_w\|u\|_1$, $|u'\nabla_{\alpha}W(\alpha_0,Y_i)'v|\leq \bar{K}_w\norm{u}_1\norm{v}_{1}$, \\
(4) $\displaystyle\max_{j \in [p], \tilde \alpha \in \mathcal{A}(\hat\alpha)}|u'\nabla_\alpha^2 W(\tilde\alpha,Y_i)'r_0v|+|u'\nabla_\alpha^2 W_j(\tilde\alpha,Y_i)v|\leq K_{\alpha,2}\|u\|\|v\|$, \\
(5) ${\displaystyle \max_{\tilde\alpha \in \mathcal{A}(\hat\alpha)}}\|\tilde \alpha-\alpha_0\|\frac{M^{1/2}K_{\alpha,1}}{\xi}<\frac{1}{4}$, ${\displaystyle\max_{\tilde\alpha \in \mathcal{A}(\hat\alpha)}}\|\tilde\alpha-\alpha_0\|^2 \frac{K_{\alpha,2}}{\xi}\leq\frac{1}{16}$, \\
(6) $|W(\alpha_0,Y_i)'v|\leq\bar{K}_w\|v\|_1$, $u'(\nabla_{\alpha}W(\alpha_0,Y_i)'r_0)(\nabla_{\alpha}W(\alpha_0,Y_i)'r_0)'u\leq K_{\alpha,2}\|u\|^2$.

\noindent With probability $1-\delta_\alpha$, for the with-covariate case, over all $i\in[N]$ where $\|X_i-x\|\leq h$, we have:
(i) {$\alpha_{*x}(X):=\alpha_0(x)+\nabla_x\alpha_0(x)(X-x) \in \mathcal{A}_x(\hat\alpha)$},\\
(ii) ${\displaystyle\max_{j\in [p], \tilde \alpha \in \mathcal{A}_x(\hat\alpha)}}|v'\nabla_\alpha W_j(\tilde \alpha(X_i),Y_i)|  +|v'\nabla_\alpha W(\tilde \alpha(X_i),Y_i)'r_0(X_i)| \leq M^{1/2}K_{\alpha,1}\|v\|$, \\
(iii)$|u'\nabla_{\alpha}W(\alpha_0(X_i),Y_i)'r_0(X_i)|\leq\bar{K}_w\|u\|_1$, $|u'\nabla_{\alpha}W(\alpha_0(X_i),Y_i)'v|\leq \bar{K}_w\norm{u}_1\norm{v}_{1}$,\\
(iv)$\displaystyle\max_{j \in [p], \tilde \alpha \in \mathcal{A}_x(\hat\alpha)}|u'\nabla_\alpha^2 W(\tilde\alpha(X_i),Y_i)'r_0(X_i)v|+|u'\nabla_\alpha^2 W_j(\tilde\alpha(X_i),Y_i)v|\leq K_{\alpha,2}\|u\|\|v\|$,\\
(v)${\displaystyle \max_{\tilde\alpha \in \mathcal{A}_x(\hat\alpha),\|\tilde x-x\|\leq h}}\|\tilde \alpha(\tilde x)-\alpha_0(\tilde x)\|\frac{M^{1/2}K_{\alpha,1}}{\xi}<1/4$,  ${\displaystyle \max_{\tilde\alpha \in \mathcal{A}_x(\hat\alpha),\|\tilde x-x\|\leq h}}\|\tilde \alpha(\tilde x)-\alpha_0(\tilde{x})\|^2 \frac{K_{\alpha,2}}{\xi}\leq1/16$,\\
(vi) $|W(\alpha_0(X_i),Y_i)'v|\leq\bar{K}_w\|v\|_1$, $u'[\nabla_{\alpha}W(\alpha_0(X_i),Y_i)'r_0(X_i)][\nabla_{\alpha}W(\alpha_0(X_i),Y_i)'r_0(X_i)]'u\leq K_{\alpha,2}\|u\|^2$, ${\displaystyle\max_{j\in [p], k\in[M]}}|v'\nabla_{x}^2 r_{0j}(x)v|+ |v'\nabla^2\alpha_{0k}(x)v|\leq \bar{K}_{x,2}\|v\|^2$.
\end{assumption}
Without loss of generality, we focus on  discussing conditions (ii)-(vi) for the with-covariate case. Given any $i\in[N],j\in[p],k\in[M]$, $\nabla_{\alpha_k}W_j(\alpha(X_i),Y_i)=\mathbf{1}\{k\in\ell(j)\}W_j(\alpha(X_i),Y_i)\frac{\partial_{\alpha_k}z_k(Y_i,\alpha(X_i))}{z_k(Y_i,\alpha(X_i))}$, where $\frac{\partial_{\alpha_k}z_k(Y_i,\alpha(X_i))}{z_k(Y_i,\alpha(X_i))}=\frac{1-2Y_{ik}}{2\alpha_{k}(X_i)(1-\alpha_{k}(X_i))}$, $\nabla_{\alpha}^2W_j(\tilde{\alpha}(X_i),Y_i)=W_j(\tilde{\alpha}(X_i),Y_i)[g_{ij}g_{ij}'+D_{ij}]$, where $g_{ij}\in\mathbb{R}^M$ and $g_{ijk}=\frac{\mathbf{1}\{k\in\ell(j)\}(1-2Y_{ik})}{2\tilde{\alpha}_k(X_i)(1-\tilde{\alpha}_k(X_i))}$, $D_{ij}\in\mathbb{R}^{M\times M}$ is a diagonal matrix with $k$-th diagonal entry $-\frac{\mathbf{1}\{k\in\ell(j)\}(1-2Y_{ik})(1-2\tilde{\alpha}_k(X_i))}{2\tilde{\alpha}_k(X_i)^2(1-\tilde{\alpha}_k(X_i))^2}$. Condition (vi) controls the curvature of $\alpha_0(\cdot)$ and $r_0(\cdot)$ with respect to the covariates.
Particularly, when $\alpha_{0k}(X_i)$ are bounded away from both $0$ and $1$, since $\tilde{\alpha}\in\mathcal{A}_x(\hat{\alpha})$ are close to $\alpha_0(\cdot)$ with high probability by the construction of $\mathcal{A}_x(\hat{\alpha})$, $\tilde{\alpha}\in\mathcal{A}_x(\hat{\alpha})$ are also bounded away from both $0$ and $1$ with high probability, then condition (vi) implies that conditions (ii) -- (iv) hold. The theoretical analysis later implies ${\displaystyle \max_{\substack{\|\tilde x-x\|\leq h\\
\tilde\alpha \in \mathcal{A}_x(\hat\alpha)}}}M^{1/2}\|\tilde \alpha(\tilde x)-\alpha_0(\tilde x)\|$ goes to zero as $N\rightarrow\infty$. So (v) holds when $N$ is sufficiently large. 

In the proofs of the theoretical results for the no-covariate case, whenever a result is stated under Assumption~\ref{ass:smoothness-assumptions:sufficient}, we directly invoke the corresponding technical conditions from Assumption~\ref{ass:smoothness-assumptions} (1) -- (6). The former is stated in the main text to streamline the exposition, whereas the latter provides the detailed conditions used in the proofs.

\medskip
\proof{Proof of Theorem~\ref{thm:main:FO}}
Firstly, define $W^{\mathrm{FO}}(\alpha,y):=W(\hat \alpha,y)+\nabla_\alpha W(\hat\alpha,y)(\alpha-\hat\alpha)$, $Q^{\mathrm{FO}}(\alpha,r):= \frac{1}{N}\sum_{i=1}^N \log(1+ W^{\mathrm{FO}}(\alpha,Y_i)'r)$, $\mathcal{E}_0:= \{ \alpha_0 \in \mathcal{A}(\hat\alpha)\}$, $\mathcal{E}_1:= \{ \min_{\alpha \in \mathcal{A}(\hat\alpha)} Q^{\mathrm{FO}}(\alpha,r_0)-Q^{\mathrm{FO}}(\alpha_0,r_0) \geq - \widetilde{\mathcal{R}}(\hat\alpha) \}$, where $\widetilde{\mathcal{R}}(\hat\alpha)= C\bar{K}_2\sqrt{\frac{\log (M/\delta)}{N}}\max_{\alpha \in \mathcal{A}(\hat\alpha)}\|\alpha-\alpha_0\|_1+\frac{CK_{\alpha,2}}{\xi^2} \max_{\alpha \in \mathcal{A}(\hat\alpha)}\|\alpha - \alpha_0\|^2$. $\mathcal{E}_2:=\{\frac{\lambda}{c} \geq \left\|{\rm diag}^{-1}(w) \nabla_r Q^{\mathrm{FO}}(\alpha_0,r_0)\right\|_\infty \}$, $\mathcal{E}_3:=\{\min_{i\in[N]} 1+W^{\mathrm{FO}}(\alpha_0,Y_i)'r_0\geq\xi/2\}$. 
Assumption \ref{ass:smoothness-assumptions} implies that $\mathcal{E}_0$ holds with probability at least $1-\delta_\alpha$. $\mathcal{E}_1$ holds with probability $1-\epsilon-\delta_{\alpha}-\delta(1+\mathrm{o}(1))$ by Lemma \ref{lemma:MinMax:Claim1}. $\mathcal{E}_2$  holds with probability $1-\delta-\delta_{\alpha}$ by the conditions of Theorem~\ref{thm:main:FO} and Lemma \ref{lemma:Score:nocovariate}. Note that with probability $1-\epsilon-C\delta_{\alpha}$, 
$\min_{i\in[N]}1+W^{\mathrm{FO}}(\alpha_0,Y_i)'r_0\geq\min_{i\in[N]}1+W(\hat{\alpha},Y_i)'r_0+(\alpha_0-\hat{\alpha})'\nabla_{\alpha}W(\hat{\alpha},Y_i)'r_0\geq_{(1)}\frac{3}{4}\xi-M^{1/2}K_{\alpha,1}\|\alpha_0-\hat{\alpha}\|\geq_{(2)}\frac{1}{2}\xi$, where inequality (1) follows from Holder's inequality and the fact that the proof of Proposition~\ref{prop:plug-in:no-covariate} shows that with probability $1-\epsilon-C'\delta_{\alpha}$, $1+W(\hat{\alpha},Y_i)'r_0\geq\frac{3}{4}\xi$, and inequality (2) follows from  Assumption~\ref{ass:smoothness-assumptions}. Let $\hat{r}=\hat{r}^{\mathrm{FO}}$. It is then sufficient to prove that the result holds on $\mathcal{E}_0\cap\mathcal{E}_1\cap\mathcal{E}_2\cap\mathcal{E}_3$. Particularly on the intersection of all these events we have 
\begin{equation}\label{eq:main:steps}
\begin{array}{rl}
&\displaystyle Q^{\mathrm{FO}}(\alpha_0,\hat r) - \lambda \|\hat r\|_{w,1}\displaystyle \geq_{(1)}\min_{\alpha \in \mathcal{A}(\hat\alpha)} Q^{\mathrm{FO}}(\alpha,\hat r) - \lambda \|\hat r\|_{w,1} \geq_{(2)} \min_{\alpha \in \mathcal{A}(\hat\alpha)} Q^{\mathrm{FO}}(\alpha,r_0) - \lambda \|r_0\|_{w,1}\\
&\displaystyle = \min_{\alpha \in \mathcal{A}(\hat\alpha)} \{Q^{\mathrm{FO}}(\alpha,r_0)-Q^{\mathrm{FO}}(\alpha_0,r_0)\} + Q^{\mathrm{FO}}(\alpha_0,r_0)- \lambda \|r_0\|_{w,1}\\
&\displaystyle \geq_{(3)} - \widetilde{\mathcal{R}}(\hat\alpha) + Q^{\mathrm{FO}}(\alpha_0,r_0)- \lambda \|r_0\|_{w,1},
\end{array}
\end{equation}
where (1) holds by $\mathcal{E}_0$, (2) holds by the optimality of $\hat r$ and (3) holds by $\mathcal{E}_1$. Following similar proof steps as in Lemma~\ref{lemma:fast-rate:estimated-alpha}, $$Q^{\mathrm{FO}}(\alpha_0,\hat r)\leq Q^{\mathrm{FO}}(\alpha_0, r_0) + \nabla_r Q^{\mathrm{FO}}(\alpha_0, r_0)'(\hat r - r_0)  - \frac{c_0}{N} \sum_{i=1}^N\left\{\frac{W^{\mathrm{FO}}(\alpha_0,Y_i)^\prime(\hat{r}-r_0)}{1+W^{\mathrm{FO}}(\alpha_0,Y_i)^\prime r_0}\right\}^2$$ 
for some universal constant $c_0>0$. Combining this and (\ref{eq:main:steps}), on event $\mathcal{E}_0\cap\mathcal{E}_1\cap\mathcal{E}_2\cap\mathcal{E}_3$, using (weighted) H\"older's inequality we have 
\begin{equation}\label{eq:thm1:main:main}
\begin{array}{rl}
&\displaystyle\quad\frac{c_0}{N} \sum_{i=1}^N\left\{\frac{W^{\mathrm{FO}}(\alpha_0,Y_i)^\prime(\hat{r}-r_0)}{1+W^{\mathrm{FO}}(\alpha_0,Y_i)^\prime r_0}\right\}^2\leq   \lambda\|r_0\|_{w,1} -  \lambda\|\hat r\|_{w,1}   + \nabla_r Q^{\mathrm{FO}}(\alpha_0,r_0)'(\hat r - r_0) + \widetilde{\mathcal{R}}(\hat\alpha)\\ 
&\displaystyle \leq \lambda\|r_0\|_{w,1} -  \lambda\|\hat r\|_{w,1}+ \widetilde{\mathcal{R}}(\hat\alpha) + \left\|{\rm diag}^{-1}(w) \nabla_r Q^{\mathrm{FO}}(\alpha_0,r_0)\right\|_\infty\|\hat r - r_0\|_{w,1} \\
&\displaystyle \leq    \lambda\|r_0\|_{w,1} -  \lambda\|\hat r\|_{w,1} + \frac{\lambda}{c}\|\hat r - r_0\|_{w,1}   + \widetilde{\mathcal{R}}(\hat\alpha)\leq \left(1+\frac{1}{c}\right)\lambda \|\hat r_T-r_0\|_{w,1} \!-\! \!\lambda \left(1-\frac{1}{c}\right)\|\hat r_{T^c}\|_{w,1} + \widetilde{\mathcal{R}}(\hat\alpha).
\end{array}
\end{equation}
Let $\mathcal{W}:={\rm diag}(w)$. H\"older's inequality implies $v'\hat Uv \geq v'Uv - \|v\|_{w,1}^2\|\mathcal{W}^{-1}(U-\hat U)\mathcal{W}^{-1}\|_{\infty,\infty}$.
Define $\displaystyle U:=\frac{c_0}{N}\sum_{i=1}^N\frac{W(\alpha_0,Y_i)W(\alpha_0,Y_i)'}{(1+W(\alpha_0,Y_i)^\prime r_0)^2}$, $\displaystyle\hat U :=\frac{c_0}{N}\sum_{i=1}^N\frac{W^{\mathrm{FO}}(\alpha_0,Y_i)W^{\mathrm{FO}}(\alpha_0,Y_i)'}{(1+W^{\mathrm{FO}}(\alpha_0,Y_i)^\prime r_0)^2}$, we have 
\begin{equation}\label{eq:DesignMatrixConnection}
\frac{c_0}{N}\!\!\sum_{i=1}^N\!\!\left\{\!\frac{W^{\mathrm{FO}}(\alpha_0,Y_i)^\prime(\hat{r}-r_0)}{1+W^{\mathrm{FO}}(\alpha_0,Y_i)^\prime r_0}\!\right\}^2\!\!\!\!-\!\frac{c_0}{N}\!\! \sum_{i=1}^N\!\!\left\{\!\frac{W(\alpha_0,Y_i)^\prime(\hat{r}-r_0)}{1+W(\alpha_0,Y_i)^\prime r_0}\!\right\}^2\!\!\!\geq\displaystyle -\|\hat r - r_0\|_{w,1}^2 \|\mathcal{W}^{-1}(U-\hat U)\mathcal{W}^{-1}\|_{\infty,\infty}
\end{equation}
Recall that $T$ is the support of $r_0$. For the rest of the proof, we show that the result holds under two exclusive cases: $\|\hat r_{T^c}\|_{w,1} \geq 2\frac{c+1}{c-1}\|\hat r_T - r_0\|_{w,1}$ and $\|\hat r_{T^c}\|_{w,1}<2\frac{c+1}{c-1}\|\hat r_T - r_0\|_{w,1}$. First assume that $\|\hat r_{T^c}\|_{w,1} \geq 2\frac{c+1}{c-1}\|\hat r_T - r_0\|_{w,1}$.
From (\ref{eq:thm1:main:main}) we have 
\begin{equation}\label{eq:nocovariate:UnRestrictedSet} 
\frac{c_0}{N} \sum_{i=1}^N\left\{\frac{W^{\mathrm{FO}}(\alpha_0,Y_i)^\prime(\hat{r}-r_0)}{1+W^{\mathrm{FO}}(\alpha_0,Y_i)^\prime r_0}\right\}^2\leq \widetilde{\mathcal{R}}(\hat\alpha)-\frac{\lambda}{2}\left(1-\frac{1}{c}\right)\|\hat r_{T^c}\|_{w,1}.
\end{equation}
Thus from \eqref{eq:DesignMatrixConnection} and \eqref{eq:nocovariate:UnRestrictedSet} we have 
$$\begin{array}{rl}
&\displaystyle\quad\frac{c_0}{N} \sum_{i=1}^N\left\{\frac{W^{\mathrm{FO}}(\alpha_0,Y_i)^\prime(\hat{r}-r_0)}{1+W^{\mathrm{FO}}(\alpha_0,Y_i)^\prime r_0}\right\}^2  -   \frac{c_0}{N}\sum_{i=1}^N\left\{\frac{W(\alpha_0,Y_i)^\prime(\hat{r}-r_0)}{1+W(\alpha_0,Y_i)^\prime r_0}\right\}^2 \\
&\displaystyle \geq  -\|\hat r - r_0\|_{w,1}^2 \|\mathcal{W}^{-1}(U-\hat U)\mathcal{W}^{-1}\|_{\infty,\infty} =  -(\|\hat r_T - r_0\|_{w,1} + \| \hat r_{T^c}\|_{w,1})^2 \|\mathcal{W}^{-1}(U-\hat U)\mathcal{W}^{-1}\|_{\infty,\infty}\\
&\displaystyle \geq -\left(\frac{1}{2\frac{c+1}{c-1}} + 1\right)^2\| \hat r_{T^c}\|_{w,1}^2 \|\mathcal{W}^{-1}(U-\hat U)\mathcal{W}^{-1}\|_{\infty,\infty}\geq_{(a)}  - \frac{9}{4}\ \frac{4\widetilde{\mathcal{R}}^2(\hat\alpha)}{\mbox{$(1-\frac{1}{c})^2$}\lambda^2} \|\mathcal{W}^{-1}(U-\hat U)\mathcal{W}^{-1}\|_{\infty,\infty},
\end{array}$$
where (a) follows since \eqref{eq:nocovariate:UnRestrictedSet} implies that $\widetilde{\mathcal{R}}(\hat\alpha)\geq\frac{\lambda}{2}\left(1-\frac{1}{c}\right)\|\hat r_{T^c}\|_{w,1}$ when $\|\hat r_{T^c}\|_{w,1} \geq 2\frac{c+1}{c-1}\|\hat r_T - r_0\|_{w,1}$. Rearranging the terms above and using (\ref{eq:nocovariate:UnRestrictedSet}) again we have $\displaystyle\frac{c_0}{N}\sum_{i=1}^N\left\{\frac{W(\alpha_0,Y_i)^\prime(\hat{r}-r_0)}{1+W(\alpha_0,Y_i)^\prime r_0}\right\}^2 \leq  \frac{9\widetilde{\mathcal{R}}^2(\hat\alpha)\|\mathcal{W}^{-1}(\hat U - U)\mathcal{W}^{-1} \|_{\infty,\infty}}{\mbox{$(1-\frac{1}{c})^2$}\lambda^2} + \widetilde{\mathcal{R}}(\hat\alpha)$. Next suppose that  $\|\hat r_{T^c}\|_{w,1} < 2\frac{c+1}{c-1}\|\hat r_T - r_0\|_{w,1}$. Then 
$$\begin{array}{rl}
&\displaystyle\frac{c_0}{N}\!\sum_{i=1}^N\!\left\{\!\frac{W(\alpha_0,Y_i)^\prime(\hat{r}-r_0)}{1+W(\alpha_0,Y_i)^\prime r_0}\!\right\}^2\!\!\!\!\leq_{(1)}\!\!\frac{\lambda(1+c)\|\hat r_{T} - r_0\|_{w,1}}{c}\!+\!\widetilde{\mathcal{R}}(\hat\alpha)\!+\!\|\hat r - r_0\|_{w,1}^2 \|\mathcal{W}^{-1}(\hat U - U)\mathcal{W}^{-1} \|_{\infty,\infty}\\
&\displaystyle \leq_{(2)}\frac{\lambda(1+c)\sqrt{s}}{\sqrt{c_0}c\kappa_{\bar c}}\!\!\sqrt{\frac{c_0}{N} \!\!\sum_{i=1}^N\left\{\frac{W(\alpha_0,Y_i)^\prime(\hat{r}-r_0)}{1+W(\alpha_0,Y_i)^\prime r_0}\right\}^2}  \\
&\displaystyle\quad+\frac{s\left(2\frac{c+1}{c-1}+1\right)^2\|\mathcal{W}^{-1}(\hat U - U)\mathcal{W}^{-1} \|_{\infty,\infty} }{c_0\kappa_{\bar c}^2}\frac{c_0}{N} \sum_{i=1}^N\left\{\frac{W(\alpha_0,Y_i)^\prime(\hat{r}-r_0)}{1+W(\alpha_0,Y_i)^\prime r_0}\right\}^2  + \widetilde{\mathcal{R}}(\hat\alpha),
\end{array}$$
where (1) follows by (\ref{eq:thm1:main:main}) and (\ref{eq:DesignMatrixConnection}),  and (2) follows from the definition of the restricted eigenvalue, $\|\hat{r}_{T}-r_0\|_{w,1}\leq\sqrt{s}\|\hat{r}_{T}-r_0\|_{w,2}$ and $\|\hat{r}-r_0\|_{w,1}^2=(\|\hat{r}_{T^c}\|_{w,1}+\|\hat{r}_T-r_0\|_{w,1})^2\leq\left(2\frac{c+1}{c-1}+1\right)^2\|\hat{r}_T-r_0\|_{w,1}^2$. Note that $\frac{s\left(2\frac{c+1}{c-1}+1\right)^2\|\mathcal{W}^{-1}(\hat U - U)\mathcal{W}^{-1} \|_{\infty,\infty} }{c_0\kappa_{\bar{c}}^2} \leq_{(i)} \frac{Cs\|\hat\alpha-\alpha_0\|^2 K_{\alpha,2}}{\xi\kappa_{\bar c}^2 \min_{j\in[p]} w_j^2}\leq_{(ii)}C\delta_N$, where $C\geq\frac{\left(2\frac{c+1}{c-1}+1\right)^2}{c_0}\max_{k,l\in[p]}\frac{1}{N}\sum_{i=1}^N5\left(1+\frac{|W_k(\alpha_0,Y_i)|}{1+W(\alpha_0,Y_i)^\prime r_0}\right)\left(1+\frac{|W_l(\alpha_0,Y_i)|}{1+W(\alpha_0,Y_i)^\prime r_0}\right)$.
Here the inequality (i) holds by Lemma \ref{lemma:MatrixConnection}, and the inequality (ii) holds by the condition in Theorem \ref{thm:main:FO}. Since $\delta_N\rightarrow0$, so for $N$ sufficiently large, $\frac{s \left(2\frac{c+1}{c-1}+1\right)^2\|\mathcal{W}^{-1}(\hat U - U)\mathcal{W}^{-1} \|_{\infty,\infty} }{c_0\kappa_{\bar{c}}^2} \leq \frac{1}{2}$.
Furthermore, note that 
$\frac{x^2}{2} \leq \lambda(1+\frac{1}{c})\frac{\sqrt{s}}{\sqrt{c_0}\kappa_{\bar{c}}} x + \widetilde{\mathcal{R}}(\hat \alpha) \Rightarrow x \leq 2\lambda(1+\frac{1}{c})\frac{\sqrt{s}}{\sqrt{c_0}\kappa_{\bar{c}}} + \sqrt{2}\widetilde{\mathcal{R}}^{1/2}(\hat \alpha)$ for $x=\sqrt{\frac{c_0}{N}\sum_{i=1}^N\left\{\frac{W(\alpha_0,Y_i)'(\hat{r}-r_0)}{1+W(\alpha_0,Y_i)'r_0}\right\}^2}$.
So the result holds. 
\proofend

\medskip
\proof{Proof of Corollary~\ref{cor:FO}}
According to \eqref{def:A:no-covariate}, $\mathcal{A}(\hat\alpha) =\prod_{j=1}^M [\hat\alpha_j - cv_{\delta_\alpha} \hat s_j, \hat\alpha_j + cv_{\delta_\alpha} \hat s_j]$, where $\hat s_j= \sqrt{\hat\alpha_j(1-\hat\alpha_j)}$ and $cv_{\delta_\alpha} := (1-\delta_\alpha)\mbox{-quantile of} \max_{j\in[M]} \left| \frac{1}{N}\sum_{i=1}^N \xi_i \frac{(Y_{ij} - \hat \alpha_j)}{\sqrt{\hat\alpha_j(1-\hat\alpha_j)}}\right|$ conditioning on $\{Y_i\}_{i\in[N]}$. 
Denote $\hat{s}=(\hat{s}_j)_{j\in[M]}\in\mathbb{R}^M$. Then for any $t>0$, 
\begin{equation}\label{eq:key:cor:no-cov}
\begin{array}{rl}
\displaystyle\mathbb{P}\bigg(\displaystyle\max_{\alpha\in\mathcal{A}(\hat{\alpha})}\norm{\alpha -\alpha_0}_2\geq t\bigg)&\displaystyle\leq\mathbb{P}\bigg(\norm{\hat{\alpha}-\alpha_0}_2+cv_{\delta_{\alpha}}\sqrt{\sum_{j\in[M]}\hat{s}_j^2}\geq t\bigg)\\
&\displaystyle\leq\mathbb{P}(\norm{\hat{\alpha}-\alpha_0}_2\geq t/2)+\mathbb{P}(cv_{\delta_{\alpha}}\sqrt{M/4}\geq t/2),
\end{array}
\end{equation}
where the second inequality above follows by triangular inequality and the fact that $\hat{s}_j^2\leq\frac{1}{4}$ for all $j\in[M]$. To bound $\mathbb{P}(\norm{\hat{\alpha}-\alpha_0}_2\geq t/2)$ of \eqref{eq:key:cor:no-cov}, note that by Proposition \ref{prop:concentration:no covariate}, with probability $\geq1-\delta$, $\norm{\hat{\alpha}-\alpha_0}_2\leq c_1\sqrt{\frac{M\log(1/\delta)}{N}}$, where $c_1$ is some absolute constant. To bound $\mathbb{P}(cv_{\delta_{\alpha}}\sqrt{M/4}\geq t/2)$ of \eqref{eq:key:cor:no-cov}, note that $\xi_i\overset{i.i.d.}{\sim}\mathcal{N}(0,1)$, therefore conditional on $\{Y_i\}_{i\in[N]}$, 
$\frac{1}{N}\sum_{i=1}^N\xi_i\frac{(Y_{ij}-\hat{\alpha}_j)}{\sqrt{\hat{\alpha}_j(1-\hat{\alpha}_j)}}\sim\mathcal{N}\left(0,\frac{1}{N^2}\sum_{i=1}^N\frac{(Y_{ij}-\hat{\alpha}_j)^2}{\hat{\alpha}_j(1-\hat{\alpha}_j)}\right)$ holds for all $j\in[M]$, so $\mathbb{E}\left[\max_{j\in[M]}\left|\frac{1}{N}\sum_{i=1}^N\xi_i\frac{(Y_{ij}-\hat{\alpha}_j)}{\sqrt{\hat{\alpha}_j(1-\hat{\alpha}_j)}}\right|\right]\leq\sqrt{\frac{2\sigma^2\log(2M)}{N}}$, where $\sigma^2=\max_{j\in[M]}\frac{1}{N}\sum_{i=1}^N\frac{(Y_{ij}-\hat{\alpha}_j)^2}{\hat{\alpha}_j(1-\hat{\alpha}_j)}$.
By Lemma~\ref{lemma:borell}, for some absolute constant $C_0$, with probability $1-\delta$ we have 
$$\begin{array}{rl}
\displaystyle\max_{j\in[M]}\left|\frac{1}{N}\sum_{i=1}^N\xi_i\frac{(Y_{ij}-\hat{\alpha}_j)}{\sqrt{\hat{\alpha}_j(1-\hat{\alpha}_j)}}\right|&\displaystyle\leq\mathbb{E}\left[\max_{j\in[M]}\left|\frac{1}{N}\sum_{i=1}^N\xi_i\frac{(Y_{ij}-\hat{\alpha}_j)}{\sqrt{\hat{\alpha}_j(1-\hat{\alpha}_j)}}\right|\right]\!\!+\sigma\sqrt{\frac{2\log\left(2/\delta\right)}{N}}\\
&\displaystyle\leq C_0\sqrt{\frac{\log(M/\delta)}{N}}\sqrt{\max_{j\in[M]}\frac{1}{N}\sum_{i=1}^N\frac{(Y_{ij}-\hat{\alpha}_j)^2}{\hat{\alpha}_j(1-\hat{\alpha}_j)}}.
\end{array}$$
Let $\delta\leq\delta_{\alpha}$ and set $t=C\sqrt{\frac{M\log(M/\delta)}{N}}$ for some absolute constant $C$ in (\ref{eq:key:cor:no-cov}), then with probability $1-2\delta$, $\max_{\alpha\in\mathcal{A}(\hat{\alpha})}\norm{\alpha -\alpha_0}_2\leq C\sqrt{\frac{M\log(M/\delta)}{N}}$. Additionally, for any $t>0$, by triangular inequality, 
\begin{equation}\label{ineq:concentration:l1:no-cov}
\begin{array}{rl}
\displaystyle\mathbb{P}\big(\max_{\alpha\in\mathcal{A}(\hat{\alpha})}\|\alpha-\alpha_0\|_1\geq t\big)&\displaystyle\leq\mathbb{P}(\|\hat{\alpha}-\alpha_0\|_1\geq t/2)+\mathbb{P}(cv_{\delta_{\alpha}}\|\hat{s}\|_1\geq t/2)\\
&\displaystyle\leq_{(a)}\mathbb{P}(\|\hat{\alpha}-\alpha_0\|_1\geq t/2)+\mathbb{P}(cv_{\delta_{\alpha}}(M/2)\geq t/2),
\end{array}
\end{equation}
where (a) follows because $\|\hat{s}\|_1\leq M/2$.
By Proposition \ref{prop:concentration:no covariate}, with probability $1-\delta$, we have $\norm{\hat\alpha-\alpha_0}_1\leq 
c_1\frac{M\sqrt{\log(1/\delta)}}{\sqrt{N}}$ for some constant $c_1>0$. Let $\delta\leq\delta_{\alpha}$ and set $t=CM\sqrt{\frac{\log (M/\delta)}{N}}$ for some absolute constant $C$, following (\ref{ineq:concentration:l1:no-cov}), with probability $1-2\delta$, we have $\max_{\alpha\in\mathcal{A}(\hat{\alpha})}\|\alpha-\alpha_0\|_1\leq CM\sqrt{\frac{\log(M/\delta)}{N}}$. Additionally, we choose $\lambda$ satisfying the condition of Theorem \ref{thm:main:FO}, so that for some absolute constant $\bar{C}$, $\lambda \leq \bar{C}N^{-1/2}\sqrt{\log(p/\delta)}$, then under the condition $\kappa_{\bar{c}}^{-1}+\bar K_2 + K_{\alpha,1} + K_{\alpha,2}/\xi^2 \leq C'$, we have $\widetilde{\mathcal{R}}(\hat\alpha) \leq \bar{C}'M\sqrt{\frac{\log(M/\delta)}{N}} \sqrt{\frac{\log(M/\delta)}{N}}+ \bar{C}'\frac{M\log(M/\delta)}{N}$ for some constant $\bar{C}'$. By Theorem \ref{thm:main:FO}, for an absolute constant $C$, 
$$\sqrt{\frac{1}{N}\sum_{i=1}^N \frac{\{W(\alpha_0,Y_i)'(\hat r^{\mathrm{FO}} - r_0 )\}^2}{(1+W(\alpha_0,Y_i)'r_0)^2}}\leq C\sqrt{\frac{s\log (p/\delta)}{N}}+ C\sqrt{\frac{M\log(M/\delta)}{N}}$$ 
holds with probability $1-C\delta-C\delta_\alpha-C\epsilon$, hence the result follows since $p=2^M-M-1$. \proofend

\end{APPENDICES}

\bibliographystyle{informs2014} 
\bibliography{bibliography} 

\clearpage
\begin{APPENDICES}
\renewcommand{\thesection}{EC.\arabic{section}} 
\renewcommand{\thesubsection}{EC.\arabic{section}.\arabic{subsection}} 

\renewcommand*{\theHsection}{ec.\arabic{section}}
\renewcommand*{\theHsubsection}
  {ec.\arabic{section}.\arabic{subsection}}
\renewcommand*{\theHsubsubsection}
  {ec.\arabic{section}.\arabic{subsection}.\arabic{subsubsection}}

\pagenumbering{arabic} 
\renewcommand{\thepage}{ec\arabic{page}}
\OneAndAHalfSpacedXI
\begin{center}
\textbf{\large{}{}Electronic Companion to ``Adversarial Estimation of Binary Configuration Probabilities under Sparse Dependence''}
{\Large{} }{\Large\par}
\par\end{center}

\section{Additional Proofs and Results for Section~\ref{Sec:Rates:no-covariates}}
\proof{Proof of Proposition~\ref{prop:plug-in:no-covariate}}
By Lemma \ref{lemma:Score:nocovariate}, with probability $1-\epsilon-\delta_{\alpha}-\delta(1+\mathrm{o}(1))$, we have 
$\left\|\frac{{\rm diag}^{-1}(w)}{N}\!\!\sum_{i=1}^N\!\!\frac{W(\hat\alpha,Y_i)}{1+W(\hat\alpha,Y_i)'r_0}\right\|_\infty\!\!\!\!\!\leq\left\|\frac{{\rm diag}^{-1}(w)}{N}\!\!\sum_{i=1}^N\!\!\frac{W(\alpha_0,Y_i)}{1+W(\alpha_0,Y_i)'r_0}\right\|_\infty \!\!\!\!+C_1(1+\beta_N)\|\hat\alpha-\alpha_0\|
\frac{M^{1/2}K_{\alpha,1}}{\xi}$. So by the lower-bound condition on $\lambda$, $\left\|
\frac{{\rm diag}^{-1}(w)}{N}
\sum_{i=1}^N
\frac{W(\hat\alpha,Y_i)}
{1+W(\hat\alpha,Y_i)'r_0}
\right\|_\infty
\leq\frac{\lambda}{c}$. Let $\mathcal{W}:={\rm diag}(w)$. By H\"older's inequality, we have 
\begin{equation}\label{ineq:no-cov:holder}
\displaystyle v'\hat Uv \geq v'Uv - \|v\|_{w,1}^2\|\mathcal{W}^{-1}(U-\hat U)\mathcal{W}^{-1}\|_{\infty,\infty}.    
\end{equation}
We apply \eqref{ineq:no-cov:holder} for 
$U:=\frac{1}{N}\sum_{i=1}^N\frac{W(\alpha_0,Y_i)W(\alpha_0,Y_i)'}{(1+W(\alpha_0,Y_i)^\prime r_0)^2},\ \hat U :=\frac{1}{N}\sum_{i=1}^N\frac{W(\hat\alpha,Y_i)W(\hat\alpha,Y_i)'}{(1+W(\hat\alpha,Y_i)^\prime r_0)^2}$. Note that with probability $1-\delta_{\alpha}$, $1+W(\hat{\alpha},Y_i)'r_0\geq_{(1)}1+W(\alpha_0,Y_i)'r_0-|(\hat{\alpha}-\alpha_0)'\nabla_{\alpha}W(\tilde{\alpha},Y_i)'r_0|\geq_{(2)}\xi-M^{1/2}K_{\alpha,1}\|\hat{\alpha}-\alpha_0\|\geq_{(3)}\frac{3}{4}\xi>0$, where (1) follows from Taylor expansion and $\tilde{\alpha}$ is on the line segment between $\alpha_0, \hat{\alpha}$. So when $\alpha_0\in\mathcal{A}(\hat{\alpha})$, by the construction of $\mathcal{A}(\hat{\alpha})$, $\tilde\alpha\in\mathcal{A}(\hat{\alpha})$. Hence following from Assumption~\ref{ass:iid-min-den} and  Assumption~\ref{ass:smoothness-assumptions}, $\min_{i\in[N]}1+W(\hat{\alpha},Y_i)'r_0\geq\frac{3}{4}\xi>0$ with probability $1-\epsilon-3\delta_{\alpha}$. Hence with probability $1-C\epsilon-C\delta_{\alpha}-C\delta(1+\mathrm{o}(1))$, we have  
$$\begin{array}{rl}
&\displaystyle\frac{c_0}{N}\!\!\sum_{i=1}^N\!\!\left\{\!\frac{W(\alpha_0,Y_i)^\prime(\hat{r}^{\mathrm{PI}}-r_0)}{1+W(\alpha_0,Y_i)^\prime r_0}\!\right\}^2\!\!\!\!\!\leq_{(1)}\!\!\frac{c_0}{N}\!\!\sum_{i=1}^N\!\!\left\{\!\frac{W(\hat\alpha,Y_i)^\prime(\hat{r}^{\mathrm{PI}}-r_0)}{1+W(\hat\alpha,Y_i)^\prime r_0}\!\right\}^2\!\!\!\!\!+\!c_0\|\hat r^{\mathrm{PI}} \!\!\!-\! r_0\|_{w,1}^2\!\|\!\mathcal{W}^{-1}\!(\hat U\! -\! U)\!\mathcal{W}^{-1}\!\|_{\infty,\infty}\\
&\displaystyle \leq_{(2)}2\lambda\frac{c+1}{c-1}\norm{\hat{r}_T^{\mathrm{PI}}-r_0}_{w,1} + c_0\left(\frac{2c}{c-1}\right)^2\|\hat r_T^{\mathrm{PI}} - r_0\|_{w,1}^2 \|\mathcal{W}^{-1}\!(\hat U - U)\mathcal{W}^{-1}\|_{\infty,\infty}\\
&\displaystyle \leq_{(3)} \frac{2\lambda\sqrt{s}(c+1)/(c-1)}{\kappa_{\bar c}\sqrt{c_0}}\sqrt{\frac{c_0}{N} \sum_{i=1}^N\left\{\frac{W(\alpha_0,Y_i)^\prime(\hat{r}^{\mathrm{PI}}-r_0)}{1+W(\alpha_0,Y_i)^\prime r_0}\right\}^2}  \\
&\displaystyle\quad\quad + \frac{4c^2s\|\mathcal{W}^{-1}(\hat U - U)\mathcal{W}^{-1} \|_{\infty,\infty}}{(c-1)^2\kappa_{\bar c}^2}\frac{c_0}{N} \sum_{i=1}^N\left\{\frac{W(\alpha_0,Y_i)^\prime(\hat{r}^{\mathrm{PI}}-r_0)}{1+W(\alpha_0,Y_i)^\prime r_0}\right\}^2,
\end{array}$$
where (1) follows by (\ref{ineq:no-cov:holder}). The first term of (2) follows since by Lemma \ref{lemma:fast-rate:estimated-alpha}, for some universal constant $c_0>0$, with probability $1-\epsilon-\delta_{\alpha}-\delta(1+\mathrm{o}(1))$, we have $\frac{c_0}{N}\sum_{i=1}^N\!\left\{\frac{W(\hat\alpha,Y_i)^\prime(\hat{r}^{\mathrm{PI}}-r_0)}{(1+W(\hat\alpha,Y_i)^\prime r_0)}\right\}^2\!\!\leq2\lambda\frac{c+1}{c-1}\norm{\hat{r}_T^{\mathrm{PI}}-r_0}_{w,1}$. The second term of (2) follows since $\|\hat{r}^{\mathrm{PI}}-r_0\|_{w,1}=\|\hat{r}_T^{\mathrm{PI}}-r_0\|_{w,1}+\|\hat{r}_{T^c}^{\mathrm{PI}}-r_0\|_{w,1}\leq\frac{2c}{c-1}\|\hat{r}_T^{\mathrm{PI}}-r_0\|_{w,1}$ according to Lemma~\ref{lem:diff-restricted}. (3) follows from the definition of restricted eigenvalues and $\|\hat{r}_T^{\mathrm{PI}}-r_0\|_{w,1}\leq\sqrt{s}\|\hat{r}_T^{\mathrm{PI}}-r_0\|_{w,2}$. By Taylor expansion, for each $j\in[p]$, $i\in[N]$ there exists $\tilde \alpha^{ji}$ on the line segment between $\hat \alpha$ and $\alpha_0$ such that
$W_j(\hat\alpha,Y_i) = W_j(\alpha_0,Y_i) + (\hat\alpha-\alpha_0)'\nabla_\alpha W_j(\tilde \alpha^{ji},Y_i)$. Simiarly, for some $\tilde\alpha$ on the line segment between $\hat{\alpha}$ and $\alpha_0$, for any $j\in[p]$, 
$$\begin{array}{rl}
\displaystyle\frac{W_j(\hat\alpha,Y_i)}{1+W(\hat\alpha,Y_i)'r_0}&\displaystyle= \frac{W_j(\alpha_0,Y_i) +   (\hat\alpha-\alpha_0)'\nabla_\alpha W_j(\tilde \alpha^{ji},Y_i)}{1+W(\alpha_0,Y_i)'r_0 + (\hat\alpha-\alpha_0)'\nabla_\alpha W(\tilde \alpha,Y_i)'r_0} \\
&\displaystyle = \frac{W_j(\alpha_0,Y_i) }{1+W(\alpha_0,Y_i)'r_0} \left( 1 -  \frac{(\hat\alpha-\alpha_0)'\nabla_\alpha W(\tilde \alpha,Y_i)'r_0}{1+W(\alpha_0,Y_i)'r_0 + (\hat\alpha-\alpha_0)'\nabla_\alpha W(\tilde \alpha,Y_i)'r_0}\right)\\
&\displaystyle\quad + \frac{(\hat\alpha-\alpha_0)'\nabla_\alpha W_j(\tilde \alpha^{ji},Y_i)}{1+W(\alpha_0,Y_i)'r_0 + (\hat\alpha-\alpha_0)'\nabla_\alpha W(\tilde \alpha,Y_i)'r_0} = \frac{W_j(\alpha_0,Y_i) }{1+W(\alpha_0,Y_i)'r_0} + \widetilde R_j(\hat\alpha,Y_i).
\end{array}$$ 
By the construction of $\mathcal{A}(\hat{\alpha})$ as shown later in Section~\ref{sec:marginals}, $\hat{\alpha}\in\mathcal{A}(\hat{\alpha})$. Particularly, $|(\hat\alpha-\alpha_0)'\nabla_\alpha W(\tilde \alpha,Y_i)'r_0|\leq_{(a)} K_{\alpha,1}\sqrt{M}\|\hat\alpha-\alpha_0\|\leq_{(b)}K_{\alpha,1}c_1\bigg(\frac{M}{\sqrt{N}}+\sqrt{\frac{M\log(1/\delta)}{N}}\bigg)<_{(c)}\frac{1}{2}\xi$, where (a) holds with probability $1-\delta_{\alpha}$ from Assumption~\ref{ass:smoothness-assumptions}, (b) holds with probability $1-\delta$ according to Proposition~\ref{prop:concentration:no covariate}, (c) holds when $N$ is sufficiently large. Since $1+W(\alpha_0,Y_i)'r_0>\xi$ with probability $1-\epsilon$, $|\widetilde R_j(\hat\alpha,Y_i)|\leq \frac{CK_{\alpha,1}}{\xi}\left(1+\frac{|W_j(\alpha_0,Y_i)|}{1+W(\alpha_0,Y_i)'r_0}\right)M^{1/2}\|\hat\alpha-\alpha_0\|$ with probability $1-\delta_\alpha-\delta-\epsilon$ when $N$ is sufficiently large, where $C$ is an absolute constant.
Define $W_i := \frac{W(\alpha_0,Y_i)}{1+W(\alpha_0,Y_i)'r_0}$, $\hat W_i := \frac{W(\hat \alpha,Y_i)}{1+W(\hat \alpha,Y_i)'r_0}$. For any $l\in[p]$, define $\tilde R_{il}:=\tilde{R}_l(\hat{\alpha},Y_i)$. By Assumption \ref{ass:smoothness-assumptions}, with probability $1-\delta_\alpha$, $\left(1+\frac{|W_l(\alpha_0,Y_i)| }{1+W(\alpha_0,Y_i)'r_0}\right)\|\hat\alpha - \alpha_0\|\frac{M^{1/2}K_{\alpha,1}}{\xi} \leq\frac{1}{4}\left(1+\frac{|W_l(\alpha_0,Y_i)|}{1+W(\alpha_0,Y_i)'r_0}\right)$, and $\frac{1}{N}\sum_{i=1}^N |\tilde R_{il}\tilde R_{ik}|\leq\frac{C}{4N}\sum_{i=1}^N (1+|W_{il}|)|\tilde R_{ik}|$. Note $|\hat W_{il}\hat W_{ik} -W_{il}W_{ik}| \leq \left| W_{il}\tilde R_{ik} + \tilde R_{il} W_{ik} + \tilde R_{il}\tilde R_{ik}\right|$. Thus with probability $1-C'\delta_{\alpha}-C'\delta-C'\epsilon$, $\displaystyle\left| U_{lk} - \hat U_{lk} \right| \leq C'\frac{M^{1/2}K_{\alpha,1}}{\xi}\|\hat\alpha - \alpha_0\|\max_{l,k\in[p]}\frac{1}{N}\sum_{i=1}^N  \left(1+| W_{il}|\right)\left(1+| W_{ik}|\right)$ for some constant $C'$ independent of $s,N,M$, so the given conditions further imply that for some constant $\bar{C}$ independent of $s,N,M$, we have $s\|\mathcal{W}^{-1}(\hat U - U)\mathcal{W}^{-1} \|_{\infty,\infty}/\kappa_{\bar{c}}^2\leq \bar{C}sM^{1/2}K_{\alpha,1}\|\hat\alpha-\alpha_0\|/\{\xi\kappa_{\bar c}^2 \min_{j\in[p]} w_j^2\}\leq\delta_N\bar{C}$.
Recall that $\delta_N\rightarrow0$. So for sufficiently large $N$, $4c^2s\|\mathcal{W}^{-1}(\hat U - U)\mathcal{W}^{-1} \|_{\infty,\infty}/\{\kappa_{\bar{c}}^2(c-1)^2\}\leq1/2$. Let $x=\sqrt{\frac{c_0}{N}\sum_{i=1}^N\left\{\frac{W(\alpha_0,Y_i)^\prime(\hat{r}^{\mathrm{PI}}-r_0)}{1+W(\alpha_0,Y_i)^\prime r_0}\right\}^2}$. Note that $\frac{x^2}{2} \leq \frac{2\lambda\sqrt{s}}{\sqrt{c_0}\kappa_{\bar{c}}}\frac{c+1}{c-1} x \Rightarrow x \leq \frac{4\lambda\sqrt{s}}{\sqrt{c_0}\kappa_{\bar{c}}}\frac{c+1}{c-1}$. This implies 
$\sqrt{\frac{c_0}{N} \sum_{i=1}^N\left\{\frac{W(\alpha_0,Y_i)^\prime(\hat{r}^{\mathrm{PI}}-r_0)}{1+W(\alpha_0,Y_i)^\prime r_0}\right\}^2}\leq\frac{4\lambda\sqrt{s}}{\sqrt{c_0}\kappa_{\bar{c}}}\frac{c+1}{c-1}$. By Proposition \ref{prop:concentration:no covariate}, with probability $1-\delta$, $\|\hat{\alpha}-\alpha_0\|\leq c_1\big(\sqrt{\frac{M}{N}}+\sqrt{\frac{\log(1/\delta)}{N}}\big)$ for some $c_1>0$. So the result follows. \proofend

\medskip
\begin{example}\label{ex:plug-in}
Let $M^6=\mathrm{o}(N)$ and $M\to\infty$, $\alpha_{0j}=\frac{1}{2}, j\in[M]$. Here we use $r_{0\ell}$ to denote the coordinate of $r_0$ corresponding to configuration $\ell$. For a fixed constant $a>0$, let $r_{0\ell}= a\,M/\sqrt N$ when $|\ell|=2$, and $r_{0\ell}= 0$ when $|\ell|\ge 3$. So $s=\binom{M}{2}$. So $\mathbb P(Y=y)
=2^{-M}[1+\sum_{1\le j<k\le M}r_{0,jk}(2y_j-1)(2y_k-1)],\forall y\in\{0,1\}^M$. Particularly, the probability is well defined since $1+W(\alpha_0,Y_i)'r_0=(1+a\frac{M}{\sqrt{N}}\sum_{j<k}(2Y_{ij}-1)(2Y_{ik}-1))=1+\mathrm{o}(1)$. Also $\sup_y|W(\alpha_0,y)'r_0|\le\frac{aM}{\sqrt{N}}\frac{M^2}{2}=\frac{aM^3}{2\sqrt N}=\mathrm{o}(1)$. Proposition~\ref{prop:concentration:no covariate} implies that $\|\hat{\alpha}-\alpha_0\|\asymp\sqrt{M/N}$. Proposition~\ref{prop:plug-in:no-covariate} requires $\lambda=\bar{C}M/\sqrt{N}$ for some absolute constant $\bar{C}$. By smoothness conditions, $\bigl\|\mathbb E[W(\hat\alpha,Y)]-\mathbb E[W(\alpha_0,Y)]\bigr\|_\infty\lesssim \sqrt{\frac{M}{N}}$. Additionally, $\|\mathbb E[W(\alpha_0,Y)]\|_\infty=\max_{\ell}|r_{0\ell}|=
a\,\frac{M}{\sqrt N}$. Hence $\|\mathbb E[W(\hat\alpha,Y)]\|_\infty
\le
a\,\frac{M}{\sqrt N}+C\sqrt{\frac{M}{N}}
\le \frac{\lambda}{4}$ for $\bar{C}$ sufficiently large. Let $\hat Q(r)=\frac1N\sum_{i=1}^N \log\bigl(1+W(\hat\alpha,Y_i)'r\bigr)-\lambda\|r\|_1$. Since the coordinates of $W(\alpha_0,Y_i)$ are uniformly bounded,
Hoeffding's inequality and a union bound over $p\le 2^M$ coordinates yield $\left\|
\frac{1}{N}\sum_{i=1}^N W(\alpha_0,Y_i)-\mathbb{E}[W(\alpha_0,Y_i)]\right\|_\infty=\mathrm{O}_{\mathrm{P}}\!(\sqrt{M/N})=\mathrm{o}_{\mathrm{P}}\!(M/\sqrt{N})$. Then the smoothness conditions imposed imply $\left\|
\frac{1}{N}\sum_{i=1}^N W(\hat\alpha,Y_i)-\mathbb{E}[W(\hat\alpha,Y)]\right\|_\infty=\mathrm{O}_{\mathrm{P}}\!(\sqrt{M/N})=\mathrm{o}_{\mathrm{P}}\!(M/\sqrt{N})$. Therefore, with probability tending to one, $\left\|\frac{1}{N}\sum_{i=1}^N W(\hat\alpha,Y_i)\right\|_\infty\le\frac{\lambda}{2}$. On this event, for any feasible $r\neq 0$, $\hat Q(r)-\hat Q(0)\le\left(\frac1N\sum_{i=1}^N W(\hat\alpha,Y_i)\right)'r-\lambda\|r\|_1\le
\left\|\frac{1}{N}\sum_{i=1}^N W(\hat\alpha,Y_i)
\right\|_\infty \|r\|_1-\lambda\|r\|_1\le-\frac{\lambda}{2}\|r\|_1<0$. Hence $\hat r^{\mathrm{PI}}=0$ with probability tending to one. Further, on this event, $\sqrt{\frac{1}{N}\sum_{i=1}^N\frac{\{W(\alpha_0,Y_i)'(\hat{r}^{\mathrm{PI}}-r_0)\}^2}{(1+W(\alpha_0,Y_i)'r_0)^2}}=\sqrt{\frac{1}{N}\sum_{i=1}^N\frac{\{W(\alpha_0,Y_i)'r_0\}^2}{(1+W(\alpha_0,Y_i)'r_0)^2}}$. Note that $r_0$'s support corresponds to all the bundles $|\ell|=2$. Let $\Sigma=\mathbb{E}\left[\widetilde{W}_i\widetilde{W}_i'\right]\in\mathbb{R}^{s\times s}$, where $\widetilde{W}_i$ denotes the subvector of $W(\alpha_0,Y_i)$ restricted to coordinates corresponding to configurations with $|\ell|=2$ only, and $s={M\choose 2}$. Index the $s$ configurations with $|\ell|=2$ by edges $e=(j,k)$ of the complete graph $K_M$, where $j,k\in[M]$ and $j\neq k$. Then $\Sigma=\mathbf{I}_s+\frac{aM}{\sqrt{N}}A$. Here $\mathbf{I}_s$ is the $s\times s$ identity matrix, and $A$ is the adjacency matrix for the line graph of $K_M$, where $A_{e,e'}=1$ iff the two edges $e,e'$ share a vertex. For $M$ large, $\lambda_{\min}(\Sigma)$(i.e. the minimum eigenvalue of $\Sigma$) satisfies $\lambda_{\min}(\Sigma)=1-\frac{2aM}{\sqrt{N}}\geq1/2$ since $M^6=\mathrm{o}(N)$. By applying a standard matrix Chernoff lower-tail inequality, with high probability we get $\lambda_{\min}\left(\frac{1}{N}\sum_{i=1}^N\tilde{W}_i\tilde{W}_i'\right)\geq c_0$ for some absolute constant $c_0>0$. Note $1+W(\alpha_0,Y_i)'r_0=1+\frac{aM}{\sqrt{N}}\sum_{j<k}(2Y_{ij}-1)(2Y_{ik}-1)\in[\frac{1}{2},\frac{3}{2}]$ for $M$ sufficiently large. Thus $\sqrt{\frac{1}{N}\sum_{i=1}^N\frac{\{W(\alpha_0,Y_i)'r_0\}^2}{(1+W(\alpha_0,Y_i)'r_0)^2}}\geq\frac{2}{3}\sqrt{r_{0T}'\left(\frac{1}{N}\sum_{i=1}^N\widetilde{W}_i\widetilde{W}_i'\right)r_{0T}}\geq\Omega_p\left(M\sqrt{s/N}\right)$. 
\end{example}

\subsection{Technical Lemmas and Additional Results for Section \ref{Sec:Rates:no-covariates}}\label{appendix:no-covariate}
\begin{lemma}[Borell-TIS inequality \citep{borell1975brunn}]\label{lemma:borell} 
Let $X$ be a centered Gaussian random variable on $\mathbb{R}^n$ and $\displaystyle\sigma^2:=\max_{i\in[n]}\mathbb{E}[X_i^2]$. Then $\displaystyle\mathbb{P}\bigg(\bigg|\max_{i\in[n]}X_i-\mathbb{E}\big[\max_{i\in[n]}X_i\big]\bigg|>t\bigg)\leq 2e^{-t^2/2\sigma^2}$ holds for every $t>0$.  
\end{lemma}

\begin{lemma}\label{lemma:MinMax:Claim1}
Suppose Assumptions \ref{ass:iid-min-den}, \ref{ass:smoothness-assumptions} hold. Then $\min_{\alpha \in \mathcal{A}(\hat\alpha)} Q^{\mathrm{FO}}(\alpha,r_0)-Q^{\mathrm{FO}}(\alpha_0,r_0) \geq- C\max_{\alpha \in \mathcal{A}(\hat\alpha)}\|\alpha-\alpha_0\|_1 \left\| \frac{1}{N}\sum_{i=1}^N \frac{\nabla_\alpha W(\alpha_0,Y_i)'r_0}{1+W(\alpha_0,Y_i)'r_0} \right\|_\infty -C\frac{K_{\alpha,2}}{\xi^2}\max_{\alpha \in \mathcal{A}(\hat\alpha)} \|\alpha - \alpha_0\|^2$ holds  with probability $1-\epsilon-\delta_{\alpha}$, where $C$ is an absolute constant independent of $s,N,M$. Moreover, under Assumption \ref{ass:moment}, $\left\| \frac{1}{N}\sum_{i=1}^N \frac{\nabla_\alpha W(\alpha_0,Y_i)'r_0}{1+W(\alpha_0,Y_i)'r_0} \right\|_\infty \leq \Phi^{-1}(1-\delta/2M) C'\bar K_2\sqrt{1/N}$ holds with probability $1-\delta((1+\mathrm{o}(1)))$ for some universal constant $C'$.
\end{lemma}
\proof{Proof of Lemma \ref{lemma:MinMax:Claim1}}
Recall that $Q^{\mathrm{FO}}(\alpha,r)= \frac{1}{N}\sum_{i=1}^N \log(1+ W^{\mathrm{FO}}(\alpha,Y_i)'r)$ where $W^{\mathrm{FO}}(\alpha,y)=W(\hat \alpha,y)+\nabla_\alpha W(\hat\alpha,y)(\alpha-\hat\alpha)$. Suppose all conditions of Assumption~\ref{ass:smoothness-assumptions} hold and condition (iii) of Assumption~\ref{ass:iid-min-den} holds. Then for any $\alpha\in\mathcal{A}(\hat{\alpha})$,
\begin{equation}\label{eq:Main:MinMax}
Q^{\mathrm{FO}}(\alpha,r_0)-Q^{\mathrm{FO}}(\alpha_0,r_0)  = \frac{1}{N}\sum_{i=1}^N \log\left( 1 + (\alpha - \alpha_0)'\frac{\nabla_\alpha W(\hat\alpha,Y_i)'r_0}{1+W^{\mathrm{FO}}(\alpha_0,Y_i)'r_0} \right).
\end{equation}
By the mean value theorem, $(\alpha-\alpha_0)'\nabla_\alpha W(\hat\alpha,Y_i)'r_0 = (\alpha-\alpha_0)'\nabla_\alpha W(\alpha_0,Y_i)'r_0 + (\alpha-\alpha_0)'\{\nabla_\alpha^2 W(\bar{\alpha},Y_i)'r_0\}(\hat\alpha-\alpha_0)$ for some $\bar{\alpha}$, $\tilde \alpha$ in the line segment between $\hat\alpha$ and $\alpha_0$. Further, $W^{\mathrm{FO}}(\alpha_0,Y_i)'r_0= W(\hat \alpha,Y_i)'r_0 + (\alpha_0-\hat\alpha)'\nabla_\alpha W(\hat\alpha,Y_i)'r_0=_{(1)} W(\alpha_0,Y_i)'r_0 -\frac{1}{2}(\alpha_0-\hat\alpha)'\{\nabla_\alpha^2 W(\tilde\alpha,Y_i)'r_0\}(\alpha_0-\hat\alpha)$, and $\frac{1}{1+W^{\mathrm{FO}}(\alpha_0,Y_i)'r_0}= \frac{1}{1+W(\alpha_0,Y_i)'r_0}+\frac{W(\alpha_0,Y_i)'r_0 -W^{\mathrm{FO}}(\alpha_0,Y_i)'r_0}{(1+W(\alpha_0,Y_i)'r_0)(1+W^{\mathrm{FO}}(\alpha_0,Y_i)'r_0)} = \frac{\left(1 -\Delta R_0\right)}{1+W(\alpha_0,Y_i)'r_0}$, where (1) follows from the Taylor expansion of $W(\alpha_0,Y_i)'r_0$ around $\hat{\alpha}$, and $\Delta R_0:=-1/2[(\alpha_0-\hat\alpha)'\nabla_\alpha^2 W(\tilde\alpha,Y_i)'r_0(\alpha_0-\hat\alpha)]/\{1+W^{\mathrm{FO}}(\alpha_0,Y_i)'r_0\}$. Thus we have
\begin{equation}\label{eq:Aux:MinMax}
\begin{array}{rl}
\displaystyle\frac{(\alpha - \alpha_0)'\nabla_\alpha W(\hat\alpha,Y_i)'r_0}{1+W^{\mathrm{FO}}(\alpha_0,Y_i)'r_0} &\displaystyle\!\!=\!\!\frac{[\!(\alpha - \alpha_0)'\nabla_\alpha W(\alpha_0,Y_i)'r_0\!+\!(\alpha-\alpha_0)'\!\{\nabla_\alpha^2 W(\bar{\alpha},Y_i)'r_0\}(\hat\alpha-\alpha_0)\!]\!(1-\Delta R_0)}{1+W(\alpha_0,Y_i)'r_0}\\
&\displaystyle\!\!=\!\!\frac{(\alpha - \alpha_0)'\nabla_\alpha W(\alpha_0,Y_i)'r_0}{1+W(\alpha_0,Y_i)'r_0} +\widetilde R(\alpha,\hat\alpha,Y_i),
\end{array}
\end{equation}
where $|\widetilde R(\alpha,\hat\alpha,Y_i)|\leq_{(i)}  \|\alpha-\alpha_0\| \frac{M^{1/2}K_{\alpha,1}}{\xi} |\Delta R_0| +\max_{\tilde\alpha \in \mathcal{A}(\hat\alpha)} 2\|\tilde\alpha-\alpha_0\|^2 \frac{K_{\alpha,2}}{\xi} (1+|\Delta R_0|)\leq_{(ii)} 3\max_{\tilde\alpha \in \mathcal{A}(\hat\alpha)} \|\tilde\alpha-\alpha_0\|^2 \frac{K_{\alpha,2}}{\xi^2}<_{(iii)}\frac{3}{16}$. Here inequality (i) uses condition (iii) of Assumption~\ref{ass:iid-min-den} and (2), (4) of Assumption~\ref{ass:smoothness-assumptions}. Since $\hat{\alpha}\in\mathcal{A}(\hat{\alpha})$, then (2), (4), (5) of Assumption~\ref{ass:smoothness-assumptions} imply $|v'\nabla_\alpha^2 W(\tilde\alpha,Y_i)'r_0| \leq K_{\alpha,2}\|v\| $ and $ | W^{\mathrm{FO}}(\alpha_0,Y_i)'r_0 - W(\alpha_0,Y_i)'r_0| \leq \frac{1}{16} \xi$, which then imply that $|\Delta R_0|\leq \frac{16}{15}\max_{\tilde \alpha \in \mathcal{A}(\hat \alpha)} \|\tilde\alpha-\alpha_0\|^2 \frac{K_{\alpha,2}}{\xi}$, and (5) of Assumption~\ref{ass:smoothness-assumptions} further implies inequality (ii). inequality (iii) follows from (5) of Assumption~\ref{ass:smoothness-assumptions}. Next, note that $\log(1+x)\geq \frac{x}{1+x} = x - \frac{x^2}{1+x}$ and conditions (2), (4), (5) of Assumption~\ref{ass:smoothness-assumptions} imply that $\left| \frac{(\alpha - \alpha_0)'\nabla_\alpha W(\alpha_0,Y_i)'r_0}{1+W(\alpha_0,Y_i)'r_0} +\widetilde R(\alpha,\hat\alpha,Y_i))\right| \leq 1/2$. Further by (\ref{eq:Main:MinMax}) and (\ref{eq:Aux:MinMax}), for a universal constant $C$,
$$\begin{array}{rl}
&\quad Q^{\mathrm{FO}}(\alpha,r_0)-Q^{\mathrm{FO}}(\alpha_0,r_0) \\
&\displaystyle \geq  \frac{1}{N}\sum_{i=1}^N \frac{(\alpha - \alpha_0)'\nabla_\alpha W(\alpha_0,Y_i)'r_0}{1+W(\alpha_0,Y_i)'r_0} +\widetilde R(\alpha,\hat\alpha,Y_i) - \!\frac{2}{N}\sum_{i=1}^N \left(\frac{(\alpha - \alpha_0)'\nabla_\alpha W(\alpha_0,Y_i)'r_0}{1+W(\alpha_0,Y_i)'r_0} +\widetilde R(\alpha,\hat\alpha,Y_i))\right)^2 \\
&\displaystyle \geq -C\max_{\alpha \in \mathcal{A}(\hat\alpha)} \|\alpha-\alpha_0\|_1 \left\|\frac{1}{N}\sum_{i=1}^N \frac{\nabla_\alpha W(\alpha_0,Y_i)'r_0}{1+W(\alpha_0,Y_i)'r_0}  \right\|_\infty - C\frac{K_{\alpha,2}}{\xi^2}\max_{\alpha \in \mathcal{A}(\hat\alpha)} \|\alpha-\alpha_0\|^2,
\end{array}$$
where the second inequality follows from H\"older's inequality, the bound derived above for $|\widetilde R(\alpha,\hat\alpha,Y_i)|$ and Assumption~\ref{ass:smoothness-assumptions}. The first statement then holds since (iii) of Assumption~\ref{ass:iid-min-den} holds with probability $1-\epsilon$ and the conditions of Assumption~\ref{ass:smoothness-assumptions} hold with probability $1-\delta_{\alpha}$. The second statement of the Lemma follows from an application of self-normalized moderate deviation theory.  Let $\tilde X_{ij} := \nabla_{\alpha_j} W(\alpha_0,Y_i)'r_0 / \{1+W(\alpha_0,Y_i)'r_0\}$. so under Assumptions \ref{ass:iid-min-den}, \ref{ass:moment} we have $\min_{j\in[M]}\mathbb{E}[\tilde X_{ij}^2 ] \geq \bar\psi_2^2 $, $\max_{j\in[M]}\mathbb{E}[ |\tilde X_{ij}|^3 ]^{1/3} \leq \bar K_3$ and $\sqrt{\frac{\log^3p}{N}}\leq\sqrt{\frac{M^3}{N}}\leq\delta_N$, thus 
$$\begin{array}{rl}
&\displaystyle \mathbb{P} \left( \left| \frac{1}{N}\sum_{i=1}^N \tilde X_{ij} \right| > \frac{\Phi^{-1}(1-\delta/2M)}{N^{1/2}} \sqrt{\frac{1}{N}\sum_{i=1}^N \tilde X_{ij}^2}\mbox{ for some }j\in[M]\right)  \\
&\displaystyle \leq  \sum_{j\in [M]} \mathbb{P} \left( \left|\frac{1}{N}\sum_{i=1}^N\tilde X_{ij} \right| >\frac{\Phi^{-1}(1-\delta/2M)}{N^{1/2}}\sqrt{\frac{1}{N}\sum_{i=1}^N \tilde X_{ij}^2}\right)\leq \delta(1+\mathrm{o}(1)),
\end{array}$$
where the first inequality follows from the union bound, and the second inequality follows by applying Theorem 2.3 in \citet{jing2003self}. Hence the result follows.
\proofend

\begin{lemma}\label{lemma:Score:nocovariate}
Define $\displaystyle\beta_N := \max_{j\in[p]}\frac{1}{w_j}\frac{1}{N}\sum_{i=1}^N \frac{|W_j(\alpha_0,Y_i)|}{1+W(\alpha_0,Y_i)'r_0}$.
Suppose Assumption~\ref{ass:iid-min-den} and (i)--(iv):\\
(i) $\displaystyle\max_{\alpha \in\mathcal{A}(\hat{\alpha}), j\in[p]}|v'\nabla_\alpha W_j(\alpha,Y_i)| \leq M^{1/2}K_{\alpha,1}\|v\|$, (ii) $\max_{\tilde \alpha \in \mathcal{A}(\hat\alpha), j\in [p]} |v'\nabla_\alpha^2 W_j(\tilde\alpha,Y_i)u| \leq K_{\alpha,2}\|v\|\ \|u\| $, (iii) $\displaystyle|v'\nabla_\alpha W(\alpha_0,Y_i)'r_0| \leq M^{1/2}K_{\alpha,1}\|v\|$, $\max_{\tilde \alpha \in \mathcal{A}(\hat\alpha)} |v'\nabla_\alpha^2 W(\tilde\alpha,Y_i)'r_0v| \leq K_{\alpha,2}\|v\|^2$, (iv) $\displaystyle\|\alpha-\alpha_0\|M^{1/2}K_{\alpha,1}/\xi < 1/4$,  $4\max_{\tilde\alpha \in \mathcal{A}(\hat\alpha)}\|\tilde \alpha-\alpha_0\|^2 K_{\alpha,2}/\xi \leq 1/4$, we have 
$$\begin{array}{rl}
&\displaystyle\left\|\frac{{\rm diag}^{-1}(w)}{N}\!\!\sum_{i=1}^N \frac{W(\hat\alpha,Y_i)}{1+W(\hat\alpha,Y_i)'r_0}\right\|_\infty\!\!\!\!\!\!\leq \left\|\frac{{\rm diag}^{-1}(w)}{N}\!\!\!\sum_{i=1}^N\frac{W(\alpha_0,Y_i)}{1+W(\alpha_0,Y_i)'r_0}\right\|_\infty\!\!\!\!\!\!+\!\frac{C(1+\beta_N)M^{1/2}K_{\alpha,1}\|\hat\alpha-\alpha_0\|}{\xi} \\
&\displaystyle\left\|\frac{{\rm diag}^{-1}(w)}{N}\!\!\sum_{i=1}^N\!\!\frac{W^{\mathrm{FO}}(\alpha_0,Y_i)}{1+W^{\mathrm{FO}}(\alpha_0,Y_i)'r_0}\right\|_\infty\!\!\!\!\!\!\leq \left\|\frac{{\rm diag}^{-1}(w)}{N}\!\!\sum_{i=1}^N \frac{W(\alpha_0,Y_i)}{1+W(\alpha_0,Y_i)'r_0}\right\|_\infty\!\!\!\!\!\!+\!\frac{C (1+\beta_N)K_{\alpha,2}\|\hat\alpha-\alpha_0\|^2}{\xi}
\end{array}$$
where $C$ is an absolute constant in the above, $w=(w_1,\ldots,w_p)$, and for each $j\in[p]$, $w_j$ is the weight for \eqref{eq:PMLE-no-covariates} and \eqref{eq:PMLE-no-covariates-bilinear}. Moreover, with probability $1-\delta(1+\mathrm{o}(1))$,  $\left\|\frac{{\rm diag}^{-1}(w)}{N}\!\!\sum_{i=1}^N \!\!\frac{W(\alpha_0,Y_i)}{1+W(\alpha_0,Y_i)'r_0}\right\|_\infty\!\!\!\!\!\!\leq\frac{\Phi^{-1}(1-\delta/(2p))}{\sqrt{N}}\max_{j\in[p]}\frac{1}{w_j}\sqrt{\frac{1}{N}\sum_{i=1}^N \frac{W_j^2(\alpha_0,Y_i)}{(1+W(\alpha_0,Y_i)'r_0)^2}}$ holds under Assumptions \ref{ass:iid-min-den} and \ref{ass:moment}.
\end{lemma}

\proof{Proof of Lemma \ref{lemma:Score:nocovariate}}
Note that by Taylor expansion, for each $j\in [p], i\in[N]$ there exist $\tilde \alpha^{ji}$, $\tilde{\alpha}$ in the line segment between $\hat \alpha$ and $\alpha_0$ such that
$W_j(\hat\alpha,Y_i) = W_j(\alpha_0,Y_i) + (\hat\alpha-\alpha_0)'\nabla_\alpha W_j(\tilde \alpha^{ji},Y_i)$. Thus 
$$\begin{array}{rl}
\displaystyle\frac{W_j(\hat\alpha,Y_i)}{1+W(\hat\alpha,Y_i)'r_0}&\displaystyle= \frac{W_j(\alpha_0,Y_i) }{1+W(\alpha_0,Y_i)'r_0} \left( 1 -  \frac{(\hat\alpha-\alpha_0)'\nabla_\alpha W(\tilde \alpha,Y_i)'r_0}{1+W(\alpha_0,Y_i)'r_0 + (\hat\alpha-\alpha_0)'\nabla_\alpha W(\tilde \alpha,Y_i)'r_0}\right)\\
&\displaystyle\quad + \frac{(\hat\alpha-\alpha_0)'\nabla_\alpha W_j(\tilde \alpha^{ji},Y_i)}{1+W(\alpha_0,Y_i)'r_0 + (\hat\alpha-\alpha_0)'\nabla_\alpha W(\tilde \alpha,Y_i)'r_0} = \frac{W_j(\alpha_0,Y_i) }{1+W(\alpha_0,Y_i)'r_0} + \widetilde R(\hat\alpha,Y_i),
\end{array}$$ 
where $\displaystyle|\widetilde R(\hat\alpha,Y_i)| \leq 2\left(1+\frac{|W_j(\alpha_0,Y_i)|}{1+W(\alpha_0,Y_i)'r_0}\right) \frac{\|\hat\alpha-\alpha_0\| M^{1/2}K_{\alpha,1}}{\xi}$ following from (iii) and (iv). So the first inequality holds.

Next we prove the second inequality. By Taylor expansion, for each $j\in [p], i\in[n]$ there is $\bar\alpha_{ji}$ in the line segment between $\hat \alpha$ and $\alpha_0$ such that
$W_j(\alpha_0,Y_i) = W_j(\hat{\alpha},Y_i) + (\alpha_0-\hat{\alpha})'\nabla_\alpha W_j(\hat{\alpha},Y_i) + \frac{1}{2}(\alpha_0-\hat\alpha)'\nabla_\alpha^2 W_j(\bar{\alpha}_{ji},Y_i)(\alpha_0-\hat\alpha)$. Recall that $W^{\mathrm{FO}}(\alpha,y)=W(\hat \alpha,y)+\nabla_\alpha W(\hat\alpha,y)(\alpha-\hat\alpha)$. Thus 
$$ \begin{array}{rl}
W^{\mathrm{FO}}_j(\alpha_0,Y_i) &= W_j(\hat \alpha,Y_i) + (\alpha_0-\hat\alpha)'\nabla_\alpha W_j(\hat\alpha,Y_i)\\
& = W_j(\alpha_0,Y_i)-\frac{1}{2}(\alpha_0-\hat\alpha)'\nabla_\alpha^2 W_j(\bar{\alpha}_{ji},Y_i)(\alpha_0-\hat\alpha)\\
\displaystyle\frac{1}{1+W^{\mathrm{FO}}(\alpha_0,Y_i)'r_0} &\displaystyle= \frac{1}{1+W(\alpha_0,Y_i)'r_0}+\frac{W(\alpha_0,Y_i)'r_0 -W^{\mathrm{FO}}(\alpha_0,Y_i)'r_0}{(1+W(\alpha_0,Y_i)'r_0)(1+W^{\mathrm{FO}}(\alpha_0,Y_i)'r_0)}.\ \mbox{So}
\end{array}$$ 
\begin{equation}\label{eq:Wj:WjFO:nocovariate}  \frac{W^{\mathrm{FO}}_j(\alpha_0,Y_i) }{1+W^{\mathrm{FO}}(\alpha_0,Y_i)'r_0} =  \frac{W_j(\alpha_0,Y_i)}{1+W(\alpha_0,Y_i)'r_0} + \widetilde R_0.
\end{equation}
Using Assumption~\ref{ass:iid-min-den} and conditions (i) -- (iv), 
$$\begin{array}{rl}
\displaystyle|\widetilde R_0| \leq C\left(1+\frac{|W_j(\alpha_0,Y_i)| }{1+W(\alpha_0,Y_i)'r_0}\right) \|\hat\alpha - \alpha_0\|^2 K_{\alpha,2}/\xi\ \ \mbox{for some absolute constant }C.
\end{array}$$
The third inequality follows by applying the self-normalized moderate deviation theory. Let $\tilde X_{ij} := W_j(\alpha_0,Y_i) / \{1+W(\alpha_0,Y_i)'r_0\}$ so under Assumptions \ref{ass:iid-min-den} and \ref{ass:moment} we have $$\min_{j\in[p]}\mathbb{E}[\tilde X_{ij}^2 ]^{1/2}\geq \bar\psi_2, \qquad \max_{j\in[p]}\mathbb{E}[ |\tilde X_{ij}|^3 ]^{1/3} \leq \bar K_3, \qquad \sqrt{\frac{\log^3p}{N}}\leq\sqrt{\frac{M^3}{N}}\leq\delta_N,$$
so we have 
$$\begin{array}{rl}
&\displaystyle \mathbb{P} \left( \left| \frac{1}{N}\sum_{i=1}^N \frac{\tilde X_{ij}}{w_j} \right| > \frac{\Phi^{-1}(1-\delta/(2p))}{w_j N^{1/2}} \sqrt{\frac{1}{N}\sum_{i=1}^N \tilde X_{ij}^2} \mbox{ for some }j\in[p] \right)\\
&\displaystyle \leq  \sum_{j\in [p]} \mathbb{P} \left( \left|\frac{1}{N}\sum_{i=1}^N\frac{\tilde X_{ij}}{w_j} \right| > \frac{\Phi^{-1}(1-\delta/(2p))}{w_jN^{1/2}} \sqrt{\frac{1}{N}\sum_{i=1}^N \tilde X_{ij}^2 } \right)\leq \delta(1+\mathrm{o}(1)).
\end{array}$$
where the first inequality follows from the union bound, and the second inequality follows by applying Theorem 2.3 of \citet{jing2003self}.
\proofend

\begin{lemma}\label{lemma:MatrixConnection}
Suppose conditions (i) -- (iv) of Lemma~\ref{lemma:Score:nocovariate} hold. Let $U:= \frac{1}{N}\sum_{i=1}^N \frac{W(\alpha_0,Y_i)W(\alpha_0,Y_i)'}{(1+W(\alpha_0,Y_i)'r_0)^2}$ and 
$\hat U:= \frac{1}{N}\sum_{i=1}^N \frac{W^{\mathrm{FO}}(\alpha_0,Y_i)  W^{\mathrm{FO}}(\alpha_0,Y_i)'}{(1+W^{\mathrm{FO}}(\alpha_0,Y_i)'r_0)^2}$. 
Then $\left| U_{lk} - \hat U_{lk} \right| \leq 5\|\hat\alpha - \alpha_0\|^2 \frac{K_{\alpha,2}}{\xi} \max_{l,k\in[p]}\frac{1}{N}\sum_{i=1}^N(1+|W_{il}|)(1+|W_{ik}|)$, where $W_i:=W(\alpha_0,Y_i)/(1+W(\alpha_0,Y_i)'r_0)$.
\end{lemma}
\proof{Proof of Lemma \ref{lemma:MatrixConnection}}
Note that by (\ref{eq:Wj:WjFO:nocovariate}) we have $\frac{W^{\mathrm{FO}}(\alpha_0,Y_i) }{1+W^{\mathrm{FO}}(\alpha_0,Y_i)'r_0} =  \frac{W(\alpha_0,Y_i)}{1+W(\alpha_0,Y_i)'r_0} + \widetilde R(\hat\alpha,\alpha_0,Y_i)$,
where $|\widetilde R_j(\hat{\alpha},\alpha_0,Y_i)| \leq  2\left(1+\frac{|W_j(\alpha_0,Y_i)| }{1+W(\alpha_0,Y_i)'r_0}\right) \|\hat\alpha - \alpha_0\|^2 K_{\alpha,2}/\xi$, for any $j\in[p]$. Define 
$\hat W_i := \frac{W^{\mathrm{FO}}(\alpha_0,Y_i) }{1+W^{\mathrm{FO}}(\alpha_0,Y_i)'r_0}$, $\tilde R_{il}:=\tilde R_l(\hat{\alpha},\alpha_0,Y_i)$. Thus 
$\left|\hat W_{il}\hat W_{ik} -W_{il}W_{ik}\right| \leq \left| W_{il}\tilde R_{ik} + \tilde R_{il} W_{ik} + \tilde R_{il}\tilde R_{ik}\right|$. By (iv) of Lemma~\ref{lemma:Score:nocovariate}, $\frac{1}{N}\sum_{i=1}^N|\tilde{R}_{ik}|\left(1+\frac{|W_l(\alpha_0,Y_i)|}{1+W(\alpha_0,Y_i)'r_0}\right)2\|\hat\alpha - \alpha_0\|^2 \frac{K_{\alpha,2}}{\xi} \leq \frac{1}{2N}\sum_{i=1}^N|\tilde{R}_{ik}|\big(1+\frac{|W_l(\alpha_0,Y_i)|}{1+W(\alpha_0,Y_i)'r_0}\big)$, so $\frac{1}{N}\sum_{i=1}^N |\tilde R_{il}\tilde R_{ik}|\leq\frac{1}{2N}\sum_{i=1}^N (1+|W_{il}|)|\tilde R_{ik}|$. So the upper bound for $\left| U_{lk} - \hat U_{lk} \right|$ holds by the upper bound of $\tilde{R}_{il}$ for any $l\in[p]$.
\proofend

\begin{lemma}\label{lem:diff-restricted}
If {$\displaystyle\norm{{\rm diag}^{-1}(w)\frac{1}{N}\sum_{i=1}^N\frac{W(\hat\alpha,Y_i)}{1+W(\hat\alpha,Y_i)^\prime r_0}}_\infty\leq\frac{\lambda}{c}$} for some constant $c>1$, then for $\nu=\hat r^{\mathrm{PI}}-r_0$, $T$ as the support of $r_0$, we have $\displaystyle\norm{\nu_{T^c}}_{w,1}\leq\frac{c+1}{c-1}\norm{\nu_T}_{w,1}$. 
\end{lemma}
\proof{Proof of Lemma \ref{lem:diff-restricted}}
For a given $\lambda$, $\frac{1}{N}\sum_{i=1}^N\left[\log\left(1+W(\hat\alpha,Y_i)^\prime\hat r^{\mathrm{PI}}\right)-\log\left(1+W\left(\hat\alpha,Y_i\right)^\prime r_0\right)\right]\geq\lambda\left(\norm{\hat r^{\mathrm{PI}}}_{w,1}-\norm{r_0}_{w,1}\right)$.
Let $\hat Q(\alpha,r)=\frac{1}{N}\sum_{i=1}^N\log\left(1+W\left(\alpha,Y_i\right)^\prime r\right)$, we can write the above inequality as $\hat Q(\hat\alpha,\hat r^{\mathrm{PI}})-\hat Q(\hat\alpha,r_0)\geq\lambda\left(\norm{\hat r^{\mathrm{PI}}}_{w,1}-\norm{r_0}_{w,1}\right)$,
so using concavity of $\hat Q$ in parameter $r$, we have $\lambda\left(\norm{r_0}_{w,1}-\norm{\hat r^{\mathrm{PI}}}_{w,1}\right)\geq\hat Q\left(\hat\alpha,r_0\right)-\hat Q\left(\hat\alpha,\hat r^{\mathrm{PI}}\right)\geq\nabla_r\hat Q(\hat\alpha,r_0)^T(r_0-\hat r^{\mathrm{PI}})\geq_{(a)}-\frac{\lambda}{c}\norm{\nu}_{w,1}$,
where (a) follows from H\"older's inequality. Note that $\norm{\hat r^{\mathrm{PI}}}_{w,1}=\norm{r_0+\nu_T}_{w,1}+\norm{\nu_{T^c}}_{w,1}\geq\norm{r_0}_{w,1}-\norm{\nu_T}_{w,1}+\norm{\nu_{T^c}}_{w,1}$, so $\lambda(\norm{r_0}_{w,1}-\norm{\hat r^{\mathrm{PI}}}_{w,1})\leq\lambda\left(\norm{\nu_T}_{w,1}-\norm{\nu_{T^c}}_{w,1}\right)$. Thus $\lambda\left(\norm{\nu_T}_{w,1}-\norm{\nu_{T^c}}_{w,1}\right)\geq-\frac{\lambda}{c}\norm{\nu}_{w,1}=-\frac{\lambda}{c}\left(\norm{\nu_T}_{w,1}+\norm{\nu_{T^c}}_{w,1}\right)$,
implying $\frac{c+1}{c-1}\norm{\nu_T}_{w,1}\geq\norm{\nu_{T^c}}_{w,1}$.
\proofend

\begin{lemma}\label{lemma:fast-rate:estimated-alpha}
    Suppose Assumptions \ref{ass:iid-min-den}  and \ref{ass:moment} hold. Suppose $\norm{{\rm diag}^{-1}(w)\frac{1}{N}\sum_{i=1}^N\frac{W(\hat\alpha,Y_i)}{1+W(\hat\alpha,Y_i)^\prime r_0}}_\infty\leq\frac{\lambda}{c}$ for some $c>1$ and $\min_{i\in[N]}1+W(\hat\alpha,Y_i)^\prime r_0\geq \underline{c}\xi$ for an absolute constant $\underline{c}$. Then we have $\frac{c_0}{N}\sum_{i=1}^N\left\{\frac{W(\hat\alpha,Y_i)^\prime(\hat{r}^{\mathrm{PI}}-r_0)}{(1+W(\hat\alpha,Y_i)^\prime r_0)}\right\}^2\leq2\lambda\frac{c+1}{c-1}\norm{\hat{r}_T^{\mathrm{PI}}-r_0}_{w,1}$, where $c_0>0$ is a universal constant.
\end{lemma}
\proof{Proof of Lemma \ref{lemma:fast-rate:estimated-alpha}}
Let $T=\mathrm{supp}(r_0)$ and $f(Y_i,\alpha,r)=1+W(\alpha,Y_i)'r$. 
By optimality of $\hat{r}^{\mathrm{PI}}$, we have $\frac{1}{N}\sum_{i=1}^N\log f(Y_i,\hat\alpha,
\hat{r}^{\mathrm{PI}})-\lambda\norm{\hat{r}^{\mathrm{PI}}}_{w,1}\geq\frac{1}{N}\sum_{i=1}^N\log f(Y_i,\hat\alpha,r_0)-\lambda\norm{r_0}_{w,1}$. So $\lambda\left(\norm{\hat r^{\mathrm{PI}}}_{w,1}-\norm{r_0}_{w,1}\right)\leq\frac{1}{N}\sum_{i=1}^N\log f(Y_i,\hat\alpha,\hat r^{\mathrm{PI}})-\log f(Y_i,\hat\alpha,r_0)$. Moreover, $\displaystyle\log f(Y_i,\hat\alpha,\hat{r}^{\mathrm{PI}})-\log f(Y_i,\hat\alpha, r_0)=\log\left(1+\frac{W(\hat\alpha,Y_i)^\prime(\hat{r}^{\mathrm{PI}}-r_0)}{1+W(\hat\alpha,Y_i)^\prime r_0}\right)$. Additionally, note that given any $\bar{y}>-1$, let $c'=\frac{1}{2(1+\max\{\bar{y},0\})}$, then $\log(1+y)\leq y-c'y^2$ holds for $y\in(-1,\bar{y}]$. Let $\bar F$ be a universal constant such that $\max_{i\in[N]}
\frac{1+W(\hat\alpha,Y_i)'\hat r}
     {1+W(\hat\alpha,Y_i)'r_0}
\leq \bar{F}$, set $c_0=\{2\max(1,\bar F)\}^{-1}$. Hence we have $\log f(Y_i,\hat\alpha,\hat{r}^{\mathrm{PI}})-\log f(Y_i,\hat\alpha, r_0)\leq\frac{W(\hat\alpha,Y_i)^\prime(\hat{r}^{\mathrm{PI}}-r_0)}{1+W(\hat\alpha,Y_i)^\prime r_0}-c_0\left\{\frac{W(\hat\alpha,Y_i)^\prime(\hat{r}^{\mathrm{PI}}-r_0)}{1+W(\hat\alpha,Y_i)^\prime r_0}\right\}^2$, where $c_0>0$ is a universal constant.
Thus $\displaystyle\frac{c_0}{N}\sum_{i=1}^N\left\{\frac{W(\hat\alpha,Y_i)^\prime(\hat{r}^{\mathrm{PI}}-r_0)}{1+W(\hat\alpha,Y_i)^\prime r_0}\right\}^2\leq\frac{1}{N}\sum_{i=1}^N\frac{W(\hat\alpha,Y_i)^\prime(\hat{r}^{\mathrm{PI}}-r_0)}{1+W(\hat\alpha,Y_i)^\prime r_0}+\lambda\left(\norm{r_0}_{w,1}-\norm{\hat r^{\mathrm{PI}}}_{w,1}\right)$. Since $\norm{{\rm diag}^{-1}(w) \frac{1}{N}\sum_{i=1}^N\frac{W(\hat\alpha,Y_i)}{1+W(\hat\alpha,Y_i)^\prime r_0}}_\infty\leq\frac{\lambda}{c}$, by H\"older's inequality, we have 
$$\begin{array}{rl}
&\quad\displaystyle\frac{c_0}{N}\sum_{i=1}^N\frac{\left[W(\hat\alpha,Y_i)^\prime(\hat{r}^{\mathrm{PI}}-r_0)\right]^2}{(1+W(\hat\alpha,Y_i)^\prime r_0)^2}\\
&\displaystyle\leq\norm{\frac{{\rm diag}^{-1}(w)}{N}\sum_{i=1}^N\frac{W(\hat\alpha,Y_i)}{1+W(\hat\alpha,Y_i)^\prime r_0}}_\infty\norm{\hat{r}^{\mathrm{PI}}-r_0}_{w,1}+\lambda\norm{r_0}_{w,1}-\lambda\norm{\hat{r}^{\mathrm{PI}}}_{w,1}\\
&\displaystyle\leq\lambda(1+1/c)\norm{\hat{r}^{\mathrm{PI}}-r_0}_{w,1}\leq\lambda(1+1/c)(\norm{\hat{r}_T^{\mathrm{PI}}-r_0}_{w,1}+\norm{\hat{r}_{T^c}^{\mathrm{PI}}}_{w,1})\\
&\displaystyle\leq_{(a)}\lambda(1+1/c)\norm{\hat{r}_T^{\mathrm{PI}}-r_0}_{w,1}\left(1+\frac{c+1}{c-1}\right)=2\lambda\frac{c+1}{c-1}\norm{\hat{r}_T^{\mathrm{PI}}-r_0}_{w,1},
\end{array}$$ 
where inequality (a) follows from Lemma~\ref{lem:diff-restricted}. Hence the result follows.
\proofend

\section{Additional Definitions and Proofs for Section \ref{Sec:Rates:covariates}}\label{sec:proof:cov}
\begin{equation}\label{eq:PI:with-covariate}
\hat r^{\mathrm{PI}}(\cdot,x) \in \arg \max_{a, b}  \displaystyle  \frac{1}{N}\sum_{i=1}^N  K_h\left( \mbox{$X_i-x$}\right)  \log( 1+W(\hat\alpha (X_i),Y_i)'r_{a,b}(X_i,x)) - \lambda_x\|r_{a,b}\|_{w_x,h,1}.
\end{equation}
\begin{equation}\label{eq:adversarial:with-covariate}
\hat r^{A}(\cdot,x) \in \arg \max_{a, b} \min_{\alpha \in \mathcal{A}_x(\hat\alpha)} \displaystyle  \frac{1}{N}\sum_{i=1}^N  K_h\left( \mbox{$X_i-x$}\right)  \log( 1+W(\alpha(X_i),Y_i)'r_{a,b}(X_i,x) ) - \lambda_x \|r_{a,b}\|_{w_x,h,1}.
\end{equation}
We define $\beta_{N,x}\!\!:={\displaystyle\max_{\substack{j\in[p]\\
k\in[d]}}}\left\{\frac{1}{w_{xj}}\frac{1}{N}\sum_{i=1}^N\frac{K_h(X_i-x)|W_j(\alpha_0(X_i),Y_i)|}{1+W(\alpha_0(X_i),Y_i)^\prime r_0(X_i)},\frac{1}{w_{xjk}}\frac{1}{N}\sum_{i=1}^N\frac{K_h(X_i-x)|W_j(\alpha_0(X_i),Y_i)(X_{ik}-x_k)/h|}{1+W(\alpha_0(X_i),Y_i)^\prime r_0(X_i)}\right\}$. Similar to its no-covariate counterpart, $\underline{\lambda}^{\mathrm{FO}}(x)$ in the below characterizes the lower bound on the penalty parameter $\lambda$ for the first-order estimator in the with-covariate case:
\begin{equation}\label{eq:lambda:FO:covariate}
\begin{array}{rl}
\underline{\lambda}^{\mathrm{FO}}(x)&\displaystyle:=\frac{\Phi^{-1}(1-\delta/\{2p(1+d)\})}{(Nh^d)^{1/2}}{\displaystyle \max_{j\in [p]}} \frac{1}{w_{xj}} \sqrt{\frac{1}{N}\sum_{i=1}^N\frac{K_h(X_i-x)W_j^2(\alpha_0(X_i),Y_i)}{(1+W(\alpha_0(X_i),Y_i)'r_0(X_i))^2}}\\
&\displaystyle\quad+\frac{\Phi^{-1}(1-\delta/\{2p(1+d)\})}{(Nh^d)^{1/2}}{\displaystyle \max_{j\in [p], k\in[d]}} \frac{1}{w_{xjk}} \sqrt{\frac{1}{N}\sum_{i=1}^N\frac{K_h(X_i-x)W_j(\alpha_0(X_i),Y_i)^2\big(\frac{X_{ik}-x_k}{h}\big)^2}{(1+W(\alpha_0(X_i),Y_i)'r_0(X_i))^2}}\\
&\displaystyle\quad+(1+\beta_{N,x})K_{\alpha,2}\max_{\tilde \alpha\in\mathcal{A}_x(\hat\alpha), \|X_i-x\|\leq h}\norm{\tilde{\alpha}(X_i)-\alpha_0(X_i)}^2/\xi+h^2.
\end{array}
\end{equation}
Define $\underline{\delta}_{N,x}$ that characterizes the lower bound on $\delta_N$ in the with-covariate case and parallels its no-covariate counterpart in Theorem~\ref{thm:main:FO}: 
\begin{equation}\label{eq:delta:FO:covariate}
\underline{\delta}_{N,x}:=\max\left\{\frac{{\displaystyle\max_{\tilde \alpha\in\mathcal{A}_x(\hat\alpha), \|X_i-x\|\leq h}}\|\tilde\alpha(X_i)-\alpha_0(X_i)\|^2s}{\xi\kappa_{x,\bar c}^2 {\displaystyle\min_{j\in[p], k\in[d]}}\{w_{xjk}^2,w_{xj}^2\}},\frac{\displaystyle\max_{\tilde \alpha\in\mathcal{A}_x(\hat\alpha),\|X_i-x\|\leq h}\|\tilde\alpha(X_i)-\alpha_0(X_i)\|}{\displaystyle\xi\min_{j\in[p], k\in[d]}\{w_{xjk},w_{xj}\}}\right\}.
\end{equation}
\begin{equation}\label{bound:R:x}
\begin{array}{rl}
\overline{R}_x&\displaystyle:=\max_{\substack{\|X_i-x\|\leq h\\
\alpha \in \mathcal{A}_x(\hat\alpha)}}\|\alpha(X_i)-\alpha_0(X_i)\|_1 \sqrt{\frac{\log(M/\delta)}{Nh^d}}+ sdh^2\sqrt{\frac{\log(1/\delta)}{Nh^d}}+s^2d^2h^4\\
&\displaystyle\quad+\max_{\substack{\|X_i-x\|\leq h\\
\alpha \in \mathcal{A}_x(\hat\alpha)}} \|\alpha(X_i)-\alpha_0(X_i)\|^2+sdh^2\left(\frac{Md}{\sqrt{Nh^{d}}}+Mdh^2\right).
\end{array}
\end{equation}
Similar to before, whenever a result is stated under Assumption~\ref{ass:smoothness-assumptions:sufficient} for the with-covariate case, we directly invoke the corresponding technical conditions from Assumption~\ref{ass:smoothness-assumptions} (i) -- (vi). 

\subsection{Proof of Theorem \ref{thm:main:covariate}}
It is sufficient to verify each condition of Lemma \ref{lemma:main:FO-covariate} to prove Theorem~\ref{thm:main:covariate}. 
Firstly, the conditions (i)-(iv) hold as follows:
\begin{enumerate}
    \item Condition (i): By Lemma \ref{lemma:MinMax:Claim1:covariate}, with probability $1-\delta$ where $\delta\geq\max\{\exp(-\bar{C}_0\{\sqrt{\log M}\sqrt{Nh^d}+Nh^d\}),\log(M)/(\delta_NNh^d)\}$ for some universal constant $C$,
$$\begin{array}{rl}
&\quad\displaystyle \min_{\alpha \in \mathcal{A}_x(\hat\alpha)} Q^{\mathrm{FO}}(\alpha,r_{*x})-Q^{\mathrm{FO}}(\alpha_{*x},r_{*x})\\
& \displaystyle 
 \geq-\max_{\alpha \in \mathcal{A}_x(\hat\alpha), \|X_i-x\|\leq h}\|\alpha(X_i)-\alpha_0(X_i)\|_1 C\sqrt{\frac{\log (M/\delta)}{Nh^d}}\\
&\quad \displaystyle {-} C\max_{\|X_i-x\|\leq h,\alpha \in \mathcal{A}_x(\hat\alpha)}\!\!\|\alpha(X_i)\! - \!\alpha_0(X_i)\|^2 \!\!-Csdh^2\left(\frac{Md}{\sqrt{Nh^{d}}}+Mdh^2\right).
\end{array}$$
Thus with 
 $\displaystyle\mathcal{R}_x(\hat{\alpha})=\max_{\|X_i-x\|\leq h,\alpha \in \mathcal{A}_x(\hat\alpha)}\|\alpha(X_i)-\alpha_0(X_i)\|_1 C\sqrt{\frac{\log(M/\delta)}{Nh^d}}+C\max_{\|X_i-x\|\leq h,\alpha \in \mathcal{A}_x(\hat\alpha)}\|\alpha(X_i) - \alpha_0(X_i)\|^2+Csdh^2\left(\frac{Md}{\sqrt{Nh^{d}}}+Mdh^2\right)$, the first term in (i) of Lemma \ref{lemma:main:FO-covariate} holds. Denote $\widetilde{W}_i:=\frac{W(\alpha_0(X_i),Y_i)}{1+W(\alpha_0(X_i),Y_i)'r_0(X_i)}$ and $T_i:=\mathrm{supp}(r_0(X_i)-r_{*x}(X_i))$, where $r_{*x}(X_i):=r_0(x)+\nabla r_0(x)(X_i-x)$. By Assumption \ref{ass:iid-min-den} $|T_{i}|\leq 2s$. Let $r_{0j}(\cdot)$ denote the $j$-th element of $r_0(\cdot)$ and $r_{*xj}(\cdot)$ denote the $j$-th element of $r_{*x}(\cdot)$. Since $r_{*xj}(X_i)-r_{0j}(X_i)=r_{0j}(x)+\nabla_x r_{0j}(x)(X_i-x) - r_{0j}(X_i)$, we have 
\begin{equation}\label{eq:r_0 estimate:ell_1 error}
\left|r_{*xj}(X_i)-r_{0j}(X_i)\right|\leq\bar{K}_{x,2}dh^2\ \mbox{for any $j\in T_i$},
\end{equation}
where we use the definition of $\bar{K}_{x,2}$ in the (component wise) Taylor expansion, and $\bar{K}_{x,2}$ is independent of $s,h,N,M,d$. Then with probability $1-\delta_{\alpha}$, $\frac{1}{N}\!\!\sum_{i=1}^N\!K_h(X_i-x)\!\left\{\!\frac{W(\alpha_0(X_i),Y_i)'(r_0(X_i)-r_{*x}(X_i))}{1+ W(\alpha_0(X_i),Y_i)'r_0(X_i)}\right\}^2\leq_{(1)}C's^2\{\max_{i\in[N], j\in T_i}(r_{0j}(X_i)-r_{*xj}(X_i))^2\}\leq_{(2)}Cs^2d^2h^4$, where (1) holds with probability $1-\delta_{\alpha}$ by Assumption \ref{ass:smoothness-assumptions} and sparsity condition in Assumption \ref{ass:iid-min-den}, (2) holds by (\ref{eq:r_0 estimate:ell_1 error}), and $C,C'$ are universal constants. Thus the second term in (i) of Lemma \ref{lemma:main:FO-covariate} holds with $\mathcal{G}(x)=Cs^2d^2h^4$. 
\item Condition (ii): By Lemma \ref{lemma:choice of lambda:covariate}, the choice of $\lambda_x$ in Theorem \ref{thm:main:covariate} makes (ii) of Lemma \ref{lemma:main:FO-covariate} holds. 
\item Condition (iii) of Lemma \ref{lemma:main:FO-covariate} follows immediately from Lemma \ref{lemma:covariate:score*approx}, which holds with probability $1-\delta-\frac{C(\bar{K}_{\max}/\bar{K}_2)^2}{Nh^d}$ for some universal constant $C$. 
\item Condition (iv): note that by Taylor expansion, for each $j\in [p], i\in[n]$ there is $\tilde \alpha^{ji}$ in the line segment between $\hat \alpha(X_i)$ and $\alpha_0(X_i)$ such that
$$ \begin{array}{rl}
&\quad W^{\mathrm{FO}}_j(\alpha_{*x}(X_i),Y_i) = W_j(\hat \alpha(X_i),Y_i) + (\alpha_{*x}(X_i)-\hat\alpha(X_i))'\nabla_\alpha W_j(\hat\alpha(X_i),Y_i)\\ 
& = W_j(\alpha_{0}(X_i),Y_i)+(\hat\alpha(X_i)-\alpha_{0}(X_i))'\nabla_\alpha W_j(\alpha_{0}(X_i),Y_i)+(\alpha_{*x}(X_i)-\hat\alpha(X_i))'\nabla_\alpha W_j(\hat\alpha(X_i),Y_i)\\
&\quad\displaystyle+\frac{1}{2}(\hat\alpha(X_i)-\alpha_{0}(X_i))'\nabla_\alpha^2W_j(\tilde \alpha^{ji},Y_i)(\hat\alpha(X_i)-\alpha_{0}(X_i))\\
& = W_j(\alpha_{0}(X_i),Y_i)+(\alpha_{*x}(X_i)-\alpha_{0}(X_i))'\nabla_\alpha W_j(\alpha_{0}(X_i),Y_i) + R_{ij}^{num}(x,\hat\alpha),
\end{array}$$
where $R_{ij}^{num}(x,\hat\alpha)$ above contains second order terms. Following similar proof steps as above, for $R_i^{dem}(x,\hat \alpha)$ containing second order terms, we have  $\frac{1+W(\alpha_0(X_i),Y_i)'r_0(X_i)}{1+W^{\mathrm{FO}}(\alpha_{*x}(X_i),Y_i)'r_0(X_i)}=1 - \frac{(\alpha_{*x}(X_i)-\alpha_0(X_i))'\nabla_\alpha W(\alpha_0(X_i),Y_i)'r_0(X_i)+ R_i^{dem}(x,\hat \alpha)}{1+W^{\mathrm{FO}}(\alpha_{*x}(X_i),Y_i)'r_0(X_i)}$. Particularly, Taylor expansion implies that for some $\tilde \alpha(X_i)$, $\bar \alpha(X_i)$, 
$$ \begin{array}{rl}
W(\hat \alpha(X_i),Y_i)'r_0(X_i)&= W(\alpha_0(X_i),Y_i)'r_0(X_i) + (\hat \alpha(X_i)-\alpha_0(X_i))'\nabla_\alpha W(\alpha_0(X_i),Y_i)'r_0(X_i)\\
&\quad\displaystyle + \frac{1}{2}(\hat \alpha(X_i)-\alpha_0(X_i))'\nabla_\alpha^2 \{W(\tilde \alpha(X_i),Y_i)'r_0(X_i)\}(\hat \alpha(X_i)-\alpha_0(X_i)),
\end{array}$$
$$\begin{array}{rl}
&\quad(\alpha_{*x}(X_i)-\hat\alpha(X_i))'\nabla_\alpha \{W(\hat\alpha(X_i),Y_i)'r_0(X_i)\}\\  
&= (\alpha_{*x}(X_i)-\hat\alpha(X_i))'\left[\nabla_\alpha \{W(\alpha_0(X_i),Y_i)'r_0(X_i)\}+ \nabla_\alpha^2 \{W(\bar{\alpha}(X_i),Y_i)'r_0(X_i)\}(\hat\alpha(X_i)-\alpha_0(X_i))\right].
\end{array}$$
Combining all these arguments above, we have $\frac{W^{\mathrm{FO}}_j(\alpha_{*x}(X_i),Y_i) }{1+W^{\mathrm{FO}}(\alpha_{*x}(X_i),Y_i)'r_0(X_i)} =  \frac{W_j(\alpha_0(X_i),Y_i)}{1+W(\alpha_0(X_i),Y_i)'r_0(X_i)} + \widetilde R_{ij}(x,\hat\alpha)$, $|\widetilde R_{ij}(x,\hat\alpha)|\leq C'(\|\alpha_{*x}(X_i)-\alpha_0(X_i)\|+\|\hat{\alpha}(X_i)-\alpha_0(X_i)\|^2)$ for some universal constant $C'$. Thus (iv) of Lemma \ref{lemma:main:FO-covariate} holds with $\widetilde{R}_i=C'(\|\alpha_{*x}(X_i)-\alpha_0(X_i)\|+\|\hat{\alpha}(X_i)-\alpha_0(X_i)\|^2)$. 
\end{enumerate}
Hence from the argument above and statement of Theorem \ref{thm:main:covariate}, conditions (i)-(iv) of Lemma \ref{lemma:main:FO-covariate} hold with probability $1-\delta$. Next, the remaining conditions of Lemma \ref{lemma:main:FO-covariate} hold with probability $1-\delta_{\alpha}$ by Assumption \ref{ass:smoothness-assumptions}. Thus Theorem \ref{thm:main:covariate} holds. 
 \proofend
 
\subsection{Proof of Corollary \ref{cor:FO:covariate}}
Recall  $\displaystyle\mathcal{A}_x(\hat\alpha) =\prod_{j=1}^M \left\{ \alpha_j+\beta_j(X-x) : \begin{array}{l}\alpha_j \in [\hat \alpha_j(x) - cv_x \hat s_{j,0}(x), \hat\alpha_j(x) + cv_x \hat s_{j,0}(x)], \\ 
\beta_{jk} \in [\tilde \beta_{jk}(x) - cv_x \hat s_{j,k}(x), \tilde \beta_{jk}(x) + cv_x \hat s_{j,k}(x)], k \in [d] 
\end{array}\right\}$,
$cv_x := (1-\delta_\alpha)\mbox{-quantile of} \max_{ j \in  [M], \ell \in\{0\}\cup[d]}\left|\sum_{i=1}^N \xi_i \frac{\hat \epsilon_{ij} \hat{\zeta}_{i\ell}}{\hat s_{j,\ell}(x)}\right|$
$\mbox{conditional on }\{Y_i,X_i\}_{i\in[N]}$ where $\{\xi_i\}_{i\in[N]}$ are i.i.d. $\mathcal{N}(0,1)$ random variables. Here $\hat\epsilon_{ij}:=Y_{ij} - \hat{\alpha}_j(x)-\tilde\beta_j(x)(X_i-x)$, and $\hat{\zeta}_{i} := (Z'W_xZ)^{-1}Z_i K_h(X_i-x)$  with $Z_i=(1, (X_i-x)')'$ and $W_{x}$ is a diagonal $N\times N$ matrix with $W_{x,ii}:=K_h(X_i-x)$, $\hat s_{j,\ell}(x) := \sqrt{\sum_{i=1}^N \hat \epsilon_{ij}^2 \hat{\zeta}_{i\ell}^2}$, where $\hat{\zeta}_{i\ell}$ is the $\ell$-th entry of $\hat{\zeta}_i$, $\forall \ell\in\{0\}\cup[d]$.  Note that by triangular inequality we have 
$$\begin{array}{rl}
&\quad\displaystyle\mathbb{P}\left(\max_{\alpha \in \mathcal{A}_x(\hat\alpha), \|X_i-x\|\leq h}\norm{\alpha(X_i)-\alpha_0(X_i)}\geq t\right)\\
&\displaystyle\leq\mathbb{P}\left(\max_{\alpha \in \mathcal{A}_x(\hat\alpha), \|X_i-x\|\leq h}\norm{\hat{\alpha}(x)+\tilde{\beta}(x)(X_i-x)-\alpha_0(X_i)}\geq\frac{t}{2}\right)\\
&\displaystyle\ +\mathbb{P}\left(cv_{x}\sqrt{\sum_{j\in[M]}\left(\sum_{\ell\in[d]}\hat{s}_{j,\ell}(x)|X_{il}-x_{l}|+\hat{s}_{j,0}(x)\right)^2}\geq\frac{t}{2}\right)\\
&\displaystyle\leq_{(a)}\mathbb{P}\left(\max_{\alpha \in \mathcal{A}_x(\hat\alpha), \|X_i-x\|\leq h}\norm{\hat{\alpha}(x)+\tilde{\beta}(x)(X_i-x)-\alpha_0(X_i)}\geq\frac{t}{2}\right)\\
&\displaystyle\ +\mathbb{P}\left(cv_{x}\sqrt{2dh^2\sum_{j=1}^M\max_{\ell\in[d]}\hat{s}_{j,\ell}(x)^2+2\sum_{j=1}^M\hat{s}_{j,0}(x)^2}\geq\frac{t}{2}\right),
\end{array}$$
where the first inequality follows from the triangular inequality, (a) above follows because $\sum_{\ell\in[d]}(X_{i\ell}-x_{\ell})^2\leq h^2$, $\sum_{j\in[M]}\bigg(\sum_{\ell\in[d]}\hat{s}_{j,\ell}(x)(X_{il}-x_{l})+\hat{s}_{j,0}(x)\bigg)^2\leq 2dh^2\sum_{j=1}^M\max_{\ell\in[d]}\hat{s}_{j,\ell}(x)^2+2\sum_{j=1}^M\hat{s}_{j,0}(x)^2$. Note that 
$$\begin{array}{rl}
&\displaystyle\quad\mathbb{P}\left(\max_{\alpha \in \mathcal{A}_x(\hat\alpha), \|X_i-x\|\leq h}\norm{\hat{\alpha}(x)+\tilde{\beta}(x)(X_i-x)-\alpha_0(X_i)}\geq\frac{t}{2}\right)\\
&\displaystyle\leq_{(1)}\mathbb{P}\left(\max_{\alpha \in \mathcal{A}_x(\hat\alpha), \|X_i-x\|\leq h}\norm{(\hat{\alpha}(x)-\alpha_0(x))+(\tilde{\beta}(x)-\nabla\alpha_0(x))(X_i-x)}+\sqrt{M}\bar{K}_{x,2}\|X_i-x\|^2\geq\frac{t}{2}\right),
\end{array}$$
where (1) holds by Taylor expansion and Assumption~\ref{ass:smoothness-assumptions}. Furthermore, Proposition \ref{prop:concentration:covariate} implies that both $\norm{\hat{\alpha}(x)-\alpha_0(x)}_2\leq\tilde{C}_0(x)\left(\sqrt{\frac{M}{Nh^d}}+\sqrt{M}dh^2+\sqrt{\left(\frac{1}{Nh^d}+h^4d^2\right)\log(1/\delta)}\right)$ and $\norm{\tilde{\beta}(x)-\nabla\alpha_0(x)}_{2,2}\leq\displaystyle\tilde{C}_1(x)\sqrt{d}\bigg(\sqrt{\frac{M}{Nh^{d+2}}}+\sqrt{M}dh^2+\sqrt{\left(\frac{1}{Nh^{d+2}}+h^4d^4\right)\log(1/\delta)}\bigg)$ hold with probability $1-\delta$ for some absolute constants $\tilde{C}_0(x)$ and $\tilde{C}_1(x)$. Thus for $X_i$ such that $\|X_i-x\|\leq h$, with probability $1-\delta$, for some absolute constant $\tilde{C}(x)$, 
$$\begin{array}{rl}
&\displaystyle\quad\norm{(\hat{\alpha}(x)-\alpha_0(x))+(\tilde{\beta}(x)-\nabla\alpha_0(x))(X_i-x)}\\
&\displaystyle\leq\tilde{C}(x)\left[\sqrt{\frac{M}{Nh^d}}+\sqrt{M}dh^2+\sqrt{\left(\frac{1}{Nh^d}+h^4d^2\right)\log(1/\delta)}\right]\\
&\displaystyle\quad+\tilde{C}(x)\left[h\sqrt{d}\left(\sqrt{\frac{M}{Nh^{d+2}}}+\sqrt{M}dh^2+\sqrt{\left(\frac{1}{Nh^{d+2}}+h^4d^4\right)\log(1/\delta)}\right)\right].
\end{array}$$
Additionally, $\sum_{i=1}^N\frac{\hat \epsilon_{ij}^2\hat{\zeta}_{i\ell}^2}{\hat s_{j,\ell}(x)^2}=1$ for any $j\in[M], \ell\in[d]\cup\{0\}$, so 
conditional on $\{Y_i,X_i\}$, $\sum_{i=1}^N \xi_i \frac{\hat \epsilon_{ij} \hat{\zeta}_{i\ell}}{\hat s_{j,\ell}(x)}\sim \mathcal{N}\left(0,1\right)$, thus 
$\mathbb{E}\left[\max_{ j \in  [M], \ell \in\{0\}\cup[d]}\sum_{i=1}^N \xi_i \frac{\hat \epsilon_{ij} \hat{\zeta}_{i\ell}}{\hat s_{j,\ell}(x)}\right]\leq\sqrt{2\log\{M(d+1)\}}$.
By Lemma~\ref{lemma:borell} (Borell-TIS inequality \citep{borell1975brunn}), with probability $1-\delta$, we have 
\begin{equation}\label{eq:cov:concentration:normal}
    \max_{ j \in  [M], \ell \in [d]\cup\{0\}} \left|\sum_{i=1}^N\frac{ \xi_i \hat \epsilon_{ij} \hat{\zeta}_{i\ell}}{\hat s_{j,\ell}(x)}\right|\leq2\sqrt{\log\left(\frac{M(d+1)}{\delta}\right)}.
\end{equation}
Let $\delta\leq\delta_{\alpha}$ and recall that $\delta\geq\exp(-C\{\sqrt{\log M}\sqrt{Nh^d}+Nh^d\})$, and set $t=C\sqrt{\frac{Md\log\{M(d+1)/\delta\}}{Nh^d}}+C\sqrt{M}dh^2+C\sqrt{\frac{\log(1/\delta)}{Nh^d}}$ for some absolute constant $C$ in the above, then the above results imply that $\max_{\alpha \in \mathcal{A}_x(\hat\alpha), \|X_i-x\|\leq h}\norm{\alpha(X_i)-\alpha_0(X_i)}\leq C\sqrt{Md}\left[\sqrt{\frac{\log\{M(d+1)/\delta\}}{Nh^d}}+\sqrt{d}h^2\right]+C\sqrt{\frac{\log(1/\delta)}{Nh^d}}$ holds with probability $1-2\delta$. Additionally, by triangular inequality, for any $t'>0$ we have 
$$\begin{array}{rl}
&\displaystyle\quad\mathbb{P}\left(\max_{\alpha \in \mathcal{A}_x(\hat\alpha), \|X_i-x\|\leq h}\norm{\alpha(X_i)-\alpha_0(X_i)}_1\geq t'\right)\\
&\displaystyle\leq\mathbb{P}\left(\max_{\alpha \in \mathcal{A}_x(\hat\alpha), \|X_i-x\|\leq h}\norm{\hat{\alpha}(x)+\tilde{\beta}(x)(X_i-x)-\alpha_0(X_i)}_1\!\!\geq\!\frac{t'}{2}\right)\\
&\displaystyle\quad+\mathbb{P}\left(\!cv_{x}\!\!\left[\sum_{j\in[M],\ell\in[d]}\!\!\!\!\!\!\!\!\!\!\!|\hat{s}_{j,\ell}(x)(X_{i\ell}-x_{\ell})|+\!\!\!\!\!\sum_{j\in[M]}\!\!\!\!|\hat{s}_{j,0}(x)|\right]\geq\frac{t'}{2}\right)\\
&\displaystyle\leq_{(b)}\mathbb{P}\left(\max_{\alpha \in \mathcal{A}_x(\hat\alpha), \|X_i-x\|\leq h}\norm{\hat{\alpha}(x)+\tilde{\beta}(x)(X_i-x)-\alpha_0(X_i)}_1\geq\frac{t'}{2}\right)\\
&\displaystyle\quad\quad+\mathbb{P}\left(cv_{x}\left[\sqrt{Md}h\sqrt{\sum_{j=1}^M\max_{\ell\in[d]}\hat{s}_{j,\ell}^2(x)}+\sqrt{M}\sqrt{\sum_{j\in[M]}\hat{s}_{j,0}^2(x)}\right]\geq\frac{t'}{2}\right),
\end{array}$$
where (b) follows since $$\sum_{j\in[M],\ell\in[d]}|\hat{s}_{j,\ell}(x)(X_{il}-x_{l})|+\sum_{j\in[M]}|\hat{s}_{j,0}(x)|\leq\sqrt{Md}h\sqrt{\sum_{j=1}^M\max_{\ell\in[d]}\hat{s}_{j,\ell}^2(x)}+\sqrt{M}\sqrt{\sum_{j\in[M]}\hat{s}_{j,0}^2(x)}.$$ 
Note that  
$$\begin{array}{rl}
&\displaystyle\quad\mathbb{P}\left(\max_{\alpha \in \mathcal{A}_x(\hat\alpha), \|X_i-x\|\leq h}\norm{\hat{\alpha}(x)+\tilde{\beta}(x)(X_i-x)-\alpha_0(X_i)}_1\geq\frac{t'}{2}\right)\\
&\displaystyle\leq_{(1)}\mathbb{P}\left(\max_{\alpha \in \mathcal{A}_x(\hat\alpha), \|X_i-x\|\leq h}\norm{(\hat{\alpha}(x)-\alpha_0(x))+(\tilde{\beta}(x)-\nabla\alpha_0(x))(X_i-x)}_1+M\bar{K}_{x,2}h^2\geq\frac{t'}{2}\right),
\end{array}$$
where (1) holds by Taylor expansion. By Proposition \ref{prop:concentration:covariate}, with probability at least $1-\delta$, $$\norm{\hat{\alpha}(x)-\alpha_0(x)}_1\leq\tilde{C}_0(x)\left(\frac{M}{\sqrt{Nh^d}}+Mdh^2+\sqrt{M\left(\frac{1}{Nh^d}+h^4d^2\right)\log(1/\delta)}\right),$$ 
and $\norm{\tilde{\beta}(x)-\nabla\alpha_0(x)}_{1,1}\leq\tilde{C}_1(x)d\left(\frac{M}{\sqrt{Nh^{d+2}}}+Md^2h^2+\sqrt{M\left(\frac{1}{Nh^{d+2}}+h^4d^4\right)\log(1/\delta)}\right)$.
Henceforth, for some absolute constant $\tilde{C}'(x)$, for $X_i$ such that $\|X_i-x\|\leq h$, with probability at least $1-\delta$, it holds that 
$$\begin{array}{rl}
&\displaystyle\quad\norm{(\hat{\alpha}(x)-\alpha_0(x))+(\tilde{\beta}(x)-\nabla\alpha_0(x))(X_i-x)}_1\\
&\displaystyle\leq\tilde{C}'(x)\left[\frac{M}{\sqrt{Nh^d}}+Mdh^2+\sqrt{M\left(\frac{1}{Nh^d}+h^4d^2\right)\log(1/\delta)}\right]\\
&\displaystyle\quad+\tilde{C}'(x)\left[hd\left(\frac{M}{\sqrt{Nh^{d+2}}}+Md^2h^2+\sqrt{M\left(\frac{1}{Nh^{d+2}}+h^4d^4\right)\log(1/\delta)}\right)\right].
\end{array}$$
Let $\delta\leq\delta_{\alpha}$, set $t'=C'\frac{Md}{\sqrt{Nh^d}}+C'Mdh^2+C'd\sqrt{\frac{M\log(1/\delta)}{Nh^d}}$ in the above, where $C'$ is an absolute constant. Note that Assumption~\ref{ass:dimension} implies $dh^2<1$, then the previous results imply that with probability $1-2\delta$, $\max_{\alpha \in \mathcal{A}_x(\hat\alpha), \|X_i-x\|\leq h}\norm{\alpha(X_i)-\alpha_0(X_i)}_1\leq C'M\left(\frac{d}{\sqrt{Nh^d}}+dh^2\right)+C'd\sqrt{\frac{M\log(1/\delta)}{Nh^d}}$. Note that $\lambda_x=C\sqrt{\log(p/\delta)/Nh^d}+Ch^2$ under the given condition, Theorem~\ref{thm:main:covariate} implies that with probability $1-C\delta-C\delta_\alpha-\frac{C(\bar{K}_{\max}/\bar{K}_2)^2}{Nh^d}-C\epsilon$,
$$\begin{array}{rl}
\widetilde{\mathcal{R}}(\hat\alpha) &\displaystyle\leq C_1\max_{\alpha \in \mathcal{A}_x(\hat\alpha), \|X_i-x\|\leq h}\|\alpha(X_i)-\alpha_0(X_i)\|_1 \sqrt{\frac{\log(M/\delta)}{Nh^d}}+C_1sdh^2\sqrt{\frac{\log(1/\delta)}{Nh^d}}+C_1s^2d^2h^4\\
&\displaystyle\quad + C_1\!\!\!\!\max_{\alpha \in \mathcal{A}_x(\hat\alpha), \|X_i-x\|\leq h} \|\alpha(X_i)-\alpha_0(X_i)\|^2+C_1sdh^2\left(\frac{Md}{\sqrt{Nh^{d}}}+Mdh^2\right)\leq_{(1)}C_2(M+s^2)d^2h^4,
\end{array}$$
where $C,C_1,C_2$ are constants independent of $s,N,M,h,d$ above. Here (1) above uses Assumption~\ref{ass:dimension}, and the high-probability bounds derived for $\max_{\alpha \in \mathcal{A}_x(\hat\alpha), \|X_i-x\|\leq h}\|\alpha(X_i)-\alpha_0(X_i)\|_1$ and $\max_{\alpha \in \mathcal{A}_x(\hat\alpha), \|X_i-x\|\leq h} \|\alpha(X_i)-\alpha_0(X_i)\|^2$. Thus, Theorem~\ref{thm:main:covariate} implies $\sqrt{\frac{1}{N}\sum_{i=1}^N K_h(X_i-x)\left(\Delta_{x,i}^{\mathrm{FO}}\right)^2}\leq C\sqrt{\frac{sM+s\log(1/\delta)}{Nh^d}}+C(\sqrt{M}+s)dh^2$. Further, Assumption \ref{ass:dimension} requires $h$ to satisfy 
$\left(\frac{M^3}{N\delta_N^2}\right)^{1/d}\leq h\leq\sqrt{\frac{\delta_N\xi}{sd}}$. By first order condition, we have $h^*=\min\left\{\sqrt{\frac{\delta_N\xi}{sd}},\max\left\{\left(\frac{M^3}{N\delta_N^2}\right)^{1/d},4^{-\frac{2}{d+4}}\left(\frac{sM+s\log(1/\delta)}{N(s+\sqrt{M})^2}\right)^{\frac{1}{d+4}}\right\}\right\}$. Particularly, when $h=h^*$, by direct calculation, we have $\sqrt{\frac{sM+s\log(1/\delta)}{Nh^d}}+(\sqrt{M}+s)dh^2\leq\bar{C}\min\left\{B_{1},B_2\right\}$ for some universal constant $\bar{C}$, where $B_1:=\delta_N\sqrt{\frac{s(M+\log(1/\delta))}{M^3}}+(\sqrt{M}+s)d\left(\frac{M^3}{N\delta_N^2}\right)^{\frac{2}{d}}$,
$B_2:=\sqrt{\frac{s(M+\log(1/\delta))}{N}}\left(\frac{sd}{\delta_N\xi}\right)^{d/4}\!\!+\frac{(\sqrt{M}+s)\delta_N\xi}{s}$.
\proofend

\subsection{Technical Lemmas for Section \ref{Sec:Rates:covariates}}
\begin{lemma}\label{lemma:main:FO-covariate}
For any $x\in \mathcal{X}$, define $\alpha_{*x}(X):= \alpha_0(x)+\nabla_x\alpha_0(x)(X-x)$ and define $r_{*x}(X):= r_0(x)+\nabla_xr_0(x)(X-x)$. Moreover for any $\bar \alpha \in [0,1]^M, \bar y \in \{0,1\}^M$, let $W^{\mathrm{FO}}_i(\bar \alpha,\bar{y})=W(\hat \alpha(X_i),\bar y)+\nabla_\alpha W(\hat\alpha(X_i),\bar y)(\bar \alpha-\hat\alpha(X_i))$. Suppose (i) $\min_{\alpha\in\mathcal{A}_x(\hat \alpha)}Q^{\mathrm{FO}}(\alpha,r_{*x})-Q^{\mathrm{FO}}(\alpha_{*x},r_{*x}) \geq -\mathcal{R}_x(\hat{\alpha})$, and $\displaystyle\frac{1}{N}\sum_{i=1}^NK_h(X_i-x) \left\{\frac{W(\alpha_0(X_i),Y_i)'(r_0(X_i)-r_{*x}(X_i))}{1+ W(\alpha_0(X_i),Y_i)'r_0(X_i)}\right\}^2 \leq \mathcal{G}(x)$,\\
(ii)$\displaystyle\frac{\lambda_x}{c} \geq \left\|  \mathrm{diag}^{-1}(w_x)\frac{1}{N}\sum_{i=1}^N\frac{  K_h(X_i-x) }{1+W^{\mathrm{FO}}(\alpha_{*x}(X_i),Y_i)^\prime r_0(X_i)} \begin{pmatrix} W^{\mathrm{FO}}(\alpha_{*x}(X_i),Y_i)\\{\rm vec}( W^{\mathrm{FO}}(\alpha_{*x}(X_i),Y_i)\frac{(X_i-x)'}{h} )
\end{pmatrix}\right\|_\infty $\\
(iii) $\displaystyle\left|\frac{1}{N}\sum_{i=1}^N K_h(X_i-x)\frac{ W(\alpha_0(X_i),Y_i)^\prime(r_{*x}(X_i)-r_0(X_i))}{1+W(\alpha_0(X_i),Y_i)^\prime r_0(X_i)}\right|\leq Csdh^2\sqrt{\frac{\log(1/\delta)}{Nh^d}}$\\ 
(iv) $\displaystyle\left\|\frac{ W^{\mathrm{FO}}(\alpha_{*x}(X_i),Y_i)}{1+W^{\mathrm{FO}}(\alpha_{*x}(X_i),Y_i)^\prime r_0(X_i)}-\frac{ W(\alpha_0(X_i),Y_i)}{1+W(\alpha_0(X_i),Y_i)^\prime r_0(X_i)}\right\|_\infty \leq \widetilde{R}_{i}$. (v) $\alpha_{*x} \in \mathcal{A}_x(\hat\alpha)$. \\
(vi) $\displaystyle \max_{\|X_i-x\|\leq h, j\in [p], \tilde \alpha \in \mathcal{A}_x(\hat\alpha)}|v'\nabla_\alpha W_j(\tilde \alpha(X_i),Y_i)|  +|v'\nabla_\alpha W(\tilde \alpha(X_i),Y_i)'r_0(X_i)| \leq M^{1/2}K_{\alpha,1}\|v\|$, \\
(vii) $\displaystyle\max_{\|X_i-x\|\leq h, j \in [p], \tilde \alpha \in \mathcal{A}_x(\hat\alpha)}|u'\nabla_\alpha^2 W(\tilde\alpha(X_i),Y_i)'r_0(X_i)v|+|u'\nabla_\alpha^2 W_j(\tilde\alpha(X_i),Y_i)v|\leq K_{\alpha,2}\|u\|\|v\|$, and\\
${\displaystyle\max_{\tilde{x}\in\mathcal{X},i\in[N]}}|u'\nabla_{\alpha}W(\alpha_0(\tilde{x}),Y_i)'r_0(X_i)|\leq\bar{K}_w\|u\|_1$, ${\displaystyle\max_{\tilde{x}\in\mathcal{X},i\in[N]}}|u'\nabla_{\alpha}W(\alpha_0(\tilde{x}),Y_i)'v|\leq \bar{K}_w\norm{u}_{\infty}\norm{v}_{1}$. (viii) ${\displaystyle\max_{\|\tilde{x}-x\|\leq h, \alpha \in \mathcal{A}_x(\hat\alpha)}}\|\alpha(x)-\alpha_0(x)\|M^{1/2}K_{\alpha,1}/\xi < 1/4$, ${\displaystyle\max_{\|\tilde{x}-x\|\leq h, \tilde\alpha \in \mathcal{A}_x(\hat\alpha)}}\|\tilde \alpha(x)-\alpha_0(x)\|^2 K_{\alpha,2}/\xi \leq 1/16$. 
(ix) ${\displaystyle\max_{j\in [p], k\in[M], x\in \mathcal{X}}}|v'\nabla_{x}^2 r_{0j}(x)v|+ |v'\nabla^2\alpha_{0k}(x)v|\leq \bar{K}_{x,2}\|v\|^2$. \\
(x) $\delta_N\geq s{\displaystyle\max_{\alpha \in \mathcal{A}_x(\hat\alpha), \|X_i-x\|\leq h}}\|\tilde\alpha(X_i)-\alpha_0(X_i)\|^2/ \{\xi\kappa_{x,\bar c}^2{\displaystyle\min_{j\in[p], k\in[d]}}\{w_{xjk}^2,w_{xj}^2\}\}$. \\
$\delta_N\geq{\displaystyle\max_{\alpha \in \mathcal{A}_x(\hat\alpha), \|X_i-x\|\leq h}}\|\tilde\alpha(X_i)-\alpha_0(X_i)\|/ \{\xi{\displaystyle\min_{j\in[p], k\in[d]}}\{w_{xjk},w_{xj}\}\}$. 

Then for all $x \in \mathcal{X}$, for some universal constant $C$, with probability $1-C\epsilon$, 
$$\sqrt{\frac{1}{N} \sum_{i=1}^NK_h(X_i-x)\left\{\frac{W(\alpha_0(X_i),Y_i)^\prime(\hat{r}^{\mathrm{FO}}(X_i,x)-r_{*x}(X_i))}{1+W(\alpha_0(X_i),Y_i)^\prime r_0(X_i)}\right\}^2}\leq C\left(\lambda_x \left( 1 + \frac{1}{c}\right) \frac{\sqrt{s}}{\kappa_{x,\bar{c}}} + \widetilde{\mathcal{R}}(x)^{1/2}\right),$$
where $\displaystyle\widetilde{\mathcal{R}}(x)\leq \mathcal{R}_x(\hat\alpha) + \mathcal{G}(x) + Csdh^2\left(\!\!\sqrt{\frac{\log(1/\delta)}{Nh^d}}\! + \! \frac{1}{N}\sum_{i=1}^NK_h(X_i-x)|\widetilde{R}_i|\!+\!\frac{1}{N}\sum_{i=1}^NK_h(X_i-x)|\widetilde{R}_i|^2\right)$.
\end{lemma}
\proof{Proof of Lemma \ref{lemma:main:FO-covariate}}
Let $Q^{\mathrm{FO}}(\alpha,r_{a,b}):= \frac{1}{N}\sum_{i=1}^N K_h(X_i-x)\log(1+ W^{\mathrm{FO}}(\alpha(X_i),Y_i)'r_{a,b}(X_i,x))$. Following similar proof steps as in Proposition~\ref{prop:plug-in:no-covariate} and Theorem~\ref{thm:main:FO}, using Assumption~\ref{ass:iid-min-den} and the given conditions, with probability $1-C\epsilon$, we have $1 +  W^{\mathrm{FO}}(\alpha_{*x}(X_i),Y_i)'\hat r_{\hat a,\hat b}(X_i))>c_1\xi, 1 + W^{\mathrm{FO}}(\alpha_{*x}(X_i),Y_i)' r_0(X_i)>c_1\xi$ for some $c_1\in(0,1)$. Additionally, note that given any $\bar{y}>-1$, there exists a constant $c'=\frac{1}{2(1+\max\{\bar{y},0\})}$ such that $\log(1+y)\leq y-c'y^2$ holds for $y\in(-1,\bar{y}]$. Hence for some universal constant $c_0>0$ we have 
$$\begin{array}{rl}
       &\displaystyle\quad \log ( 1 +  W^{\mathrm{FO}}(\alpha_{*x}(X_i),Y_i)'\hat r_{\hat a,\hat b}(X_i))   -\log( 1 + W^{\mathrm{FO}}(\alpha_{*x}(X_i),Y_i)' r_0(X_i)) \\
       &\displaystyle=\log\left(1+\frac{ W^{\mathrm{FO}}(\alpha_{*x}(X_i),Y_i)^\prime(\hat{r}_{\hat a,\hat b}(X_i,x)-r_0(X_i))}{1+ W^{\mathrm{FO}}(\alpha_{*x}(X_i),Y_i)^\prime r_0(X_i)}\right)\\
       &\displaystyle\leq\frac{ W^{\mathrm{FO}}(\alpha_{*x}(X_i),Y_i)^\prime(\hat{r}_{\hat a,\hat b}(X_i,x)-r_0(X_i))}{1+W^{\mathrm{FO}}(\alpha_{*x}(X_i),Y_i)^\prime r_0(X_i)}-c_0\left\{\frac{W^{\mathrm{FO}}(\alpha_{*x}(X_i),Y_i)^\prime(\hat{r}_{\hat a,\hat b}(X_i,x)-r_0(X_i))}{1+W^{\mathrm{FO}}(\alpha_{*x}(X_i),Y_i)^\prime r_0(X_i)}\right\}^2,
\end{array}$$
Define 
$$\Delta_{0,ab}^{\mathrm{FO}}:=\frac{W^{\mathrm{FO}}(\alpha_{*x}(X_i),Y_i)^\prime(\hat{r}_{\hat a, \hat b}(X_i,x)-r_0(X_i))}{1+W^{\mathrm{FO}}(\alpha_{*x}(X_i),Y_i)^\prime r_0(X_i)},\qquad\mbox{then}$$ 
\begin{equation}\label{eq:covariate:quadratic:Q}
Q^{\mathrm{FO}}(\alpha_{*x},\hat r_{\hat a,\hat b})-Q^{\mathrm{FO}}(\alpha_{*x}, r_0) \leq \frac{1}{N}\sum_{i=1}^N K_h(X_i-x)\left\{\Delta_{0,ab}^{\mathrm{FO}}-c_0(\Delta_{0,ab}^{\mathrm{FO}})^2\right\}.
\end{equation}
$$\begin{array}{rl}
\mbox{Next note that}&\displaystyle\quad \frac{1}{N}\sum_{i=1}^N \frac{K_h(X_i-x)W^{\mathrm{FO}}(\alpha_{*x}(X_i),Y_i)^\prime(\hat{r}_{\hat a,\hat b}(X_i,x)-r_0(X_i))}{1+W^{\mathrm{FO}}(\alpha_{*x}(X_i),Y_i)^\prime r_0(X_i)} \\
&\displaystyle = \underbrace{\frac{1}{N}\sum_{i=1}^N\frac{K_h(X_i-x)W^{\mathrm{FO}}(\alpha_{*x}(X_i),Y_i)^\prime( \hat a - r_0(x) + \hat b(X_i-x) - \nabla_x r_0(x)(X_i-x))}{1+W^{\mathrm{FO}}(\alpha_{*x}(X_i),Y_i)^\prime r_0(X_i)}}_{\textrm{(I)}} \\
&\displaystyle\quad + \underbrace{\frac{1}{N}\sum_{i=1}^N\frac{K_h(X_i-x)W^{\mathrm{FO}}(\alpha_{*x}(X_i),Y_i)^\prime(r_{*x}(X_i)-r_0(X_i))}{1+W^{\mathrm{FO}}(\alpha_{*x}(X_i),Y_i)^\prime r_0(X_i)}}_{\textrm{(II)}} = \textrm{(I)} + \textrm{(II)}.\ \ \mbox{Note that}
\end{array}$$
$$\begin{array}{rl}
|\textrm{(I)}|
&\displaystyle\leq\left\|\mathrm{diag}^{-1}(w_x)\frac{1}{N}\sum_{i=1}^N\frac{  K_h(X_i-x) }{1+W^{\mathrm{FO}}(\alpha_{*x}(X_i),Y_i)^\prime r_0(X_i)} \left( { W^{\mathrm{FO}}(\alpha_{*x}(X_i),Y_i)}\atop{ {\rm vec}(W^{\mathrm{FO}}(\alpha_{*x}(X_i),Y_i)\frac{(X_i-x)'}{h} )}\right)\right\|_\infty\\
&\displaystyle\quad\quad\times\left\|\left( {\hat a - r_0(x)}\atop{ {\rm vec}\big(h(\hat b - \nabla_xr_0(x)})\big)\right) \right\|_{w_x,1}\leq\frac{\lambda_x}{c}\left\|\left( {\hat a - r_0(x)}\atop{ {\rm vec}\big(h(\hat b - \nabla_xr_0(x)})\big)\right) \right\|_{w_x,1},
\end{array}$$
where in the above, the first inequality holds by H\"older's inequality and the second inequality holds by condition (ii). Since $r_{*x}(X_i)-r_0(X_i)=r_0(x)+\nabla_x r_0(x)(X_i-x) - r_0(X_i)$, by Taylor expansion, $\left\|  r_{*x}(X_i)-r_0(X_i) \right\|_1=\| r_0(x)+\nabla_x r_0(x)(X_i-x) - r_0(X_i) \|_1\leq\bar{K}sdh^2$ for some universal constant $\bar{K}$ by condition (ix). Define $\Delta^{\mathrm{FO}}(X_i,Y_i):=\left\{\frac{W^{\mathrm{FO}}(\alpha_{*x}(X_i),Y_i)^\prime}{1+W^{\mathrm{FO}}(\alpha_{*x}(X_i),Y_i)^\prime r_0(X_i)}-\frac{W(\alpha_0(X_i),Y_i)^\prime}{1+W(\alpha_0(X_i),Y_i)^\prime r_0(X_i)}\right\}$.
So for some universal constant $C$,
$$\begin{array}{rl}
|\textrm{(II)}|&\displaystyle \leq_{(1)} \left|\frac{1}{N}\sum_{i=1}^N K_h(X_i-x)\frac{ W(\alpha_0(X_i),Y_i)^\prime(r_{*x}(X_i)-r_0(X_i))}{1+W(\alpha_0(X_i),Y_i)^\prime r_0(X_i)}\right|\\
&\displaystyle\quad+\bigg|\frac{1}{N}\sum_{i=1}^N K_h(X_i-x)\Delta^{\mathrm{FO}}(X_i,Y_i)(r_{*x}(X_i)-r_0(X_i)) \bigg|\\
&\displaystyle\leq_{(2)} Csdh^2\sqrt{\frac{\log(1/\delta)}{Nh^d}}+ \frac{1}{N}{\displaystyle\sum_{i=1}^N} K_h(X_i-x)\left\|\Delta^{\mathrm{FO}}(X_i,Y_i)\right\|_\infty\|r_{*x}(X_i)-r_0(X_i)\|_1 \\
&\displaystyle \leq_{(3)} Csdh^2\sqrt{\frac{\log(1/\delta)}{Nh^d}} +  \frac{Csdh^2}{N}\sum_{i=1}^N K_h(X_i-x)\widetilde{R}_i,
\end{array}$$
where (1) holds by the triangle inequality, the first term of (2) holds by condition (iii), and the second term of (2) holds by H\"older's inequality, (3) holds by condition (iv) and the bound for $\left\|  r_{*x}(X_i)-r_0(X_i) \right\|_1$ above. So the arguments above imply
\begin{equation}\label{eq:bound:I:II}
\begin{array}{rl}
    &\quad\displaystyle\frac{1}{N}\sum_{i=1}^N K_h(X_i-x)\Delta_{0,ab}^{\mathrm{FO}}\\
    &\displaystyle\leq Csdh^2\left[\sqrt{\frac{\log(1/\delta)}{Nh^d}} +  \frac{1}{N}\sum_{i=1}^N K_h(X_i-x)\widetilde{R}_i\right]+\frac{\lambda_x}{c}\left\|\left( {\hat a - r_0(x)}\atop{ {\rm vec}\big(h(\hat b - \nabla_xr_0(x)})\big)\right) \right\|_{w_x,1}.
\end{array}
\end{equation}
Let $a^*:=r_0(x), b^*:=\nabla r_0(x)$. So $r_{*x}(X_i)=r_{a^*,b^*}(X_i,x)$, and 
$$\left\|\left( {\hat a - r_0(x)}\atop{ {\rm vec}\big(h(\hat b - \nabla_xr_0(x)})\big)\right) \right\|_{w_x,1}=\big\|\big(\hat a - a^*,{\rm vec}\big(h(\hat b - b^*)\big)\big)\big\|_{w_x,1}=\|r_{\hat{a},\hat{b}}-r_{a^*,b^*}\|_{w_x,h,1}.$$
\begin{equation}\label{eq:covariate:main:steps}
\begin{array}{rl}
&\quad\displaystyle Q^{\mathrm{FO}}(\alpha_{*x},\hat r_{\hat a,\hat b}) - \lambda_x \|\hat r_{\hat a,\hat b}\|_{w_x,h,1}\\
& \geq_{(1)}\displaystyle \min_{\alpha \in \mathcal{A}_x(\hat\alpha)} Q^{\mathrm{FO}}(\alpha,\hat r_{\hat a,\hat b}) - \lambda_x \|\hat r_{\hat a,\hat b}\|_{w_x,h,1} \geq_{(2)} \min_{\alpha \in \mathcal{A}_x(\hat\alpha)} Q^{\mathrm{FO}}(\alpha,r_{*x}) - \lambda_x \|r_{a^*,b^*}\|_{w_x,h,1}\\
&\displaystyle \geq \min_{\alpha \in \mathcal{A}_x(\hat\alpha)} \{Q^{\mathrm{FO}}(\alpha,r_{*x})-Q^{\mathrm{FO}}(\alpha_{*x},r_{*x})\} + Q^{\mathrm{FO}}(\alpha_{*x},r_{*x})- \lambda_x \|r_{a^*,b^*}\|_{w_x,h,1} \\
&\displaystyle  \geq_{(3)}  - \mathcal{R}_x(\hat\alpha) + Q^{\mathrm{FO}}(\alpha_{*x},r_{*x})- \lambda_x \|r_{a^*,b^*}\|_{w_x,h,1} 
\end{array}
\end{equation}
where (1) holds from condition (v) $\alpha_{*x} \in \mathcal{A}_x(\hat\alpha)$, (2) holds by the optimality of $\hat a,\hat b$, (3) holds since $\min_{\alpha \in \mathcal{A}_x(\hat\alpha)} \{Q^{\mathrm{FO}}(\alpha,r_{*x})-Q^{\mathrm{FO}}(\alpha_{*x},r_{*x})\} \geq - \mathcal{R}_x(\hat\alpha)$ by condition (i). Hence,
$$\begin{array}{rl}
&\displaystyle\frac{c_0}{N}\sum_{i=1}^NK_h(X_i-x)(\Delta_{0,ab}^{\mathrm{FO}})^2\leq_{(1)}\frac{1}{N}\sum_{i=1}^NK_h(X_i-x)(\Delta_{0,ab}^{\mathrm{FO}})-\{Q^{\mathrm{FO}}(\alpha_{*x},\hat r_{\hat a,\hat b})-Q^{\mathrm{FO}}(\alpha_{*x}, r_0)\}\\
&\displaystyle=\frac{1}{N}\sum_{i=1}^NK_h(X_i-x)(\Delta_{0,ab}^{\mathrm{FO}})-\{Q^{\mathrm{FO}}(\alpha_{*x},\hat r_{\hat a,\hat b})-Q^{\mathrm{FO}}(\alpha_{*x}, r_{*x})+Q^{\mathrm{FO}}(\alpha_{*x},r_{*x})-Q^{\mathrm{FO}}(\alpha_{*x}, r_0)\}\\
&\displaystyle\leq_{(2)}2Csdh^2\left[\sqrt{\frac{\log(1/\delta)}{Nh^d}} +  \frac{1}{N}\sum_{i=1}^N K_h(X_i-x)\widetilde{R}_i\right]+\frac{\lambda_x}{c}\|\hat r_{\hat a, \hat b} - r_{a^*,b^*}\|_{w_x,h,1}\\
&\quad\quad\displaystyle+\lambda_x\|r_{a^*,b^*}\|_{w_x,h,1}-\lambda_x\|\hat r_{\hat a,\hat b}\|_{w_x,h,1} + \mathcal{R}_x(\hat\alpha),\ \mbox{where (1) is due to }\eqref{eq:covariate:quadratic:Q}, \mbox{(2) from \eqref{eq:bound:I:II}, \eqref{eq:covariate:main:steps}}
\end{array}$$
and following from similar proof steps for \eqref{eq:covariate:quadratic:Q}, $-(Q^{\mathrm{FO}}(\alpha_{*x},r_{*x})-Q^{\mathrm{FO}}(\alpha_{*x}, r_0))\leq|\textrm{(II)}|$. The above inequality is equivalent to $\frac{c_0}{N}\sum_{i=1}^NK_h(X_i-x)\left\{\frac{W^{\mathrm{FO}}(\alpha_{*x}(X_i),Y_i)^\prime(\hat{r}_{\hat a,\hat b}(X_i,x)-r_0(X_i))}{1+W^{\mathrm{FO}}(\alpha_{*x}(X_i),Y_i)^\prime r_0(X_i)}\right\}^2\leq\lambda_x\|r_{a^*,b^*}\|_{w_x,h,1} -  \lambda_x\|\hat r_{\hat a,\hat b}\|_{w_x,h,1} + \mathcal{R}_x(\hat\alpha) + \frac{\lambda_x}{c}\|\hat r_{\hat a, \hat b} - r_{a^*,b^*}\|_{w_x,h,1}+2Csdh^2\left(\sqrt{\frac{\log(1/\delta)}{Nh^d}}+\frac{1}{N}\sum_{i=1}^N K_h(X_i-x)\widetilde{R}_i\right)$. Define $\Delta_{*x,ab}^{\mathrm{FO}}:=\frac{W^{\mathrm{FO}}(\alpha_{*x}(X_i),Y_i)^\prime(\hat{r}_{\hat a, \hat b}(X_i,x)-r_{*x}(X_i))}{1+W^{\mathrm{FO}}(\alpha_{*x}(X_i),Y_i)^\prime r_0(X_i)}$, $\Delta_0W_i:=\frac{W(\alpha_0(X_i),Y_i)^\prime(r_0(X_i)-r_{*x}(X_i))}{1+W(\alpha_0(X_i),Y_i)^\prime r_0(X_i)}$.
Then for some universal constant $C'$,
$$\begin{array}{rl}
&\displaystyle \sqrt{\frac{c_0}{N}\sum_{i=1}^NK_h(X_i-x)\left\{\Delta_{0,ab}^{\mathrm{FO}}\right\}^2}-\sqrt{\frac{c_0}{N} \sum_{i=1}^NK_h(X_i-x)\left\{\Delta_{*x,ab}^{\mathrm{FO}}\right\}^2}\\
&\displaystyle \geq - \sqrt{\frac{c_0}{N} \sum_{i=1}^NK_h(X_i-x)\left\{\Delta^{\mathrm{FO}}(X_i,Y_i)(r_0(X_i)-r_{*x}(X_i))\right\}^2} - \sqrt{\frac{c_0}{N} \sum_{i=1}^NK_h(X_i-x)\left\{\Delta_0W_i\right\}^2}   \\
&\displaystyle \geq - \sqrt{\frac{c_0}{N} \sum_{i=1}^NK_h(X_i-x)\left\|\Delta^{\mathrm{FO}}(X_i,Y_i)\right\|_\infty^2\|r_0(X_i)-r_{*x}(X_i)\|^2_1}-\sqrt{\frac{c_0}{N} \sum_{i=1}^NK_h(X_i-x)\left\{\Delta_0W_i\right\}^2}   \\
&\displaystyle \geq -C'sdh^2\sqrt{\frac{c_0}{N} \sum_{i=1}^NK_h(X_i-x) \widetilde{R}_i^2} - \sqrt{\frac{c_0}{N} \sum_{i=1}^NK_h(X_i-x)\left\{\Delta_0W_i\right\}^2},   
\end{array}$$ 
where the last inequality follows from condition (iv) and the bound for $\left\|  r_{*x}(X_i)-r_0(X_i) \right\|_1$ above. Using the fact that that for non-negative numbers $a, b, c$ we have $a+c\geq b\Rightarrow 2a^2+2c^2\geq b^2$, 
$$\begin{array}{rl}
&\displaystyle\frac{2c_0}{N}\!\!\sum_{i=1}^N\!K_h(X_i-x)\!\!\left\{\!\frac{W^{\mathrm{FO}}(\alpha_{*x}(X_i),Y_i)^\prime(\hat{r}_{\hat a, \hat b}(X_i,x)-r_0(X_i))}{1+W^{\mathrm{FO}}(\alpha_{*x}(X_i),Y_i)^\prime r_0(X_i)}\right\}^2 \!\!\!\!\!+\!\! 4\!\left\{\!C'sdh^2\sqrt{\frac{c_0}{N} \sum_{i=1}^NK_h(X_i-x) \widetilde{R}_i^2}\right\}^2\\
&\displaystyle + \frac{4c_0}{N} \sum_{i=1}^NK_h(X_i-x)\left\{\frac{W(\alpha_0(X_i),Y_i)^\prime(r_0(X_i)-r_{*x}(X_i))}{1+W(\alpha_0(X_i),Y_i)^\prime r_0(X_i)}\right\}^2 \\
&\displaystyle \geq\frac{c_0}{N} \sum_{i=1}^NK_h(X_i-x)\left\{\frac{W^{\mathrm{FO}}(\alpha_{*x}(X_i),Y_i)^\prime(\hat{r}_{\hat a, \hat b}(X_i,x)-r_{*x}(X_i))}{1+W^{\mathrm{FO}}(\alpha_{*x}(X_i),Y_i)^\prime r_0(X_i)}\right\}^2.\ \mbox{Next, define }
\end{array}$$
$$\begin{array}{rl}
\widetilde{\mathcal{R}}(x,\hat{\alpha})&\displaystyle:= \mathcal{R}_x(\hat\alpha)  + 2Csdh^2\sqrt{\frac{\log(1/\delta)}{Nh^d}}+2(C')^2s^2d^2h^4\frac{c_0}{N} \sum_{i=1}^NK_h(X_i-x)\widetilde{R}_i^2\\
&\displaystyle \ +  \frac{2Csdh^2}{N}\sum_{i=1}^N K_h(X_i-x)\widetilde{R}_i+\frac{2c_0}{N} \sum_{i=1}^NK_h(X_i-x)\left\{\frac{W(\alpha_0(X_i),Y_i)^\prime(r_0(X_i)-r_{*x}(X_i))}{1+W(\alpha_0(X_i),Y_i)^\prime r_0(X_i)}\right\}^2.
\end{array}$$
The upper bound for $\frac{c_0}{N}\sum_{i=1}^NK_h(X_i-x)\left\{\frac{W^{\mathrm{FO}}(\alpha_{*x}(X_i),Y_i)^\prime(\hat{r}_{\hat a,\hat b}(X_i,x)-r_0(X_i))}{1+W^{\mathrm{FO}}(\alpha_{*x}(X_i),Y_i)^\prime r_0(X_i)}\right\}^2$ then further implies that $\frac{c_0}{2N} \sum_{i=1}^NK_h(X_i-x)\left\{\frac{W^{\mathrm{FO}}(\alpha_{*x}(X_i),Y_i)^\prime(\hat{r}_{\hat a, \hat b}(X_i,x)-r_{*x}(X_i))}{1+W^{\mathrm{FO}}(\alpha_{*x}(X_i),Y_i)^\prime r_0(X_i)}\right\}^2\leq \lambda_x\|r_{a^*,b^*}\|_{w_x,h,1} -  \lambda_x\|\hat r_{\hat a,\hat b}\|_{w_x,h,1} + \frac{\lambda_x}{c}\|\hat r_{\hat a,\hat b} - r_{a^*,b^*}\|_{w_x,h,1}   + \widetilde{\mathcal{R}}(x,\hat{\alpha})$. Define $\displaystyle D_{i*}^{\mathrm{FO}}:=\left\{\frac{W^{\mathrm{FO}}(\alpha_{*x}(X_i),Y_i)^\prime(\hat{r}_{\hat a, \hat b}(X_i,x)-r_{*x}(X_i))}{1+W^{\mathrm{FO}}(\alpha_{*x}(X_i),Y_i)^\prime r_0(X_i)}\right\}^2$. Let $T_x$ denote the support of $r_{a^*,b^*}$. Using the fact that $\|r_{a^*,b^*}\|_{w_x,h,1} - \|(\hat r_{\hat a,\hat b})_{T_x}\|_{w_x,h,1}\leq\|r_{a^*,b^*}-(\hat r_{\hat a,\hat b})_{T_x}\|_{w_x,h,1}$, 
\begin{equation}\label{ineq:*:ineq}
\frac{c_0}{2N} \sum_{i=1}^NK_h(X_i-x)D_{i*}^{\mathrm{FO}}-\widetilde{\mathcal{R}}(x,\hat{\alpha})\leq \lambda_x(1+1/c)\|(\hat r_{\hat a,\hat b})_{T_x} - r_{a^*,b^*}\|_{w_x,h,1} -\lambda_x(1-1/c)\|(\hat r_{\hat a,\hat b})_{T_x^c}\|_{w_x,h,1}
\end{equation}
Let $\mathcal{W}_x:={\rm diag}(w_x)$. Define $U_x,\hat{U}_x$ as follows: 
$$\begin{array}{rcl}
U_x&:=&\displaystyle\frac{c_0}{2N}\sum_{i=1}^N\displaystyle\frac{K_h(X_i-x)\left( {W(\alpha_0(X_i),Y_i)}\atop{ {\rm vec}( W(\alpha_0(X_i),Y_i) (X_i-x)'/h)}\right)\left( {W(\alpha_0(X_i),Y_i)}\atop{ {\rm vec}( W(\alpha_0(X_i),Y_i) (X_i-x)'/h)}\right)'}{(1+W(\alpha_0(X_i),Y_i)^\prime r_0(X_i))^2},\\
\hat{U}_x&:=&\displaystyle\frac{c_0}{2N}\sum_{i=1}^N\frac{K_h(X_i-x)\left( {W^{\mathrm{FO}}(\alpha_{*x}(X_i),Y_i)}\atop{ {\rm vec}( W^{\mathrm{FO}}(\alpha_{*x}(X_i),Y_i) (X_i-x)'/h)}\right)\left( {W^{\mathrm{FO}}(\alpha_{*x}(X_i),Y_i)}\atop{ {\rm vec}( W^{\mathrm{FO}}(\alpha_{*x}(X_i),Y_i) (X_i-x)'/h)}\right)'}{(1+W^{\mathrm{FO}}(\alpha_{*x}(X_i),Y_i)^\prime r_0(X_i))^2}.
\end{array}$$ 
Let $\hat v_x = [ (\hat a)', h{\rm vec}(\hat b)' ]' - [ (a^{*})', h{\rm vec}(b^*)']'$. By H\"older's inequality, 
\begin{equation}\label{eq:holder}
\hat v_x'\hat U_x \hat v_x  - \hat v_x' U_x \hat v_x\geq \! - \|\hat r_{\hat a,\hat b} - r_{a^*,b^*}\|_{w_x,h,1}^2 \! \|\mathcal{W}_x^{-1}\!(U_x-\hat U_x)\!\mathcal{W}_x^{-1}\|_{\infty,\infty}
\end{equation}
We next consider two exclusive cases: $\|(\hat r_{\hat a,\hat b})_{T_x^c}\|_{w_x,h,1} \geq 2\frac{c+1}{c-1}\|(\hat r_{\hat a,\hat b})_{T_x} - r_{a^*,b^*}\|_{w_x,h,1}$ and $\|(\hat r_{\hat a,\hat b})_{T_x^c}\|_{w_x,h,1} < 2\frac{c+1}{c-1}\|(\hat r_{\hat a,\hat b})_{T_x} - r_{a^*,b^*}\|_{w_x,h,1}$. Firstly suppose \textbf{Case 1}: $\|(\hat r_{\hat a,\hat b})_{T_x^c}\|_{w_x,h,1} \geq 2\frac{c+1}{c-1}\|(\hat r_{\hat a,\hat b})_{T_x} - r_{a^*,b^*}\|_{w_x,h,1}$.
then \eqref{ineq:*:ineq} implies $\frac{c_0}{2N}\sum_{i=1}^NK_h(X_i-x)D_{i*}^{\mathrm{FO}} + \frac{1}{2}\lambda_x\left(1-\frac{1}{c}\right)\|(\hat r_{\hat a,\hat b})_{T^c_x}\|_{w_x,h,1}\leq\widetilde{\mathcal{R}}(x,\hat\alpha)$, which immediately implies $\lambda_x\left(1-\frac{1}{c}\right)\|(\hat r_{\hat a,\hat b})_{T^c_x}\|_{w_x,h,1}\leq2\widetilde{\mathcal{R}}(x,\hat\alpha)$. Then 
$$\begin{array}{rl}
&\quad\displaystyle\hat v_x'\hat U_x \hat v_x  - \hat v_x' U_x \hat v_x \geq  - \|\hat r_{\hat a,\hat b} - r_{a^*,b^*}\|_{w_x,h,1}^2  \|\mathcal{W}_x^{-1}(U_x-\hat U_x)\mathcal{W}_x^{-1}\|_{\infty,\infty}\\
&\displaystyle\geq  - \left(1+\frac{(c-1)}{2(c+1)}\right)^2\| (\hat r_{\hat a, \hat b})_{T_x^c}\|_{w_x,h,1}^2 \|\mathcal{W}_x^{-1}(U_x-\hat U_x)\mathcal{W}_x^{-1}\|_{\infty,\infty}\\
&\displaystyle\geq-\left(1+\frac{(c-1)}{2(c+1)}\right)^2\ \frac{4\widetilde{\mathcal{R}}^2(x,\hat\alpha)}{(1-\frac{1}{c})^2\lambda^2_x}  \|\mathcal{W}_x^{-1}(U_x-\hat U_x)\mathcal{W}_x^{-1}\|_{\infty,\infty}\\
&\displaystyle\geq-\frac{9\widetilde{\mathcal{R}}^2(x,\hat\alpha)}{(1-1/c)^2\lambda_x^2}\|\mathcal{W}_x^{-1}(U_x-\hat U_x)\mathcal{W}_x^{-1}\|_{\infty,\infty}.
\end{array}$$
Rearranging the terms implies $\hat v_x' U_x \hat v_x \leq  \frac{9\widetilde{\mathcal{R}}^2(x,\hat\alpha)}{(1-1/c)^2\lambda_x^2} \|\mathcal{W}_x^{-1}(\hat U_x - U_x)\mathcal{W}_x^{-1} \|_{\infty,\infty} + \hat v_x'\hat U_x \hat v_x$, which implies $\hat v_x' U_x \hat v_x\leq\frac{9\widetilde{\mathcal{R}}^2(x,\hat\alpha)}{(1-1/c)^2\lambda_x^2} \|\mathcal{W}_x^{-1}(\hat U_x - U_x)\mathcal{W}_x^{-1} \|_{\infty,\infty} +\widetilde{\mathcal{R}}(x,\hat\alpha)$. This inequality holds since \eqref{ineq:*:ineq} implies that $\frac{c_0}{2N}\sum_{i=1}^NK_h(X_i-x)D_{i*}^{\mathrm{FO}} + \frac{1}{2}\lambda_x\left(1-\frac{1}{c}\right)\|(\hat r_{\hat a,\hat b})_{T^c_x}\|_{w_x,h,1}\leq\widetilde{\mathcal{R}}(x,\hat\alpha)$, which immediately implies $\hat v_x'\hat U_x \hat v_x=\frac{c_0}{2N}\sum_{i=1}^NK_h(X_i-x)D_{i*}^{\mathrm{FO}}\leq\widetilde{\mathcal{R}}(x,\hat\alpha)$. Thus in Case 1 we have $\sqrt{\hat v_x' U_x \hat v_x}\leq\frac{3\widetilde{\mathcal{R}}(x,\hat\alpha)}{(1-1/c)\lambda_x} \|\mathcal{W}_x^{-1}(\hat U_x - U_x)\mathcal{W}_x^{-1} \|_{\infty,\infty}^{1/2}+\widetilde{\mathcal{R}}(x,\hat\alpha)^{1/2}$. By Taylor expansion, $\|\mathcal{W}_x^{-1}(\hat U_x - U_x)\mathcal{W}_x^{-1} \|_{\infty,\infty}\leq C_1(\max_{i\in[N]}\|\hat{\alpha}(X_i)-\alpha_0(X_i)\|^2+h^2)$ for some universal constant $C_1$ from the conditions (vi) -- (viii). 
Condition (x) further implies that $\|\mathcal{W}_x^{-1}(\hat U_x - U_x)\mathcal{W}_x^{-1} \|_{\infty,\infty}^{1/2}$ is negligible for $N$ sufficiently large. Hence for some universal constant $C_2$, under case 1 we have 
\begin{equation}\label{eq:case1:covariate}
\sqrt{\hat v_x' U_x \hat v_x}\leq C_2\widetilde{\mathcal{R}}(x,\hat\alpha)^{1/2}.
\end{equation}
Next suppose \textbf{Case 2} holds: $\displaystyle\|(\hat r_{\hat a, \hat b})_{T_x^c}\|_{w_x,h,1} \leq 2\frac{c+1}{c-1}\|(\hat r_{\hat a,\hat b} - r_{a^*,b^*})_{T_x}\|_{w_x,h,1}$. Define $$\displaystyle\hat{\Delta}_{ab}^i:=\left\{\frac{W(\alpha_0(X_i),Y_i)^\prime(\hat{r}_{\hat a,\hat b}(X_i,x)-r_{*x}(X_i))}{1+W(\alpha_0(X_i),Y_i)^\prime r_0(X_i)}\right\}^2.$$ 
Then we have
\begin{equation}\label{eq:thm1:main:main:well:covariate}
\begin{array}{rl}
&\displaystyle\quad \hat v_x' U_x \hat v_x=\frac{c_0}{2N} \sum_{i=1}^NK_h(X_i-x)\hat{\Delta}_{ab}^i\leq_{(1)}\hat v_x' \hat{U}_x \hat v_x +\|\hat r_{\hat a,\hat b} - r_{a^*,b^*}\|_{w_x,h,1}^2\! \|\mathcal{W}_x^{-1}(U_x-\hat U_x)\mathcal{W}_x^{-1}\|_{\infty,\infty}\\
&\displaystyle \leq_{(2)} \lambda_x(1+\frac{1}{c})\|(\hat r_{\hat a,\hat b} - r_{a^*,b^*})_{T_x}\|_{w_x,h,1} + \|\hat r_{\hat a,\hat b} - r_{a^*,b^*}\|_{w_x,h,1}^2\! \|\mathcal{W}_x^{-1}(U_x-\hat U_x)\mathcal{W}_x^{-1}\|_{\infty,\infty} + \widetilde{\mathcal{R}}(x,\hat\alpha)\\
&\displaystyle\leq_{(3)}\lambda_x(1+1/c)\|(\hat r_{\hat a,\hat b} - r_{a^*,b^*})_{T_x}\|_{w_x,h,1}\\
&\quad\displaystyle+2\left[4\left(\frac{c+1}{c-1}\right)^2+1\right]\|(\hat{r}_{\hat{a},\hat{b}}-r_{a^*,b^*})_{T_x}\|_{w_x,h,1}^2\|\mathcal{W}_x^{-1}(U_x-\hat U_x)\mathcal{W}_x^{-1}\|_{\infty,\infty}+ \widetilde{\mathcal{R}}(x,\hat\alpha)\\
&\displaystyle\leq_{(4)}\frac{\lambda_x}{\sqrt{c_0}}\left(1+\frac{1}{c}\right) \frac{\sqrt{s}}{\kappa_{x,\bar{c}}}\sqrt{\frac{c_0}{N} \sum_{i=1}^NK_h(X_i-x)\hat{\Delta}_{ab}^i}+\frac{s\eta_0\|\!\mathcal{W}_x^{-1}\!(U_x-\hat U_x)\!\mathcal{W}_x^{-1}\!\|_{\infty,\infty}}{\kappa_{x,\bar{c}}^2}\!\!\!\sum_{i=1}^N\!\!\frac{K_h(X_i-x)\hat{\Delta}_{ab}^i}{N}\\
&\displaystyle\quad + \widetilde{\mathcal{R}}(x,\hat\alpha),
\end{array}
\end{equation}
where (1) of \eqref{eq:thm1:main:main:well:covariate} follows from \eqref{eq:holder}, (2) of \eqref{eq:thm1:main:main:well:covariate} follows from \eqref{ineq:*:ineq} and the fact that $\hat{v}_x\hat{U}_x\hat{v}_x=\frac{c_0}{2N} \sum_{i=1}^NK_h(X_i-x)D_{i*}^{\mathrm{FO}}$, (3) follows since $\|\hat{r}_{\hat{a},\hat{b}}-r_{a^*,b^*}\|_{w_x,h,1}^2\leq2\|(\hat{r}_{\hat{a},\hat{b}}-r_{a^*,b^*})_{T_x}\|_{w_x,h,1}^2+2\|(\hat{r}_{\hat{a},\hat{b}})_{T_x^c}\|_{w_x,h,1}^2\leq2\left[4(\frac{c+1}{c-1})^2+1\right]\|(\hat{r}_{\hat{a},\hat{b}}-r_{a^*,b^*})_{T_x}\|_{w_x,h,1}^2$,
(4) follows from the definition of restricted eigenvalue and the fact that $\|(\hat{r}_{\hat{a},\hat{b}}-r_{a^*,b^*})_{T_x}\|_{w_x,h,1}\leq\sqrt{s}\|(\hat{r}_{\hat{a},\hat{b}}-r_{a^*,b^*})_{T_x}\|_{w_x,h,2}$, with $\eta_0=2\left[4\left(\frac{c+1}{c-1}\right)^2+1\right]$.
Note that by condition (x), $\frac{s\eta_0\|\mathcal{W}_x^{-1}(U_x-\hat U_x)\mathcal{W}_x^{-1}\|_{\infty,\infty}}{\kappa_{x,\bar{c}}^2} \leq c_0/4$ when $N$ is sufficiently large. Further note that $\frac{1}{4}t^2 \leq \frac{\lambda_x}{\sqrt{c_0}}(1+\frac{1}{c})\frac{\sqrt{s}}{\kappa_{x,\bar{c}}} t + \mathcal{\widetilde{R}}(x,\hat \alpha) \Rightarrow t \leq\frac{4}{\sqrt{c_0}}\lambda_x(1+\frac{1}{c})\frac{\sqrt{s}}{\kappa_{x,\bar{c}}} + 2{\mathcal{\widetilde{R}}}^{1/2}(x,\hat \alpha)$. Let $t=\sqrt{2}\sqrt{\hat{v}_x'U_x\hat{v}_x}$, then we have $\sqrt{\hat{v}_x'U_x\hat{v}_x}\leq\frac{2\sqrt{2}}{\sqrt{c_0}}\lambda_x(1+\frac{1}{c})\frac{\sqrt{s}}{\kappa_{x,\bar{c}}} + \sqrt{2}{\mathcal{\widetilde{R}}}^{1/2}(x,\hat \alpha)$ in Case 2. Recall that \eqref{eq:case1:covariate} holds in Case 1. Thus the result holds from the definition of $\widetilde{\mathcal{R}}(x,\hat{\alpha})$ and the second statement of condition (i). 
\proofend

Recall from Section~\ref{Sec:Rates:covariates} that $\beta_{N,x}:={\displaystyle\max_{j\in[p],k\in[d]}}\left\{\mathcal{B}_{1,xj},\mathcal{B}_{2,xjk}\right\}$, where for any $j\in[p], k\in[d]$, $\mathcal{B}_{1,xj}:=\frac{1}{w_{xj}}\frac{1}{N}\sum_{i=1}^N\frac{K_h(X_i-x)|W_j(\alpha_0(X_i),Y_i)|}{1+W(\alpha_0(X_i),Y_i)^\prime r_0(X_i)}$, $\mathcal{B}_{2,xjk}:=\frac{1}{w_{xjk}}\frac{1}{N}\sum_{i=1}^N\frac{K_h(X_i-x)|W_j(\alpha_0(X_i),Y_i)(X_{ik}-x_k)/h|}{1+W(\alpha_0(X_i),Y_i)^\prime r_0(X_i)}$. 
\begin{lemma}\label{lemma:choice of lambda:covariate}
Under Assumptions \ref{ass:iid-min-den}, \ref{ass:moment}, \ref{ass:local-linear}, \ref{ass:smoothness-assumptions}, for some absolute constant $C$,
$$\begin{array}{rl}
&\displaystyle\quad\left\|\mathrm{diag}^{-1}(w_x)\frac{1}{N}\sum_{i=1}^N\frac{  K_h(X_i-x) }{1+W(\hat{\alpha}(X_i),Y_i)^\prime r_0(X_i)} \left( { W(\hat{\alpha}(X_i),Y_i)}\atop{ {\rm vec}( W(\hat{\alpha}(X_i),Y_i)\frac{(X_i-x)'}{h} )}\right)^\prime \right\|_\infty \\
&\displaystyle\leq\left\|\mathrm{diag}^{-1}(w_x)\frac{1}{N}\sum_{i=1}^N\frac{  K_h(X_i-x) }{1+W(\alpha_0(X_i),Y_i)^\prime r_0(X_i)} \left( { W(\alpha_0(X_i),Y_i)}\atop{ {\rm vec}( W(\alpha_0(X_i),Y_i)\frac{(X_i-x)'}{h} )}\right)^\prime \right\|_\infty\\
&\displaystyle\quad+C(1+\beta_{N,x})\max_{\|X_i-x\|\leq h}\norm{\hat{\alpha}(X_i)-\alpha_0(X_i)}\frac{M^{1/2}K_{\alpha,1}}{\xi},
\end{array}$$
$$\begin{array}{rl}
&\displaystyle\left\|\mathrm{diag}^{-1}(w_x)\frac{1}{N}\sum_{i=1}^N\frac{  K_h(X_i-x) }{1+W^{\mathrm{FO}}(\alpha_{*x}(X_i),Y_i)^\prime r_0(X_i)} \left( { W^{\mathrm{FO}}(\alpha_{*x}(X_i),Y_i)}\atop{ {\rm vec}( W^{\mathrm{FO}}(\alpha_{*x}(X_i),Y_i)\frac{(X_i-x)'}{h} )}\right)^\prime \right\|_\infty \\
&\displaystyle\leq\left\|\mathrm{diag}^{-1}(w_x)\frac{1}{N}\sum_{i=1}^N\frac{ K_h(X_i-x) }{1+W(\alpha_0(X_i),Y_i)^\prime r_0(X_i)} \left( { W(\alpha_{*x}(X_i),Y_i)}\atop{ {\rm vec}( W(\alpha_{*x}(X_i),Y_i)\frac{(X_i-x)'}{h} )}\right)^\prime \right\|_\infty\\
&\displaystyle\quad+C(1+\beta_{N,x})\frac{K_{\alpha,2}}{\xi}\max_{\alpha \in \mathcal{A}_x(\hat\alpha), \|X_i-x\|\leq h}\norm{\alpha(X_i)-\alpha_0(X_i)}^2+Ch^2.
\end{array}$$
Moreover, with probability $1-C\delta_{\alpha}-C\delta(1+\mathrm{o}(1))$,
$$\begin{array}{rl}
&\displaystyle\quad\left\|\mathrm{diag}^{-1}(w_x)\frac{1}{N}\sum_{i=1}^N\frac{  K_h(X_i-x) }{1+W(\alpha_0(X_i),Y_i)^\prime r_0(X_i)} \left( { W(\alpha_0(X_i),Y_i)}\atop{ {\rm vec}( W(\alpha_0(X_i),Y_i)\frac{(X_i-x)'}{h} )}\right)^\prime \right\|_\infty \\
&\displaystyle\leq\frac{\Phi^{-1}(1-\delta/\{2p(1+d)\})}{\sqrt{Nh^d}}\max_{j\in[p],k\in\{0\}\cup[d]}\frac{1}{w_{xjk}}\sqrt{\frac{1}{N}\sum_{i=1}^N\frac{K_h(X_i-x)W_j(\alpha_0(X_i),Y_i)^2}{(1+W(\alpha_0(X_i),Y_i)^\prime  r_0(X_i))^2}}.
\end{array}$$
\end{lemma}
\proof{Proof of Lemma \ref{lemma:choice of lambda:covariate}} 
The proof is similar to the no-covariate case, i.e. Lemma \ref{lemma:Score:nocovariate}. 
By Taylor expansion, for each $j\in[p],i\in[N]$, there exists $\tilde{\alpha}^{ji}$, $\tilde{\alpha}$ in the line segment between $\hat{\alpha}(X_i)$ and $\alpha_0(X_i)$, such that 
$W_j(\hat{\alpha}(X_i),Y_i)=W_j(\alpha_0(X_i),Y_i)+(\hat{\alpha}(X_i)-\alpha_0(X_i))^\prime\nabla_{\alpha}W_j(\tilde{\alpha}^{ji},Y_i)$. Thus 
$$
\begin{array}{rl}
&\displaystyle\quad\frac{W_j(\hat\alpha(X_i),Y_i)}{1+W(\hat\alpha(X_i),Y_i)'r_0(X_i)}= \frac{W_j(\alpha_0(X_i),Y_i) +   (\hat\alpha(X_i)-\alpha_0(X_i))'\nabla_\alpha W_j(\tilde \alpha^{ji},Y_i)}{1+W(\alpha_0(X_i),Y_i)'r_0(X_i) + (\hat\alpha(X_i)-\alpha_0(X_i))'\nabla_\alpha W(\tilde{\alpha},Y_i)'r_0(X_i)} \\
&\displaystyle= \frac{W_j(\alpha_0(X_i),Y_i) }{1+W(\alpha_0(X_i),Y_i)'r_0(X_i)} \left( 1 -  \frac{(\hat\alpha(X_i)-\alpha_0(X_i))'\nabla_\alpha W(\tilde \alpha,Y_i)'r_0(X_i)}{1+W(\alpha_0(X_i),Y_i)'r_0(X_i) + (\hat\alpha(X_i)-\alpha_0(X_i))'\nabla_\alpha W(\tilde{\alpha},Y_i)'r_0(X_i)}\right)\\
&\displaystyle\quad+ \frac{(\hat\alpha(X_i)-\alpha_0(X_i))'\nabla_\alpha W_j(\tilde \alpha^{ji},Y_i)}{1+W(\alpha_0(X_i),Y_i)'r_0(X_i) + (\hat\alpha(X_i)-\alpha_0(X_i))'\nabla_\alpha W(\tilde{\alpha},Y_i)'r_0(X_i)}\\
&\displaystyle = \frac{W_j(\alpha_0(X_i),Y_i) }{1+W(\alpha_0(X_i),Y_i)'r_0(X_i)} + \widetilde R(\hat\alpha(X_i),Y_i),
\end{array}$$ 
where $\displaystyle|\widetilde R(\hat\alpha(X_i),Y_i)|\leq C\left(1+\frac{|W_j(\alpha_0(X_i),Y_i)|}{1+W(\alpha_0(X_i),Y_i)^\prime r_0(X_i)}\right)\frac{\norm{\hat{\alpha}(X_i)-\alpha_0(X_i)}M^{1/2}K_{\alpha,1}}{\xi}$ with probability $1-\delta_{\alpha}$ from Assumption~\ref{ass:smoothness-assumptions}. Note that $K_h(X_i-x)>0$ if and only if $\norm{(X_i-x)/h}\leq1$. So the first inequality holds. Now we prove the second inequality. By Taylor expansion, for each $j\in[p],i\in[N]$, there exist $\tilde{\alpha}_{ji},\tilde{\alpha}(X_i)$ on the line segment between $\hat{\alpha}(X_i)$ and $\alpha_0(X_i)$ such that 
$$\begin{array}{rl}
&\quad\displaystyle W^{\mathrm{FO}}_j(\alpha_{*x}(X_i),Y_i)= W_j(\hat \alpha(X_i),Y_i) + (\alpha_{*x}(X_i)-\hat\alpha(X_i))'\nabla_\alpha W_j(\hat\alpha(X_i),Y_i)\\
&\displaystyle= W_j(\hat \alpha(X_i),Y_i) + (\alpha_0(X_i)-\hat\alpha(X_i))'\nabla_\alpha W_j(\hat\alpha(X_i),Y_i)+(\alpha_{*x}(X_i)-\alpha_0(X_i))'\nabla_\alpha W_j(\hat\alpha(X_i),Y_i)\\
&\displaystyle = W_j(\alpha_0(X_i),Y_i) - \frac{1}{2}(\alpha_0(X_i)-\hat\alpha(X_i))'\nabla_\alpha^2 W_j(\tilde\alpha_{ji},Y_i)(\alpha_0(X_i)-\hat\alpha(X_i))\\
&\quad\displaystyle +(\alpha_{*x}(X_i)-\alpha_0(X_i))'\nabla_\alpha W_j(\hat\alpha(X_i),Y_i),\ \mbox{and}
\end{array}$$
$\frac{W^{\mathrm{FO}}_j(\alpha_{*x}(X_i),Y_i) }{1+W^{\mathrm{FO}}(\alpha_{*x}(X_i),Y_i)'r_0(X_i)} =  \frac{W_j(\alpha_0(X_i),Y_i)}{1+W(\alpha_0(X_i),Y_i)'r_0(X_i)} + \widetilde R(\hat\alpha(X_i),\alpha_0(X_i),Y_i)$, where $|\widetilde R(\hat{\alpha}(X_i),\alpha_0(X_i),Y_i)|\leq C\left(1+\frac{|W_j(\alpha_0(X_i),Y_i)| }{1+W(\alpha_0(X_i),Y_i)'r_0(X_i)}\right) \max_{\alpha \in \mathcal{A}_x(\hat\alpha), \|X_i-x\|\leq h}\norm{\alpha(X_i)-\alpha_0(X_i)}^2+Ch^2$ with probability $1-\delta_{\alpha}$ from Assumption~\ref{ass:smoothness-assumptions}, for some universal constant $C$. 

Next, define 
$\tilde X_{ij} := \left(\frac{K_h(X_i-x)W_{j}(\alpha_0(X_i),Y_i)}{1+W(\alpha_0(X_i),Y_i)'r_0(X_i)},\frac{K_h(X_i-x) W_j(\alpha_0(X_i),Y_i)(X_i-x)'/h}{1+W(\alpha_0(X_i),Y_i)'r_0(X_i)}\right)\in\mathbb{R}^{d+1}$, where we denote $\tilde{X}_{ij0}:=\frac{K_h(X_i-x)W_{j}(\alpha_0(X_i),Y_i)}{1+W(\alpha_0(X_i),Y_i)'r_0(X_i)}$, and $\tilde{X}_{ijl}:=\frac{K_h(X_i-x) W_j(\alpha_0(X_i),Y_i)(X_{il}-x_{l})/h}{1+W(\alpha_0(X_i),Y_i)'r_0(X_i)}$ for any $l\in[d]$. So under Assumptions \ref{ass:iid-min-den} and \ref{ass:moment}, we have 
$$\begin{array}{rl}
&\displaystyle\quad\min_{j\in[p],l\in\{0\}\cup[d]}\mathbb{E}[\tilde X_{ijl}^2 ]\\
&\displaystyle\geq_{(1)} \min_{j\in[p],k\in[d]}\mathbb{E}\Bigg[\frac{K(\norm{X_i-x}/h)^2\mathbf{1}\{\norm{X_i-x}\leq h\}}{h^{2d}}\frac{(X_{ik}-x_k)^2}{h^2}\left(\frac{W_j(\alpha_0(X_i),Y_i)}{1+W(\alpha_0(X_i),Y_i)'r_0(X_i)}\right)^2\Bigg]\\
&\displaystyle\geq_{(2)}\underline{C}_1\min_{j\in[p],k\in[d]}\mathbb{E}\left[\frac{\mathbf{1}\{\norm{X_i-x}\leq h\}}{h^{2d}}\frac{(X_{ik}-x_k)^2}{h^2}\left(\frac{W_j(\alpha_0(X_i),Y_i)}{1+W(\alpha_0(X_i),Y_i)'r_0(X_i)}\right)^2\right]\\
&\displaystyle\geq_{(3)}\underline{C}_1\bar\psi_2^2\int_{\norm{u-x}\leq h}\frac{(u_k-x_k)^2}{h^{2d}h^2}\rho(u)du\geq_{(4)}\underline{C}\bar\psi_2^2/h^d,
\end{array}$$
where (1) follows because $|X_{ik}-x_k|\leq h$, $K_h(X_i-x)>0\iff\norm{X_i-x}\leq h$ and $K_h(X_i-x)=\frac{1}{h^d}K(\norm{X_i-x}/h)$. (2) follows because $K(\norm{X_i-x}/h)$ is bounded away from zero on its support, (3) follows by definition of $\bar{\psi}_2$ in Assumption \ref{ass:moment}, and $\rho(\cdot)$ is the density of covariate $X$, and (4) holds because the volume of the Euclidean ball $\{u:\norm{u-x}\leq h\}$ is equal to $Ch^d$ for some constant $C$, and that $\rho(\cdot)\geq\rho_0>0$ on its support by Assumption \ref{ass:local-linear}. Also note that  
$$\begin{array}{rl}
&\displaystyle\quad\max_{j\in[p],k\in\{0\}\cup[d]}\left(\frac{1}{N}\sum_{i=1}^N\mathbb{E}[|\tilde X_{ijk}|^3 ]\right)^{1/3}\\ 
&\displaystyle\leq\max_{j\in[p]}\left(\frac{1}{N}\sum_{i=1}^N\mathbb{E}\left[\frac{K(\norm{X_i-x}/h)^3\mathbf{1}\{\norm{X_i-x}\leq h\}}{h^{3d}}\left|\frac{W_j(\alpha_0(X_i),Y_i)}{1+W(\alpha_0(X_i),Y_i)'r_0(X_i)}\right|^3\right]\right)^{1/3}\!\!\!\!\!\!\!\!\leq\bar{C}\bar{K}_3h^{-2d/3},
\end{array}$$
where $\underline{C}$, $\bar{C}$ are absolute constants. So 
$$\begin{array}{rl}
&\displaystyle\mathbb{P} \left(\left| \frac{1}{N}\sum_{i=1}^N \frac{\tilde X_{ijk}}{w_{xjk}} \right| > \frac{\Phi^{-1}(1-\delta/(2p(d+1)))}{w_{xjk} N^{1/2}} \sqrt{\frac{1}{N}\sum_{i=1}^N \tilde X_{ijk}^2}, \mbox{for some } j\in [p], k\in\{0\}\cup[d] \right)  \\
&\displaystyle \leq  \sum_{j\in [p],k\in\{0\}\cup[d]} \mathbb{P} \left( \left|\frac{1}{N}\sum_{i=1}^N\frac{\tilde X_{ijk}}{w_{xjk}} \right| > \frac{\Phi^{-1}(1-\delta/(2p(d+1)))}{w_{xjk}N^{1/2}} \sqrt{\frac{1}{N}\sum_{i=1}^N \tilde X_{ijk}^2 } \right)\\
&\displaystyle\leq\sum_{j\in [p],k\in\{0\}\cup[d]}\left[1+A(\bar{K}_3/\bar{\psi}_2)^3\frac{(1+\Phi^{-1}(1-\delta/(2p(d+1))))^{3}}{\sqrt{Nh^d}}\right]\{\delta/(2p(d+1))\}\leq \delta \left(1+\mathrm{o}(1)\right),
\end{array}$$
where the first inequality above follows from the union bound, the second inequality follows from Theorem 2.3 in \citet{jing2003self}, where $A$ is an absolute constant, and the last inequality follows since $p=2^M-M-1$ and $\sqrt{\frac{M^3}{Nh^d}}\leq\delta_N$ from Assumption~\ref{ass:dimension}, so $(1+\Phi^{-1}(1-\delta/(2p(d+1)))^3/\sqrt{Nh^d}\leq C_1\sqrt{\frac{\log^3(p(d+1))}{Nh^d}}\leq C_1\delta_N$ for some absolute constant $C_1$.
Further, Denote $\bar{W}_{ij}:=\frac{W_j(\alpha_0(X_i),Y_i)}{1+W(\alpha_0(X_i),Y_i)'r_0(X_i)}$. For any $j\in[p],k\in\{0\}\cup[d]$, 
$\sqrt{\frac{1}{N}\sum_{i=1}^N \tilde X_{ijk}^2}\leq\sqrt{\frac{1}{N}\sum_{i=1}^NK_h(X_i-x)^2\bar{W}_{ij}^2}\leq \frac{C}{\sqrt{h^d}}\sqrt{\frac{1}{N}\sum_{i=1}^NK_h(X_i-x)\bar{W}_{ij}^2}$,
where the second inequality follows because $\sqrt{K_h(X_i-x)}\leq C/\sqrt{h^{d}}$ and $C$ is an absolute constant. Thus the result follows. 
\proofend

\begin{lemma}\label{lemma:covariate:score*approx}
Suppose that $Z_i(X_i), i\in [N],$ are independent random variables such that \\
$\mathbb{E}[Z_i(X_i)\mid X_i]=0$,  $\sigma^2 \geq \max_{i\in[N]}\mathbb{E}[Z_i^2(X_i) \mid X_i]$. Then for a given $x\in \mathcal{X}$ we have that with some absolute constants $\bar{C},C_1$, with probability $1-\delta-C_1(\bar{K}/\sigma)^2/(Nh^d)$, where $\bar{K}=\mathbb{E}[\max_{i\in[N]}Z_i(X_i)^2]^{1/2}$, we have $\left|\frac{1}{N}\sum_{i=1}^N K_h(X_i-x) Z_i(X_i) \right|\leq\bar{C}\sigma\sqrt{\frac{\log(1/\delta)}{Nh^d}}$. In particular, under Assumption \ref{ass:iid-min-den}, \ref{ass:moment}, with probability $\displaystyle1-\delta-C(\bar{K}_{\max}/\bar{K}_2)^2/(Nh^d)$, for some universal constant $C$ independent of $s,N,h,M,d$, we have $\displaystyle\left|\frac{1}{N}\sum_{i=1}^N K_h(X_i-x)\frac{ W(\alpha_0(X_i),Y_i)^\prime(r_{*x}(X_i)-r_0(X_i))}{1+W(\alpha_0(X_i),Y_i)^\prime r_0(X_i)}\right| \leq C \bar{K}_2sdh^2\sqrt{\frac{\log(1/\delta)}{Nh^d}}$.
\end{lemma}
\proof{Proof of Lemma \ref{lemma:covariate:score*approx}}
Let $X_1^\prime,\ldots,X_N^\prime$ be i.i.d. copies of $X_1,\ldots,X_N$, and $Z_i(X_i)^\prime$ be i.i.d. copies of $Z_i(X_i)$ for $i\in[N]$. Let $\mathbf{X}(N):=\{X_1,\ldots,X_N\}$ and $\mathbf{Z}(N):=\{Z_1(X_1),\ldots,Z_N(X_N)\}$, define $\displaystyle W = \mathbb{E}\left[\sum_{i=1}^N\left(K_h(X_i-x)Z_i(X_i)-K_h(X_i^\prime-x)Z_i(X_i)^\prime\right)^2\big| \mathbf{X}(N),\mathbf{Z}(N)\right]$.
Denote $Z:=\sum_{i=1}^NK_h(X_i-x)Z_i(X_i)$. Then by Theorem 12.3 of \citet{boucheron2003concentration}, for all $t\geq0$, we have $\displaystyle\mathbb{P}\left\{\frac{|Z|}{\sqrt{W}}\geq2\sqrt{t}\right\}\leq4e^{-t/4}$. So for any $\delta>0$, for some absolute constant $\bar{C}_0$, with probability $1-\delta$, we have 
$$\begin{array}{rl}
(*)=\displaystyle\frac{\left|\sum_{i=1}^N K_h(X_i-x) Z_i(X_i) \right|}{\sqrt{\mathbb{E}\left[\sum_{i=1}^N\left(K_h(X_i-x)Z_i(X_i)-K_h(X_i^\prime-x)Z_i(X_i)^\prime\right)^2\big| \mathbf{X}(N),\mathbf{Z}(N)\right]}}\leq \bar{C}_0\sqrt{\log(1/\delta)}.
\end{array}$$
\begin{equation}\label{eq:exp:Z:sq}
\begin{array}{rl}
&\displaystyle\quad\mathbb{E}\left[\sum_{i=1}^N\left(K_h(X_i-x)Z_i(X_i)-K_h(X_i^\prime-x)Z_i(X_i)^\prime\right)^2\big| X_1,\ldots,X_N\right]\\
&\displaystyle\leq\mathbb{E}\left[2\sum_{i=1}^N\left(K_h(X_i-x)Z_i(X_i)\right)^2+2\sum_{i=1}^N\left(K_h(X_i^\prime-x)Z_i(X_i)^\prime\right)^2\big| X_1,\ldots,X_N\right]\\
&\displaystyle\leq_{(1)}2\sum_{i=1}^N\mathbb{E}\left[\left(K_h(X_i'-x)Z_i(X_i)'\right)^2\right]+2\sigma^2\sum_{i=1}^N\mathbb{E}\left[\frac{K((X_i-x)/h)^2\mathbf{1}\{\norm{X_i-x}\leq h\}}{h^{2d}}\big|X_i\right]\\
&\displaystyle\leq_{(2)}2\sigma^2\sum_{i=1}^N\left\{\mathbb{E}\left[\frac{K((X_i-x)/h)^2\mathbf{1}\{\|X_i-x\|\leq h\}}{h^{2d}}\right]+\mathbb{E}\left[\frac{K((X_i-x)/h)^2\mathbf{1}\{\norm{X_i-x}\leq h\}}{h^{2d}}\big|X_i\right]\right\}
\end{array}
\end{equation}
where (1) follows since $X_i', Z_i(X_i)'$ are independent of $\{X_1,\ldots,X_N\}$ and $\sigma^2 \geq \max_{i\in[N]}\mathbb{E}[Z_i^2(X_i) \mid X_i]$. (2) follows since $\mathbb{E}\left[\left(K_h(X_i'-x)Z_i(X_i)'\right)^2\right]=_{(a)}\mathbb{E}\left[\left(K_h(X_i-x)Z_i(X_i)\right)^2\right]=_{(b)}\mathbb{E}\{K_h(X_i-x)^2\mathbb{E}\left[Z_i(X_i)^2|X_i\right]\}\leq_{(c)}\sigma^2\mathbb{E}\left[\frac{K((X_i-x)/h)^2\mathbf{1}\{\|X_i-x\|\leq h\}}{h^{2d}}\right]$, where (a) follows since $X_i', Z_i(X_i)'$ are identically distributed copies of $X_i,Z_i(X_i)$, (b) follows by tower property, and (c) follows $\sigma^2 \geq \max_{i\in[N]}\mathbb{E}[Z_i^2(X_i) \mid X_i]$. Further note that $0\leq K(\cdot)\leq\bar{C}$, and for some absolute constant $C_0$ it holds that $\mathbb{P}(\norm{X_i-x}\leq h)=\int_{B_h(x)}\rho(du)=C_0h^d$, hence taking expectation on both sides of \eqref{eq:exp:Z:sq} we have $\mathbb{E}\left\{\mathbb{E}\left[\sum_{i=1}^N\left(K_h(X_i-x)Z_i(X_i)-K_h(X_i^\prime-x)Z_i(X_i)^\prime\right)^2\big| \mathbf{X}(N),\mathbf{Z}(N)\right]\right\}\leq\mathbb{E}\left[\frac{4\bar{C}^2\sigma^2}{h^{2d}}\sum_{i=1}^N\mathbf{1}\{\norm{X_i-x}\leq h\}\right]\leq\frac{4NC_0\bar{C}^2\sigma^2}{h^d}$. Let 
$$\bar{M}=\max_{i\in[N]}\mathbb{E}\left[K_h(X_i-x)^2Z_i(X_i)^2|X_1,\ldots,X_N,Z_1(X_1),\ldots,Z_N(X_N)\right],$$
$$\bar{M}^\prime=\max_{i\in[N]}\mathbb{E}\left[K_h(X_i^\prime-x)^2Z_i(X_i)^{\prime 2}|X_1,\ldots,X_N,Z_1(X_1),\ldots,Z_N(X_N)\right].$$ 
Then $\max_{i\in[N]}\mathbb{E}\left[(K_h(X_i-x)Z_i(X_i)-K_h(X_i^\prime-x)Z_i(X_i)^\prime)^2\big|\mathbf{X}(N),\mathbf{Z}(N)\right]\leq2(\bar{M}+\bar{M}^\prime)$.
Thus 
$$\mathbb{E}\left\{\max_{i\in[N]}\mathbb{E}\left[(K_h(X_i-x)Z_i(X_i)-K_h(X_i^\prime-x)Z_i(X_i)^\prime)^2\big|\mathbf{X}(N),\mathbf{Z}(N)\right]\right\}\leq2\mathbb{E}[\bar{M}+\bar{M}^\prime]\leq_{(a)}4\mathbb{E}[\bar{M}],$$
where (a) holds since $\mathbb{E}[\bar{M}']\leq\mathbb{E}[\bar{M}]$. Further note that $\mathbb{E}[\bar{M}]=\mathbb{E}[\max_{i\in[N]}K_h(X_i-x)^2Z_i(X_i)^2]\leq\frac{\bar{C}^2}{h^{2d}}\bar{K}^2$, where $\bar{K}=\mathbb{E}[\max_{i\in[N]}Z_i(X_i)^2]^{1/2}$. Thus using Lemma E.4 of \citet{chernozhukov2017central}, for any $t>0$, for some absolute constant $C_1$, with probability at least $\displaystyle1-C_1\frac{\bar{C}^2\bar{K}^2}{th^{2d}}$, we have 
$\mathbb{E}\left[\sum_{i=1}^N\left(K_h(X_i-x)Z_i(X_i)-K_h(X_i^\prime-x)Z_i(X_i)^\prime\right)^2\big| \mathbf{X}(N),\mathbf{Z}(N)\right]\leq\frac{8NC_0\bar{C}^2\sigma^2}{h^d}+t$,
indicating that with some absolute constants $\bar{C}_0$ and $\bar{C}_1$, with probability at least $\displaystyle1-\frac{C_1(\bar{K}/\sigma)^2}{Nh^d}$, $\displaystyle\mathbb{E}\left[\sum_{i=1}^N\left(K_h(X_i-x)Z_i(X_i)-K_h(X_i^\prime-x)Z_i(X_i)^\prime\right)^2\big|\mathbf{X}(N),\mathbf{Z}(N)\right]\leq\bar{C}_0\sigma^2\frac{N}{h^d}$. Thus further using the inequality for $(*)$ derived above, with some absolute constants $\bar{C},C_1$, with probability $1-\delta-\frac{C_1(\bar{K}/\sigma)^2}{Nh^d}$, $\displaystyle\left|\frac{1}{N}\sum_{i=1}^N K_h(X_i-x) Z_i(X_i) \right|\leq\bar{C}\sigma\sqrt{\frac{\log(1/\delta)}{Nh^d}}$. So the first statement holds. The second statement in the Lemma follows immediately as a corollary by letting $\displaystyle Z_i(X_i)=\frac{ W(\alpha_0(X_i),Y_i)^\prime(r_{*x}(X_i)-r_0(X_i))}{1+W(\alpha_0(X_i),Y_i)^\prime r_0(X_i)}$, so $\mathbb{E}[Z_i(X_i)|X_i]=0$. Recall that $\left\|  r_{*xj}(X_i)-r_{0j}(X_i) \right\|_1 \leq \bar{K}_{x,2}dh^2$ by component-wise Taylor expansion. Here we denote $\displaystyle\widetilde{W}_i=\frac{W(\alpha_0(X_i),Y_i)}{1+W(\alpha_0(X_i),Y_i)'r_0(X_i)}$ and let $T_i$ be the support of $r_{*x}(X_i)-r_0(X_i)$, we have $|T_i|\leq 2s$, so with some absolute constant $C$ we have $\mathbb{E}[\max_{i\in[N]}Z_i(X_i)^2]^{1/2}=\mathbb{E}\left[\max_{i\in[N]}\left(\sum_{j\in T_i}\widetilde{W}_{ij}(r_{*xj}(X_i)-r_{0j}(X_i))\right)^2\right]^{1/2}\!\!\!\!\!\!\!\!\leq C\bar{K}_{\max}sdh^2$. Also note that $\max_{i\in[N]}\mathbb{E}[Z_i^2(X_i)|X_i]\leq \max_{i\in[N]}\mathbb{E}\left[\left(\sum_{j\in T_i}\widetilde{W}_{ij}(r_{*xj}(X_i)-r_{0j}(X_i))\right)^2\big|X_i\right]\leq C\bar{K}_{2}^2s^2d^2h^4$. Since $\bar{K}=\mathbb{E}[{\displaystyle\max_{i\in[N]}}Z_i(X_i)^2]^{1/2}\leq C\bar{K}_{\max}sdh^2$, $\sigma^2\leq C\bar{K}_2^2s^2d^2h^4$. So the second statement holds.
\proofend

\begin{lemma}\label{lemma:MinMax:Claim1:covariate}
Suppose Assumptions~\ref{ass:iid-min-den}, \ref{ass:moment}, \ref{ass:smoothness-assumptions} hold. Then with probability at least $1-C(\delta+\epsilon+\delta_{\alpha})$, we have $\min_{\alpha \in \mathcal{A}_x(\hat\alpha)} Q^{\mathrm{FO}}(\alpha,r_{*x})-Q^{\mathrm{FO}}(\alpha_{*x},r_{*x}) 
 \geq-C\max_{\alpha \in \mathcal{A}_x(\hat\alpha), \|X_i-x\|\leq h}\|\alpha(X_i)-\alpha_0(X_i)\|_1 \sqrt{\frac{\log (M/\delta)}{Nh^d}}-C\max_{\alpha \in \mathcal{A}_x(\hat\alpha), \|X_i-x\|\leq h}\!\!\|\alpha(X_i)\! - \!\alpha_0(X_i)\|^2 \!\!-Csdh^2\left(\frac{M}{\sqrt{Nh^{d}}}+Mdh^2\right)$, where $C$ is a universal constant independent of $N,M,s,d,h$ and $\delta\geq\max\{\exp(-\bar{C}_0\{\sqrt{\log M}\sqrt{Nh^d}+Nh^d\}),\log(M)/(\delta_NNh^d)\}$.
\end{lemma}
\proof{Proof of Lemma~\ref{lemma:MinMax:Claim1:covariate}}
Take $\alpha(\cdot) \in \mathcal{A}_x(\hat\alpha)$, where $\alpha(X_i):=\tilde{\alpha}+\tilde{\beta}(X_i-x)$ for some $\tilde{\alpha}\in\mathbb{R}^M, \tilde{\beta}\in\mathbb{R}^{M\times d}$. Let $t=\frac{(\alpha(X_i) - \alpha_{*x}(X_i))'\nabla_\alpha W(\hat\alpha(X_i),Y_i)'r_{*x}(X_i)}{1+W^{\mathrm{FO}}(\alpha_{*x}(X_i),Y_i)'r_{*x}(X_i)}$. Then $Q^{\mathrm{FO}}(\alpha,r_{*x})-Q^{\mathrm{FO}}(\alpha_{*x},r_{*x})=\frac{1}{N}\sum_{i=1}^NK_h(X_i-x) \log\left( 1 + t\right)$. By (1) applying Taylor expansions with respect to $W$, $\nabla_{\alpha}W$, (2) note that $\hat{\alpha}(X_i)\in\mathcal{A}_x(\hat{\alpha})$, (3) $\|r_{*x}(X_i)-r_0(X_i)\|_1\leq\bar{K}sdh^2$ for some universal constant $\bar{K}$, then following similar proof steps in Lemma~\ref{lemma:MinMax:Claim1}, it is straightforward to check that when $N$ is sufficiently large, $\frac{(\alpha(X_i) - \alpha_{*x}(X_i))'\nabla_\alpha W(\hat\alpha(X_i),Y_i)'r_{*x}(X_i)}{1+W^{\mathrm{FO}}(\alpha_{*x}(X_i),Y_i)'r_{*x}(X_i)}=\frac{(\alpha(X_i) - \alpha_{*x}(X_i))'\nabla_\alpha W(\alpha_0(X_i),Y_i)'r_{*x}(X_i)}{1+W(\alpha_0(X_i),Y_i)'r_{*x}(X_i)} +\widetilde R(\alpha,\hat\alpha,Y_i)$, where $\left|\frac{(\alpha(X_i) - \alpha_{*x}(X_i))'\nabla_\alpha W(\alpha_0(X_i),Y_i)'r_{*x}(X_i)}{1+W(\alpha_0(X_i),Y_i)'r_{*x}(X_i)} +\widetilde R(\alpha,\hat\alpha,Y_i)\right|\leq c_0$ for an absolute constant $c_0>0$. Furthermore note that $\log(1+t)\geq t - \frac{t^2}{1+t}$, so 
$Q^{\mathrm{FO}}(\alpha,r_{*x})-Q^{\mathrm{FO}}(\alpha_{*x},r_{*x})\geq \Delta q_1+\Delta q_2-\Delta q_3$, where $\Delta q_1=\frac{1}{N}\sum_{i=1}^NK_h(X_i-x)\frac{(\alpha(X_i) - \alpha_{*x}(X_i))'\nabla_\alpha W(\alpha_0(X_i),Y_i)'r_{*x}(X_i)}{1+W(\alpha_0(X_i),Y_i)'r_{*x}(X_i)}$, $\Delta q_2=\frac{1}{N}\sum_{i=1}^N K_h(X_i-x)\widetilde R(\alpha,\hat\alpha,Y_i)$, $\Delta q_3=\frac{1/c_0}{N}\sum_{i=1}^N K_h(X_i-x)\left|\frac{(\alpha(X_i) - \alpha_{*x}(X_i))'\nabla_\alpha W(\alpha_0(X_i),Y_i)'r_{*x}(X_i)}{1+W(\alpha_0(X_i),Y_i)'r_{*x}(X_i)} +\widetilde R(\alpha,\hat\alpha,Y_i))\right|^2$. In the following we focus on bounding $|\Delta q_1|$, $|\Delta q_2|$ and $|-\Delta q_3|$. Note that $\Delta q_1=\Delta q_{11}+\Delta q_{12}$, where $\displaystyle\Delta q_{11}:=\frac{1}{N}\sum_{i=1}^N K_h(X_i-x)\frac{(\alpha(X_i) - \alpha_{*x}(X_i))'\nabla_\alpha W(\alpha_0(X_i),Y_i)'r_0(X_i)}{1+W(\alpha_0(X_i),Y_i)'r_0(X_i)}$ and 
$$\begin{array}{rl}
\displaystyle\Delta q_{12}\!:=\!\frac{1}{N}\!\!\sum_{i=1}^N\!\!K_h(X_i-x)(\alpha(X_i) - \alpha_{*x}(X_i))'\!\!\left\{\!\frac{\nabla_\alpha W(\alpha_0(X_i),Y_i)'r_{*x}(X_i)}{1+W(\alpha_0(X_i),Y_i)'r_{*x}(X_i)}\! -\! \frac{ \nabla_\alpha W(\alpha_0(X_i),Y_i)'r_0(X_i)}{1+W(\alpha_0(X_i),Y_i)'r_0(X_i)}\!\right\}
\end{array}$$
So with probability $1-C\delta-C(\bar{K}_{\max}/\bar{K}_2)^2/(Nh^d)$,
$$\begin{array}{rl}
&\quad\displaystyle|\Delta q_{11}|:= \left|\frac{1}{N}\sum_{i=1}^N K_h(X_i-x)\frac{(\alpha(X_i) - \alpha_{*x}(X_i))'\nabla_\alpha W(\alpha_0(X_i),Y_i)'r_0(X_i)}{1+W(\alpha_0(X_i),Y_i)'r_0(X_i)} \right|\\
&\displaystyle\leq_{(1)}\| (\tilde{\alpha}-\alpha_0(x),h\cdot\mathrm{vec}(\tilde{\beta}-\nabla \alpha_0(x))\|_1\left\| \frac{1}{N}\sum_{i=1}^N\frac{K_h(X_i-x)\left( { \nabla_\alpha W(\alpha_{0}(X_i),Y_i)'r_0(X_i)}\atop{ {\rm vec}( \nabla_\alpha W(\alpha_{0}(X_i),Y_i)'r_0(X_i)\frac{(X_i-x)'}{h} )}\right) }{1+W(\alpha_{0}(X_i),Y_i)^\prime r_0(X_i)}\right\|_\infty \\
&\displaystyle\leq_{(2)}C'\left(\norm{\tilde{\alpha}-\alpha_0(x)}_1+h\norm{\tilde{\beta}-\nabla\alpha_0(x)}_{1,1}\right)\sqrt{\frac{\log (M/\delta)}{N{h^d}}}\leq_{(3)}C\left(\frac{Md}{\sqrt{Nh^{d}}}+Mdh^2\right)\sqrt{\frac{\log (M/\delta)}{N{h^d}}}.
\end{array}$$
Next, we show that (1) and (2) hold above. 
To show that (1) holds above, note that $\alpha(X_i)=\tilde{\alpha}+\tilde{\beta}(X_i-x)$ where $\alpha(X_i)\in\mathcal{A}_x(\hat{\alpha})$, and recall that $\alpha_{*x}(X_i)=\alpha_0(x)+\nabla\alpha_0(x)(X_i-x)$, thus 
$$\begin{array}{rl}
&\displaystyle\quad\left|\frac{1}{N}\sum_{i=1}^N K_h(X_i-x)\frac{(\alpha(X_i) - \alpha_{*x}(X_i))'\nabla_\alpha W(\alpha_0(X_i),Y_i)'r_0(X_i)}{1+W(\alpha_0(X_i),Y_i)'r_0(X_i)}\right|\\
&\displaystyle=\left|\frac{1}{N}\sum_{i=1}^N K_h(X_i-x)\frac{[(\tilde{\alpha}-\alpha_0(x))+(\tilde{\beta}-\nabla\alpha_0(x))(X_i-x)]'\nabla_\alpha W(\alpha_0(X_i),Y_i)'r_0(X_i)}{1+W(\alpha_0(X_i),Y_i)'r_0(X_i)}\right|\\
&\displaystyle\leq_{(a)}\| (\tilde{\alpha},h\mathrm{vec}(\tilde{\beta}))-(\alpha_0(x),h\mathrm{vec}(\nabla \alpha_0(x)))\|_1\\ 
&\displaystyle\quad\times\left\| \frac{1}{N}\sum_{i=1}^N\frac{  K_h(X_i-x) }{1+W(\alpha_{0}(X_i),Y_i)^\prime r_0(X_i)} \left( { \nabla_\alpha W(\alpha_{0}(X_i),Y_i)'r_0(X_i)}\atop{ {\rm vec}( \nabla_\alpha W(\alpha_{0}(X_i),Y_i)'r_0(X_i)\frac{(X_i-x)'}{h} )}\right) \right\|_\infty, 
\end{array}$$
where inequality (a) here holds by applying H\"older's inequality. Additionally, inequality (2) holds with probability $1-C\delta-C(\bar{K}_{\max}/\bar{K}_2)^2/(Nh^d)$ by the self-normalized deviation theory and Lemma \ref{lemma:covariate:score*approx}. The constant $C'$ in (2) is a universal constant independent of $N,M,s,d$. And inequality (3) holds by Proposition \ref{prop:concentration:covariate}. Denote $\displaystyle\Delta^{*}(X_i,Y_i,x):=\frac{ \nabla_\alpha W(\alpha_0(X_i),Y_i)'r_{*x}(X_i)}{1+W(\alpha_0(X_i),Y_i)'r_{*x}(X_i)} - \frac{ \nabla_\alpha W(\alpha_0(X_i),Y_i)'r_0(X_i)}{1+W(\alpha_0(X_i),Y_i)'r_0(X_i)}$. Then for universal constants $\bar{C},\bar{C}'$, $C$,
$$\begin{array}{rl}
&\displaystyle\quad|\Delta q_{12}|= \left|\frac{1}{N}\sum_{i=1}^N K_h(X_i-x)(\alpha(X_i) - \alpha_{*x}(X_i))'\Delta^{*}(X_i,Y_i,x)\right|\\
&\displaystyle\leq_{(1)}\norm{(\tilde{\alpha},h\mathrm{vec}(\tilde{\beta}))-(\alpha_0(x),h\mathrm{vec}(\nabla\alpha_0(x)))}_1\bigg\| \frac{1}{N}\sum_{i=1}^NK_h(X_i-x)\left({\Delta^{*}(X_i,Y_i,x)}\atop{ {\rm vec}\left(\Delta^{*}(X_i,Y_i,x)\frac{(X_i-x)'}{h} \right)}\right)\bigg\|_\infty\\
&\displaystyle\leq_{(2)}\bar{C}'\left(\frac{Md}{\sqrt{Nh^{d}}}+Mdh^2\right)sdh^2\bigg(1 \!+\!\frac{\sqrt{1/\delta}\log M}{Nh^d}\!\bigg)
\leq_{(3)} Csdh^2\left(\frac{Md}{\sqrt{Nh^{d}}}+Mdh^2\right),
\end{array}$$
where inequality (1) for $|\Delta q_{12}|$ holds by the fact that $|(X_{ik}-x_k)/h|\leq1$ for any $k\in[d]$ and then using similar argument as for inequality (1) of $|\Delta q_{11}|$. Inequality (2) for $|\Delta q_{12}|$ holds by Proposition \ref{prop:concentration:covariate} and Lemma \ref{Lemma:xi-bound}, and (3) holds since $\delta\geq\log(M)/Nh^d$. $|\Delta q_2|$ is bounded via $\frac{1}{N}\sum_{i=1}^N K_h(X_i-x)\widetilde R(\alpha,\hat\alpha,Y_i)\leq\frac{C}{N}\sum_{i=1}^N K_h(X_i-x)\max_{\alpha \in \mathcal{A}_x(\hat\alpha), \|X_i-x\|\leq h}\| \alpha(X_i)-\alpha_{*x}(X_i)\|^2$. Finally, following similar argument as before, $|-\Delta q_3|$ is bounded by $\frac{C'}{N}\sum_{i=1}^N K_h(X_i-x)[\max_{\alpha \in \mathcal{A}_x(\hat\alpha), \|X_i-x\|\leq h}\|\alpha(X_i)-\alpha_0(X_i)\|^2  + |\widetilde R(\alpha,\hat\alpha,Y_i)|^2]\leq C_2\max_{\tilde\alpha \in \mathcal{A}_x(\hat\alpha), \|X_i-x\|\leq h}\|\tilde\alpha(X_i)-\alpha_0(X_i)\|^2$ for some universal constant $C_2$ according to Assumption~\ref{ass:smoothness-assumptions}. Thus the lemma holds by combining all the results above.
\proofend

\begin{lemma}\label{Lemma:xi-bound}
Suppose Assumptions \ref{ass:iid-min-den}, \ref{ass:smoothness-assumptions} hold, and $sdh^2 = \delta_N\xi$ for a sequence $\delta_N$ that goes zero. Then for $N$ sufficiently large, for some universal constant $C$ independent of $s,N,M,d$, with probability $1-C\epsilon-C\delta-C\delta_{\alpha}$, $\left\|\frac{1}{N}\sum_{i=1}^NK_h(X_i-x)\left[\frac{\nabla_{\alpha}W(\alpha_0(X_i),Y_i)'r_{*x}(X_i)}{1+W(\alpha_0(X_i),Y_i)'r_{*x}(X_i)}-\frac{\nabla_{\alpha}W(\alpha_0(X_i),Y_i)'r_0(X_i)}{1+W(\alpha_0(X_i),Y_i)'r_0(X_i)}\right]\right\|_\infty\leq Csdh^2\bigg(1 \!+\!\frac{\sqrt{1/\delta}\log M}{Nh^d}\!\bigg)$.
\end{lemma}
\proof{Proof of Lemma \ref{Lemma:xi-bound}}
Firstly note that $K(\cdot)$ is supported on the unit ball, so we only need to consider the covariates $X_i$ such that $\norm{X_i-x}_2\leq h$. The proof of Theorem \ref{thm:main:covariate} implies that  $\norm{r_0(X_i)-r_0(x)-\nabla r_0(x)(X_i-x)}_1\leq \bar{K}_{x,2}sdh^2$. Furthermore, using H\"older's inequality, for any $k\in[M]$ we have $\left|\nabla_{\alpha_k} W(\alpha_0(X_i),Y_i)^\prime\left(r_0(X_i)-r_{*x}(X_i)\right)\right|\leq \bar{K}_{x,2}\bar{K}_wsdh^2$,
\begin{equation}\label{eq:feasible-r*}
\begin{array}{rl}
&\displaystyle\quad1+W(\alpha_0(X_i),Y_i)^\prime r_{*x}(X_i)\geq1+W(\alpha_0(X_i),Y_i)^\prime r_0(X_i)-|W(\alpha_0(X_i),Y_i)^\prime\left(r_0(X_i)-r_{*x}(X_i)\right)|\\
&\geq_{(a)}\xi-\bar{K}_{x,2}\bar{K}_wsdh^2\geq_{(b)} C_1^\prime\xi,
\end{array}
\end{equation}
where (a) holds with probability $1-\epsilon$ by Assumption \ref{ass:iid-min-den} and (b) holds with $C_1'\in(0,1)$ for $N$ sufficiently large since $sdh^2\leq\delta_N\xi$ where $\delta_N\rightarrow0$. Given $i\in[N], k\in[M]$, define $\Delta^{*}(X_i,Y_i,x):=\frac{\nabla_{\alpha} W(\alpha_0(X_i),Y_i)' r_{*x}(X_i)}{1+W(\alpha_0(X_i),Y_i)^\prime r_{*x}(X_i)}-\frac{\nabla_{\alpha} W(\alpha_0(X_i),Y_i)'r_0(X_i)}{1+W(\alpha_0(X_i),Y_i)^\prime r_0(X_i)}$, $\Delta_0r:=\frac{r_{*x}(X_i)}{1+W(\alpha_0(X_i),Y_i)^\prime r_{*x}(X_i)}-\frac{r_0(X_i)}{1+W(\alpha_0(X_i),Y_i)^\prime r_0(X_i)}$. Then 
\begin{equation}\label{eq:target:bound}
\left\|\frac{1}{N}\sum_{i=1}^NK_h(X_i-x)\Delta^{*}(X_i,Y_i,x)\right\|_\infty=\left\|\frac{1}{N}\sum_{i=1}^NK_h(X_i-x)\nabla_{\alpha} W(\alpha_0(X_i),Y_i)'\Delta_0r\right\|_\infty.
\end{equation}
Note that $\Delta_0r=\frac{r_{*x}(X_i)-r_0(X_i)}{(1+W(\alpha_0(X_i),Y_i)'\tilde{r}_i)}-\frac{\tilde{r}_iW(\alpha_0(X_i),Y_i)'(r_{*x}(X_i)-r_0(X_i))}{(1+W(\alpha_0(X_i),Y_i)'\tilde{r}_i)^2}$ for some $\tilde{r}_i$ on the line segment between $r_{*x}(X_i), r_0(X_i)$. Following similarly from \eqref{eq:feasible-r*}, $1+W(\alpha_0(X_i),Y_i)'\tilde{r}_i>C_1'\xi$ for some absolute constant $C_1'$ for $N$ sufficiently large. The condition of this lemma implies $\|\nabla_{\alpha}W(\alpha_0(X_i),Y_i)'(r_{*x}(X_i)-r_0(X_i))\|_{\infty}\leq \bar{K}_w\|r_{*x}(X_i)-r_0(X_i)\|_1\leq\bar{K}_{x,2}\bar{K}_wsdh^2$. Define $\Delta_kr:=\nabla_{\alpha_k}W(\alpha_0(X_i),Y_i)'\Delta_0r$. Assumption~\ref{ass:iid-min-den} and Assumption~\ref{ass:smoothness-assumptions} imply that $|\Delta_kr|\leq\bar{K}sdh^2$ for some universal constant $\bar{K}$. Define $\bar \sigma^2:=\max_{k\in[M]}\sum_{i=1}^N\mathbb{E}\left[K_h(X_i-x)^2\{\Delta_kr\}^2\right]$. So for some universal constant $C_2$, $\bar{\sigma}^2\leq C_2Ns^2d^2h^4/h^d$. Define $\bar{M}:=  \max_{i\in[N],k\in[M]}K_h(X_i-x)|\nabla_{\alpha_k}W(\alpha_0(X_i),Y_i)'\Delta_0r|$. Then $\displaystyle\bar{M}\leq C_2sd/h^{d-2}$.
Lemma 8 of \citet{chernozhukov2015comparison} implies $\mathbb{E}\left[\max_{k\in[M]}\left|\sum_{i=1}^NK_h(X_i-x)\Delta_kr-\mathbb{E}[K_h(X_i-x)\Delta_kr]\right|\right]\leq\bar{C}_1\left(\bar{\sigma}\sqrt{\log M}+\mathbb{E}[\bar{M}^2]^{1/2}\log M\right)$ for some universal constant $\bar{C}_1$. Since $|\Delta_kr|\leq\bar{K}sdh^2$, $\mathbb{E}[\max_{k\in[M]}|\sum_{i=1}^NK_h(X_i-x)\Delta_kr|]\leq\bar{C}_2\left(\bar{\sigma}\sqrt{\log M}+\mathbb{E}[\bar{M}^2]^{1/2}\log M+Nsdh^2\right)$ for some unviersal constant $\bar{C}_2$. Furthermore using Lemma E.4 of \citet{chernozhukov2017central}, for some universal constant $\bar{C}$, 
\begin{equation}\label{eq:bound:term-1}
\max_{k\in[M]}\left|\frac{1}{N}\sum_{i=1}^NK_h(X_i-x)\Delta_kr\right|\leq2\bar{C}_2\bigg\{\frac{\bar\sigma \sqrt{\log M}+\sqrt{\mathbb{E}[\bar{M}^2]}\log(M)}{N}+sdh^2\bigg\} + \frac{t}{N}\leq\bar{C}sdh^2 + \frac{t}{N}
\end{equation}
with probability $1-\frac{\bar{C}\mathbb{E}[\bar M^2]}{t^2}$. Now we take $t$ in (\ref{eq:bound:term-1}) such that $\frac{\bar{C}\mathbb{E}[\bar M^2]}{t^2}=\delta$.  So with probability $1-\delta$, 
\begin{equation}\label{eq:term-1:bound:new}
\max_{k\in[M]}\left|\frac{1}{N}\sum_{i=1}^NK_h(X_i-x)\Delta_kr\right|\leq Csdh^2\bigg(1 \!+\!\frac{\sqrt{1/\delta}\log M}{Nh^d}\!\bigg). 
\end{equation}
Recall $\displaystyle\Delta^{*}(X_i,Y_i,x)=\frac{ \nabla_\alpha W(\alpha_0(X_i),Y_i)'r_{*x}(X_i)}{1+W(\alpha_0(X_i),Y_i)'r_{*x}(X_i)} - \frac{ \nabla_\alpha W(\alpha_0(X_i),Y_i)'r_0(X_i)}{1+W(\alpha_0(X_i),Y_i)'r_0(X_i)}$.  (\ref{eq:target:bound}), (\ref{eq:term-1:bound:new}) imply that with probability $1-\delta-\epsilon$, $\left\|\frac{1}{N}\sum_{i=1}^NK_h(X_i-x)\Delta^{*}(X_i,Y_i,x)\right\|_\infty\!\!\!\!\!\leq Csdh^2\bigg(1 \!+\!\frac{\sqrt{1/\delta}\log M}{Nh^d}\!\bigg)$.
\proofend

\section{Rates of Convergence for Marginal Probabilities}\label{appendix:marginal}
In this section, we provide rates of convergence for the nuisance parameters or functions, i.e., marginal probabilities. Proposition \ref{prop:concentration:no covariate} and Proposition~\ref{prop:concentration:covariate} provide results for the no-covariate case and with-covariate case. The proofs for the following two propositions are standard by applying the concentration inequalities for $\hat{\alpha}$ defined as \eqref{eq:SAA}, $\hat{\alpha}(x)$ defined as \eqref{eq:truncate-local-linear}, $\tilde\beta(x)$ defined as in Section~\ref{sec:marginals}, following from \citep{boucheron2003concentration,einmahl2005uniform,hansen2008uniform,fan1993local}, so we omit them here. 
\begin{proposition}\label{prop:concentration:no covariate} In the case of no covariates, suppose Assumption \ref{ass:iid-min-den} holds. Let $\hat{\alpha}$ and $\mathcal{A}(\hat\alpha)$ be the estimator and confidence set defined in \eqref{eq:SAA} and \eqref{def:A:no-covariate} with $\delta \in (0,1)$. Then with probability at least $1-\delta$, we have $\norm{\hat{\alpha}-\alpha_0}_{\infty}\leq c_1\sqrt{\frac{\log(M/\delta)}{N}}$, $\norm{\hat\alpha-\alpha_0}_1\leq 
c_1\left(\frac{M}{\sqrt{N}}+\sqrt{\frac{M\log(1/\delta)}{N}}\right)$, $\norm{\hat\alpha-\alpha_0}_2\leq 
c_1\left(\sqrt{\frac{M}{N}}+\sqrt{\frac{\log(1/\delta)}{N}}\right)$, 
where $c_1$ is an absolute constant independent of $M,N,s$. 
\end{proposition}

\begin{proposition}\label{prop:concentration:covariate}
Suppose Suppose Assumptions \ref{ass:iid-min-den} (iv)--(v) and Assumption \ref{ass:local-linear} holds and let $x$ be a given data point of in the interior of $\mathcal{X}$. Let $\hat{\alpha}(x)$ be the truncated local linear estimator defined by (\ref{eq:truncate-local-linear}), and let 
$\delta\geq\exp(-\bar{C}_0\{\sqrt{\log M}\sqrt{Nh^d}+Nh^d\})$. Then for some constants $\bar{C}_0$, $\bar{C}_1$ which depend on $x$ and are independent of $N,d,h,s,M$, with probability $1-\delta$, we have $\norm{\hat{\alpha}(x)-\alpha_0(x)}_1\leq\bar{C}_0\left(\frac{M}{\sqrt{Nh^d}}+Mdh^2+\sqrt{M\left(\frac{1}{Nh^d}+h^4d^2\right)\log(\frac{1}{\delta})}\right)$, $\norm{\hat{\alpha}(x)-\alpha_{0}(x)}_{\infty}\!\!\!\leq\!\!\!\bar{C}_0\left(dh^2+\sqrt{\frac{\log(M/\delta)}{Nh^d}}\right)$, $\norm{\hat{\alpha}(x)-\alpha_0(x)}_2\!\leq\!\bar{C}_0\left(\sqrt{\frac{M}{Nh^d}}+\sqrt{M}dh^2+\sqrt{\left(\frac{1}{Nh^d}+h^4d^2\right)\log(\frac{1}{\delta})}\right)$, \\
$\norm{\tilde{\beta}(x)-\nabla\alpha_0(x)}_{1,1}\leq\bar{C}_1d\left(\frac{M}{\sqrt{Nh^{d+2}}}+\!\!Mdh^2\!\!+\!\!\sqrt{M\left(\frac{1}{Nh^{d+2}}+h^4d^2\right)\log(\frac{1}{\delta})}\right)$,\\
$\norm{\tilde{\beta}(x)-\nabla\alpha_0(x)}_{2,2}\leq\bar{C}_1\sqrt{d}\left(\sqrt{\frac{M}{Nh^{d+2}}}+\sqrt{M}dh^2+\!\!\sqrt{\left(\frac{1}{Nh^{d+2}}+h^4d^2\right)\log(\frac{1}{\delta})}\right)$.
\end{proposition}

\subsection{Proof of Theorem \ref{thm:causal-inference:AIPW}}\label{sec:results:ATE}
For any fold $k\in[K]$ and any $t,t'$, define $$\hat{\eta}_{(t)}^{\mathcal{I}_k}:=\frac{1}{|\mathcal{I}_k|}\sum_{i\in\mathcal{I}_k}\left(\hat{\mu}_{(t)}^{-k}(X_i)+\mathbf{1}\{T_i=t\}\frac{O_i-\hat{\mu}_{(t)}^{-k}(X_i)}{\hat{e}_t^{-k}(X_i)}\right),$$ $$\hat{\eta}_{(t')}^{\mathcal{I}_k}:=\frac{1}{|\mathcal{I}_k|}\sum_{i\in\mathcal{I}_k}\left(\hat{\mu}_{(t')}^{-k}(X_i)+\mathbf{1}\{T_i=t'\}\frac{O_i-\hat{\mu}_{(t')}^{-k}(X_i)}{\hat{e}_{t'}^{-k}(X_i)}\right),$$ $$\hat{\eta}_{(t)}^{\mathcal{I}_k,*}:=\frac{1}{|\mathcal{I}_k|}\sum_{i\in\mathcal{I}_k}\left({\mu}_{(t)}(X_i)+\mathbf{1}\{T_i=t\}\frac{O_i-{\mu}_{(t)}(X_i)}{{e}_t(X_i)}\right),$$ $$\hat{\eta}_{(t')}^{\mathcal{I}_k,*}:=\frac{1}{|\mathcal{I}_k|}\sum_{i\in\mathcal{I}_k}\left({\mu}_{(t')}(X_i)+\mathbf{1}\{T_i=t'\}\frac{O_i-{\mu}_{(t')}(X_i)}{{e}_{t'}(X_i)}\right).$$ 

Define $\hat{\tau}^{\mathcal{I}_k}(t,t'):=\hat{\eta}_{(t)}^{\mathcal{I}_k}-\hat{\eta}_{(t')}^{\mathcal{I}_k}$, $\hat{\tau}^{\mathcal{I}_k,*}(t,t'):=\hat{\eta}_{(t)}^{\mathcal{I}_k,*}-\hat{\eta}_{(t')}^{\mathcal{I}_k,*}$, $\hat{\tau}_{\textrm{AIPW}}^*:=\sum_{k\in[K]}\frac{|\mathcal{I}_k|}{N}\hat{\tau}^{\mathcal{I}_k,*}(t,t')$. Note that $\hat{\tau}_{\textrm{AIPW}}=\sum_{k\in[K]}\frac{|\mathcal{I}_k|}{N}\hat{\tau}^{\mathcal{I}_k}(t,t')$, so 
$$\hat{\tau}_{\textrm{AIPW}}^*=\frac{1}{N}\sum_{i=1}^N \mu_{(t)}(X_i)-\mu_{(t')}(X_i)+ \frac{\mathbf{1}\{T_i=t\}}{e_t(X_i)}(O_i - \mu_{(t)}(X_i)) - \frac{1\{T_i=t'\}}{e_{t'}(X_i)}(O_i - \mu_{(t')}(X_i)).$$ 
To show $\sqrt{N}(\hat{\tau}_{\textrm{AIPW}}-\hat{\tau}_{\textrm{AIPW}}^*)=\mathrm{o}_{\mathrm{P}}(1)$, it suffices to show that for any $k\in[K]$, 
$$\sqrt{N}(\hat{\tau}^{\mathcal{I}_k}(t,t')-\hat{\tau}^{\mathcal{I}_k,*}(t,t'))=\mathrm{o}_{\mathrm{P}}(1).$$ 
Hence it suffices to show that for any $t\in\{0,1\}^M$, $\sqrt{N}(\hat{\eta}_{(t)}^{\mathcal{I}_k}-\hat{\eta}_{(t)}^{\mathcal{I}_k,*})=\mathrm{o}_{\mathrm{P}}(1)$. Note that  $$\hat{\eta}_{(t)}^{\mathcal{I}_k}-\hat{\eta}_{(t)}^{\mathcal{I}_k,*}=(a)+(b)+(c),$$ where 
$$(a)=\frac{1}{|\mathcal{I}_k|}\sum_{i\in\mathcal{I}_k}\big\{\big(\hat{\mu}_{(t)}^{-k}(X_i)-\mu_t(X_i)\big)\big(1-\frac{\mathbf{1}\{T_i=t\}}{e_t(X_i)}\big)\big\},$$ 
$$(b)=\frac{1}{|\mathcal{I}_k|}\sum_{i\in\mathcal{I}_k}\mathbf{1}\{T_i=t\}\left((O_i-\mu_{(t)}(X_i))\left(\frac{1}{\hat{e}_t^{-k}(X_i)}-\frac{1}{e_t(X_i)}\right)\right),$$ 
$$(c)=-\frac{1}{|\mathcal{I}_k|}\sum_{i\in\mathcal{I}_k}\mathbf{1}\{T_i=t\}\left(\left(\hat{\mu}_{(t)}^{-k}(X_i)-\mu_t(X_i)\right)\left(\frac{1}{\hat{e}_t^{-k}(X_i)}-\frac{1}{e_t(X_i)}\right)\right).$$ 
Further note that 
$$\begin{array}{rl}
&\displaystyle\mathbb{E}\bigg[\bigg(\frac{1}{|\mathcal{I}_k|}\sum_{i\in\mathcal{I}_k}\bigg\{\bigg(\hat{\mu}_{(t)}^{-k}(X_i)-\mu_t(X_i)\bigg)\bigg(1-\frac{\mathbf{1}\{T_i=t\}}{e_t(X_i)}\bigg)\bigg\}\bigg)^2\bigg]\\
&\displaystyle=\frac{1}{|\mathcal{I}_k|}\mathbb{E}\left[\mathbb{E}\left[\left(\hat{\mu}_{(t)}^{-k}(X_i)-\mu_t(X_i)\right)^2\left(\frac{1}{e_t(X_i)}-1\right)\big|\mathcal{I}^{-k}\right]\right]\\
&\displaystyle\leq\frac{\mathbb{E}\big[\big(\hat{\mu}_{(t)}^{-k}(X_i)-\mu_t(X_i)\big)^2\big]/(c|\mathcal{I}_k|)}{\inf_{x\in\mathcal{X}}\prod_{j=1}^M \alpha_{0j}^{t_j}(x)(1-\alpha_{0j}(x))^{1-t_j}}\leq\frac{\tilde{C}\epsilon_N^2N^{-3/2}}{\inf_{x\in\mathcal{X}}\prod_{j=1}^M \alpha_{0j}^{t_j}(x)(1-\alpha_{0j}(x))^{1-t_j}},
\end{array}$$
where the last inequality follows from condition (iii) of Assumption~\ref{ass:causal:extra}. Further, define $\Delta E:=\mathbb{E}\left[\left(\frac{1}{|\mathcal{I}_k|}\sum_{i\in\mathcal{I}_k}\mathbf{1}\{T_i=t\}\bigg(O_i-\mu_{(t)}(X_i)\bigg)\bigg(\frac{1}{\hat{e}_t^{-k}(X_i)}-\frac{1}{e_t(X_i)}\bigg)\right)^2\right]$. Conditional on the out-of-fold sample, the fold-k summands are independent and have mean zero since $\mathbb{E}[O_i-\mu_{(t)}(X_i)|X_i,T_i=t]=0$, so
it is straightforward to check that $\Delta E\leq\frac{C}{|\mathcal{I}_k|}\mathbb{E}\left[\left(1-\frac{e_t(X_i)}{\hat{e}_t^{-k}(X_i)}\right)^2\frac{1}{e_t(X_i)}\right]$ for some absolute constant $C$. For any $x\in\mathcal{X}$, $\hat{e}_t(x)-e_t(x)=\textrm{(I)+(II)+(III)}$,
where 
$$\begin{array}{rcl}
\textrm{(I)}&=&\displaystyle\left[\Big(1+W(\hat\alpha(x),t)'\hat r(x,x)\Big)-\Big(1+W(\alpha_0(x),t)'\hat r(x,x)\Big)\right]\prod_{j=1}^M\hat\alpha_j(x)^{t_j}(1-\hat\alpha_j(x))^{1-t_j},\\
\textrm{(II)}&=&\displaystyle\left[\Big(1+W(\alpha_0(x),t)'\hat r(x,x)\Big)-\Big(1+W(\alpha_0(x),t)'r_0(x)\Big)\right]\prod_{j=1}^M\hat\alpha_j(x)^{t_j}(1-\hat\alpha_j(x))^{1-t_j},\\
\textrm{(III)}&=&\displaystyle\Big(1+W(\alpha_0(x),t)'r_0(x)\Big)\left[\prod_{j=1}^M\hat\alpha_j(x)^{t_j}(1-\hat\alpha_j(x))^{1-t_j}-\prod_{j=1}^M\alpha_{0j}(x)^{t_j}(1-\alpha_{0j}(x))^{1-t_j}\right].
\end{array}$$
For any $j\in[M]$, define $\tilde{\Delta}_j\alpha_0(x):=\alpha_{0j}(x)^{t_j}(1-\alpha_{0j}(x))^{1-t_j}$, $\tilde{\Delta}_j\hat{\alpha}(x):=\hat{\alpha}_j(x)^{t_j}(1-\hat{\alpha}_j(x))^{1-t_j}$, then $$\prod_{j=1}^M\tilde{\Delta}_j\hat{\alpha}(x)-\prod_{j=1}^M\tilde{\Delta}_j\alpha_0(x)=[\tilde{\Delta}_1\hat{\alpha}(x)-\tilde{\Delta}_1\alpha_0(x)]\prod_{j=2}^M\tilde{\Delta}_j\hat{\alpha}(x)+\tilde{\Delta}_1\alpha_0(x)\bigg\{\prod_{j=2}^M\tilde{\Delta}_j\hat{\alpha}(x)-\prod_{j=2}^M\tilde{\Delta}_j\alpha_0(x)\bigg\}.$$ Denote 
$$\Delta_k := \prod_{j=k}^M\hat{\alpha}_j(x)^{t_j}(1-\hat{\alpha}_j(x))^{1-t_j}-\prod_{j=k}^M\alpha_{0j}(x)^{t_j}(1-\alpha_{0j}(x))^{1-t_j},$$ $$\hat{\delta}_j:=\hat{\alpha}_j(x)^{t_j}(1-\hat{\alpha}_j(x))^{1-t_j}-\alpha_{0j}(x)^{t_j}(1-\alpha_{0j}(x))^{1-t_j},$$ $$\hat{\pi}_k:=\prod_{j=k}^M\hat{\alpha}_j(x)^{t_j}(1-\hat{\alpha}_j(x))^{1-t_j},m_j:=\alpha_{0j}(x)^{t_j}(1-\alpha_{0j}(x))^{1-t_j}.$$ 
So by expanding the recursive formula of $\prod_{j=1}^M\tilde{\Delta}_j\hat{\alpha}(x)-\prod_{j=1}^M\tilde{\Delta}_j\alpha_0(x)$, we have $\Delta_1=\prod_{j=1}^M\hat{\alpha}_j(x)^{t_j}(1-\hat{\alpha}_j(x))^{1-t_j}-\prod_{j=1}^M\alpha_{0j}(x)^{t_j}(1-\alpha_{0j}(x))^{1-t_j}=\hat{\delta}_1\hat{\pi}_2+m_1\Delta_2=\hat{\delta}_1\hat{\pi}_2+m_1\hat{\delta}_2\hat{\pi}_3+m_1m_2\Delta_3=\ldots=\hat{\delta}_1\hat{\pi}_2+\sum_{j=2}^{M-1}\hat{\delta}_j\hat{\pi}_{j+1}\prod_{k=1}^{j-1}m_k+\hat{\delta}_M\prod_{k=1}^{M-1}m_k$. By noting that $0\leq\hat{\pi}_k\leq 1$ and $0\leq m_j\leq 1$ for any $j,k\in[M]$, and recall that $t_j\in\{0,1\}$, thus $|\Delta_1|\leq\sum_{j=1}^M|\hat{\delta}_j|=\sum_{j=1}^M|\hat{\alpha}_j(x)-\alpha_{0j}(x)|=\norm{\hat{\alpha}(x)-\alpha_0(x)}_1\leq CMdh^2+C\frac{M}{\sqrt{Nh^d}}$ with probability $1-\delta$ for $\delta\geq\exp(-C\{\sqrt{\log M}\sqrt{Nh^d}+Nh^d\})$ for some absolute constant $C$ according to Proposition~\ref{prop:concentration:covariate}. The above inequality and Taylor expansion imply $|\textrm{(I)}|\leq |(\hat{\alpha}(x)-\alpha_0(x))'\nabla_\alpha W(\tilde{\alpha}(x),t)'\hat{r}(x,x)|\prod_{j=1}^M\hat\alpha_j(x)^{t_j}(1-\hat\alpha_j(x))^{1-t_j}$. 

Since $\prod_{j=1}^M\hat\alpha_j(x)^{t_j}(1-\hat\alpha_j(x))^{1-t_j}\leq\prod_{j=1}^M\alpha_{0j}(x)^{t_j}(1-\alpha_{0j}(x))^{1-t_j}+\norm{\hat{\alpha}(x)-\alpha_0(x)}_1$, we have 
$$\begin{array}{rl}
|\textrm{(I)}|&\displaystyle\leq_{(1)} C'\left(\norm{\hat{\alpha}(x)-\alpha_0(x)}_1^2+\norm{\hat{\alpha}(x)-\alpha_0(x)}_1\sup_{x\in\mathcal{X}}\prod_{j=1}^M\alpha_{0j}(x)^{t_j}(1-\alpha_{0j}(x))^{1-t_j}\right)\\
&\displaystyle\leq C'\left(\frac{M}{\sqrt{Nh^d}}+Mdh^2\right)^2\!\!\!+C'\left(\frac{M}{\sqrt{Nh^d}}+Mdh^2\right)\sup_{x\in\mathcal{X}}\prod_{j=1}^M\alpha_{0j}(x)^{t_j}(1-\alpha_{0j}(x))^{1-t_j}\leq_{(2)}C\psi_N,
\end{array}$$ 
where (1) follows from Assumption \ref{ass:smoothness-assumptions}, the fact that $\hat{r}(x,x)=\hat{r}(x,x)-r_0(x)+r_0(x)$ and $\hat{r}(x,x)-r_0(x)$ is negligible by our theoretical result for the with-covariate estimator; (2) follows from (\ref{consistency:ATE:condition}) and (\ref{consistency:ATE:2}). 

To control (II), note 
$$|\textrm{(II)}|\leq |W(\alpha_0(x),t)'(\hat{r}(x,x)-r_0(x))|\prod_{j=1}^M\hat\alpha_j(x)^{t_j}(1-\hat\alpha_j(x))^{1-t_j},$$ 
where the right hand side is equal to $|W(\alpha_0(x),t)'(\hat{r}(x,x)-r_0(x))|\left[\prod_{j=1}^M\alpha_{0j}(x)^{t_j}(1-\alpha_{0j}(x))^{1-t_j}+\Delta_1\right]$. Let $\delta\geq\exp(-C'\{\sqrt{\log(M)Nh^d}+Nh^d\})$. Then by Proposition~\ref{prop:concentration:covariate}, with probability $1-\delta$, $|\Delta_1|\leq C'Mdh^2+C'\frac{M}{\sqrt{Nh^d}}$. Further using H\"older's inequality, Corollary \ref{cor:FO:covariate}, Assumption~\ref{ass:smoothness-assumptions} and $r_0(x)=r_{*x}(x)$, with probability $1-\delta$, we have $|\textrm{(II)}|\leq C_1'\sqrt{s}\left(\!\!\sqrt{\frac{sM}{Nh^d}}+(\sqrt{M}+s)dh^2\right)\!\left(\prod_{j=1}^M\alpha_{0j}(x)^{t_j}(1-\alpha_{0j}(x))^{1-t_j}\!+\!Mdh^2\!+\!\frac{M}{\sqrt{Nh^d}}\right)$ for some universal constant $C_1'$ and $\delta\geq\log(M)/(\delta_NNh^d)$. Additionally, following from (\ref{consistency:ATE:condition}) and \eqref{consistency:ATE:2}, $\textrm{(II)}\leq C_2\psi_N$ for an absolute constant $C_2$. 

Further by Assumption \ref{ass:iid-min-den}, we have $|\textrm{(III)}|\leq C_3\norm{\hat{\alpha}(x)-\alpha_0(x)}_1\leq C_3\psi_N$ for an absolute constant $C_3$, hence $\displaystyle\sup_{x\in\mathcal{X}}|\hat{e}_t(x)-e_t(x)|\leq C\psi_N$, where $C$ is some constant independent of $N,d,h,M,s$. So the upper bound for $\Delta E$ derived earlier implies $\mathbb{E}\left[\left(\frac{1}{|\mathcal{I}_k|}\sum_{i\in\mathcal{I}_k}\mathbf{1}\{T_i=t\}(O_i-\mu_{(t)}(X_i))\left(\frac{1}{\hat{e}_t^{-k}(X_i)}-\frac{1}{e_t(X_i)}\right)\right)^2\right]\leq\frac{CN^{-1}\psi_N^2}{[\inf_{x\in\mathcal{X}}\prod_{j=1}^M \alpha_{0j}^{t_j}(x)(1-\alpha_{0j}(x))^{1-t_j}]^3}$. Recall that $\hat{\eta}_{(t)}^{\mathcal{I}_k}-\hat{\eta}_{(t)}^{\mathcal{I}_k,*}=(a)+(b)+(c)$, where (a), (b), (c) are defined as before. Hence by Chebyshev's inequality, for an absolute constant $\tilde{C}$, we have $\mathbb{P}(\sqrt{N}|(a)|>\tilde{C} N^{-\gamma})\leq\frac{\tilde{C}\epsilon_N^2}{\inf_{x\in\mathcal{X}}\prod_{j=1}^M \alpha_{0j}^{t_j}(x)(1-\alpha_{0j}(x))^{1-t_j}}\leq\frac{\tilde{C}\epsilon_N^2}{c^M}$, $\mathbb{P}(\sqrt{N}|(b)|>\tilde{C}\sqrt{\psi_N})\leq\frac{\tilde{C}\psi_N}{[\inf_{x\in\mathcal{X}}\prod_{j=1}^M \alpha_{0j}^{t_j}(x)(1-\alpha_{0j}(x))^{1-t_j}]^3}\leq\frac{\tilde{C}\psi_N}{c^{3M}}$. Further, 
$$\begin{array}{rl}
&\displaystyle\quad\frac{1}{|\mathcal{I}_k|}\sum_{i\in\mathcal{I}_k}\mathbf{1}\{T_i=t\}\left(\left(\hat{\mu}_{(t)}^{-k}(X_i)-\mu_t(X_i)\right)\left(\frac{1}{\hat{e}_t^{-k}(X_i)}-\frac{1}{e_t(X_i)}\right)\right)\\
&\displaystyle\leq\sqrt{\frac{1}{|\mathcal{I}_k|}\sum_{i:i\in\mathcal{I}_k,T_i=t}\left(\hat{\mu}_{(t)}^{-k}(X_i)-\mu_t(X_i)\right)^2}\sqrt{\frac{1}{|\mathcal{I}_k|}\sum_{i:i\in\mathcal{I}_k,T_i=t}\left(\frac{1}{\hat{e}_t^{-k}(X_i)}-\frac{1}{e_t(X_i)}\right)^2}\\
&\displaystyle\leq_{(1)} \frac{\tilde{C}\epsilon_NN^{-1/2+\gamma}\psi_N}{[\inf_{x\in\mathcal{X}}\prod_{j=1}^M \alpha_{0j}(x)^{t_j}(1-\alpha_{0j}(x))^{1-t_j}]^{2}},\mbox{where (1) holds with probability $1-\epsilon_N$}
\end{array}$$
by Assumption~\ref{ass:causal:extra}. Thus $\sqrt{N}|(c)|\leq\frac{\tilde{C}\epsilon_NN^{\gamma}\psi_N}{[\inf_{x\in\mathcal{X}}\prod_{j=1}^M \alpha_{0j}(x)^{t_j}(1-\alpha_{0j}(x))^{1-t_j}]^2}$ with probability $1-\epsilon_N$, where $N^{\gamma}\psi_N=\mathrm{o}(1)$. Since $M=\mathrm{o}(\min\{\log(1/\psi_N),\log(1/\epsilon_N)\})$, so $\frac{\epsilon_N^2}{c^M}=\mathrm{o}(1), \frac{\psi_N}{c^{3M}}=\mathrm{o}(1), \frac{\epsilon_N}{c^{2M}}=\mathrm{o}(1)$, hence for any $k\in[K]$ we have $\sqrt{N}\left|\hat{\eta}_{(t)}^{\mathcal{I}_k}-\hat{\eta}_{(t)}^{\mathcal{I}_k,*}\right|=\mathrm{o}_{\mathrm{P}}(1)$. Thus the result holds. 
\proofend

\subsection{Numerical Results for Estimating Average Treatment Effects}\label{appendix:ATE:numerical}
We numerically illustrate average treatment effect estimation with $M=4$, yielding $16$ (vector) treatment levels (binary tuples). Setting $(0,0,0,0)$ as the (vector) control level, we estimate average treatment effects for the remaining $15$ treatment levels versus $(0,0,0,0)$. The conditional outcome models are set as $O_i(t)=\beta_t+\mu_t X_i+\epsilon_i$, where $\epsilon_i\indep X_i$, $\epsilon_i\overset{i.i.d.}{\sim}\mathcal{N}(0,1)$, $\beta_t=0.1H(t)$, $\mu_t=0.5+0.1H(t)$, where $H(t)=\sum_{j=1}^Mt_j2^{j-1}$. We generate $X_i$ uniformly from $\{0.1,0.2,0.3,0.4,0.5,0.6,0.7,0.8,0.9\}$.
Thus the average treatment effect of level $t$ is equal to $(\beta_t-\beta_0)+(\mu_t-\mu_0)\mathbb{E}[X_i]=0.15H(t)$, and the heterogeneous treatment effect of treatment level $t$ versus control $(0,0,0,0)$ is equal to $\tau(X_i)=(\beta_t-\beta_0)+(\mu_t-\mu_0)X_i$. The propensity scores are defined through Bahadur's representation: for any $t\in\{0,1\}^M$, $e_t(x)=(1+W(\alpha_0(x),t)'r_0(x))\prod_{j=1}^M\alpha_{0j}(x)^{t_j}(1-\alpha_{0j}(x))^{1-t_j}$, where $\alpha_{0}(x)$, $r_0(x)$ are computed in the same way as the with-covariate case simulation for $M=4$ in Section~\ref{sec:numerical:M:4}. So we obtain the semiparametric efficiency variance for treatment level $t$ as $V^*(t)=\mathrm{Var}[\tau(X_i)]+\mathbb{E}\left[\frac{\sigma_0(X_i)^2}{e_{0}(X_i)}\right]+\mathbb{E}\left[\frac{\sigma_t(X_i)^2}{e_t(X_i)}\right]$, where $\mathrm{Var}[\tau(X_i)]=(\mu_t-\mu_0)^2\mathrm{Var}(X_i)=\frac{0.01}{15}H(t)^2$,
$\sigma_0(X_i)^2=\mathrm{Var}[O_i|X_i,T_i=(0,0,0,0)]=1$, and $\sigma_t(x)^2=\mathrm{Var}[O_i|X_i=x,T_i=t]=1$, thus $V^*(t)=\frac{0.01}{15}H(t)^2+\mathbb{E}[1/e_0(X_i)]+\mathbb{E}[1/e_t(X_i)]$ for any $t\in\{0,1\}^M\setminus\{(0,0,0,0)\}$. We compare the performance of the AIPW estimator using four generalized propensity score estimators. The first is the multinomial (MNL) estimator, which fits a multinomial logit model of $T_i$ on $X_i$ and is widely used in multi-level treatment settings. The remaining three estimators are computed as $\hat{e}_t(x)=(1+W(\hat{\alpha}(x),t)'\hat{r}(x,x))\prod_{j=1}^M\hat{\alpha}_j(x)^{t_j}(1-\hat{\alpha}_j(x))^{1-t_j}$, where we employ the Nadaraya-Watson (NW estimator), plug-in ($w_{\textrm{I}}$), and first-order ($w_{\textrm{I}}$) estimators for $\hat{r}(x,x)$ to estimate the generalized correlation coefficients $r_0(x)$. The kernel bandwidth is chosen as $h \propto (\log(p)/N)^{1/5}$, ranging from $0.3$ to $0.4$ depending on the sample size $N$. Specifically, we use $\hat{r}_{l}^{NW}(x):= \frac{\sum_{i=1}^NK_h(X_i-x)\prod_{j\in\ell(l)}z_{j}(T_{i},\hat{\alpha}(X_i))}{\sum_{i=1}^NK_h(X_i-x)}$ as the $l$-th component of the NW estimator $\hat{r}^{NW}(x)$ to approximate \eqref{def:r0}. We checked the probability estimates before plugging them into the AIPW estimator and found that they are all positive, so this is not an issue in our simulations. In practice, if they fall outside the feasible range, one can truncate them before plugging them into the AIPW estimator. We replicate the simulations 200 times and calculate the coverage ratios of the semiparametric efficient confidence intervals for the true ATEs at each treatment level. Simulations are conducted for $s=0,2,5,10$ with total sample size $N=200,300,400$ in each case. The effective sample size for each treatment combination is equal to $N/16$ on average. Figure \ref{fig:coverage} presents the coverage ratios of true ATE for all estimators across combinations of $s$ and $N$, where the vertical axis corresponds to treatment levels and the horizontal axis to estimators. 
\begin{figure}[htp!]
    \centering
    \includegraphics[width=0.7\linewidth]{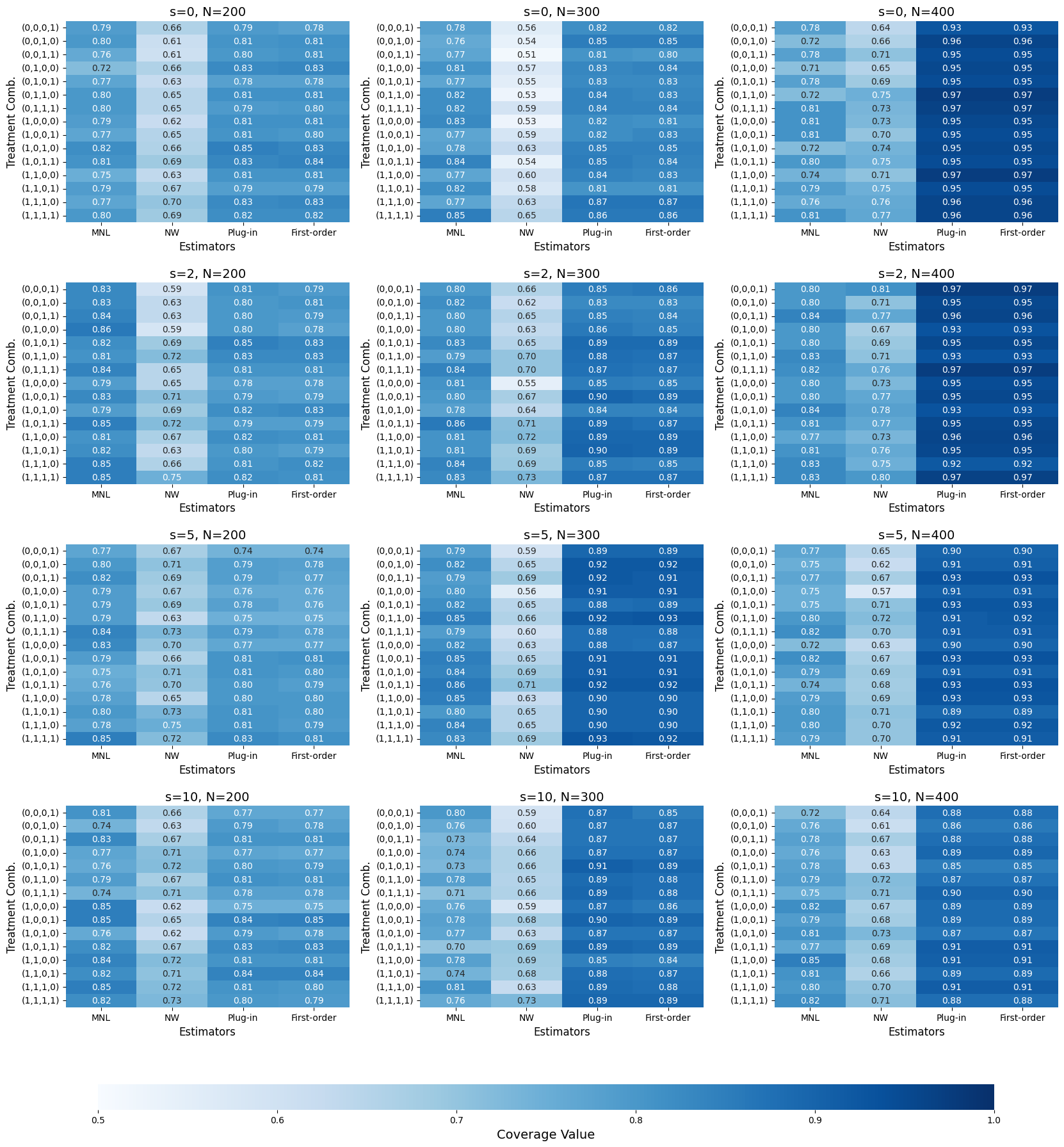}
    \caption{Coverage ratios of true ATEs under different propensity score estimators, across sparsity levels $s$ and sample sizes $N$, with $M=4$ (16 treatment combinations; $(0,0,0,0)$ as control).}
    \label{fig:coverage}
\end{figure}
Several key observations can be made from the numerical results:
First, the coverage ratios of the MNL model and the NW estimator remain stable across different values of $s$ and treatment levels. However, their coverage ratios increase only slightly as $N$ grows. Second, for the plug-in and first-order estimators, the coverage ratio generally increases as $s$ decreases and $N$ grows. In contrast, such pattern is not observed for the MNL model and the NW estimator, as both do not account for the underlying independence structure of the binary treatments. Third, the NW estimator performs significantly worse than both the plug-in and first-order estimators. 

\subsection{Numerical Results for $M=4$}\label{sec:numerical:M:4}
For the no-covariate case, we compare $\hat{r}^{\mathrm{PI}}$ (\ref{eq:PMLE-no-covariates}) and $\hat{r}^{\mathrm{FO}}$ (\ref{eq:PMLE-no-covariates-bilinear}) by running simulations with $s=2,M=4,N=500$ and $s=5, M=4, N=500$. When $s=2,M=4,N=500$, the marginal probability vector is set as $(0.613,0.491,0.653,0.510)$, $r_0 = (0,0,0,0,0, 0,0,-0.339,0,0,0.249)$. For $s=5, M=4, N=500$, the marginal probability vector is set as $(0.767,0.360,0.389,0.535)$, and $r_0=(0,-0.009,0,0,0.160,-0.011,0,0,0.031,0,-0.066)$. For the with-covariate case, We sample $N$ i.i.d. $X_i\sim\mathrm{Uniform}[0,1]$. For any $j\in[M]$, set $\theta=(0.1,0.2,0.3,0.4)$ and $\alpha_{0j}(X_i) = \frac{1}{1+e^{\theta_jX_i}}$ as the $j$-th marginal probability for $\alpha_0(X_i)$. We first choose $s$ entries of $r_0(\cdot)$ as its support. 
For each $i\in[N]$, $r_{0j}(X_i)=\{(-1)^j0.03jX_i\}$ if $j$ is in the support. 
Figure~\ref{fig:histogram} in EC provides the histogram of $1+W_i^\prime r_0$ and $1+W(\alpha_0(X_i),Y_i)^\prime r_0(X_i)$ over $1000$ replications for all numerical setups here. The values span a range not concentrated around $1$, indicating that the underlying probability distributions deviate substantially from the case of full independence, which highlights the generality of our setup. Table \ref{tab:simulation-results} in EC displays the results for $s=2,M=4,N=500$ and $s=5,M=4,N=500$ for the no-covariate cases, and the simulation results for $(s, M, N, x) = (2, 4, 100, 0.5)$ and $(5, 4, 100, 0.5)$ for the with-covariate cases. Each numerical result is based on 200 replications. The results indicate that the first-order estimators achieve lower RMSE than the plug-in estimator and provide more accurate probability estimates.
\begin{figure}[htp!]
    \centering
    \captionsetup[subfigure]{font=small,skip=3pt}

    \begin{subfigure}[t]{0.22\linewidth}
        \centering
        \includegraphics[width=\linewidth]{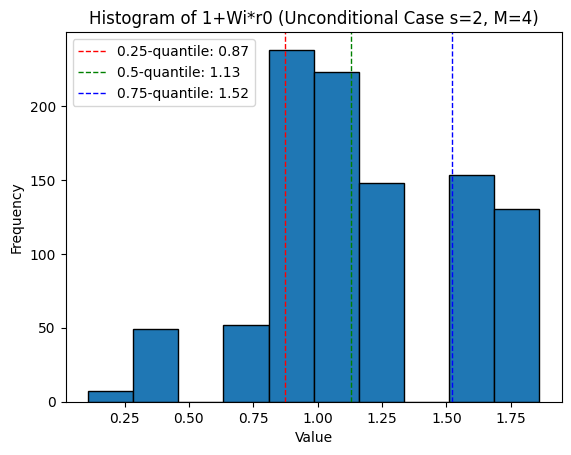}
        \caption{No covariates, $s=2$}
    \end{subfigure}\hfill
    \begin{subfigure}[t]{0.22\linewidth}
        \centering
        \includegraphics[width=\linewidth]{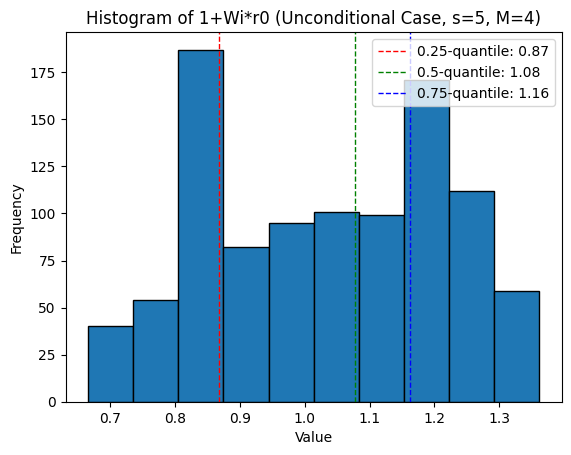}
        \caption{No covariates, $s=5$}
    \end{subfigure}\hfill
    \begin{subfigure}[t]{0.22\linewidth}
        \centering
        \includegraphics[width=\linewidth]{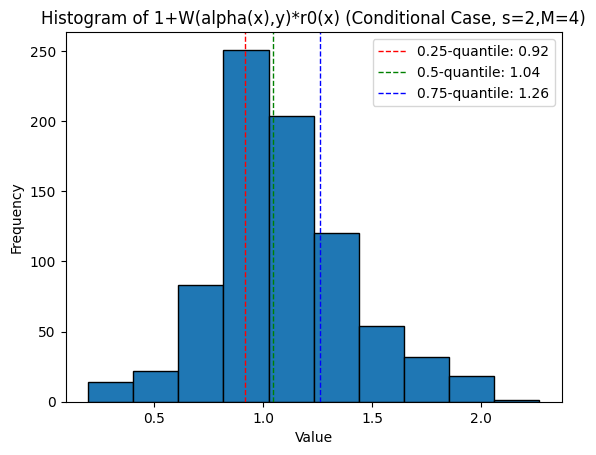}
        \caption{With covariates, $s=2$}
    \end{subfigure}\hfill
    \begin{subfigure}[t]{0.22\linewidth}
        \centering
        \includegraphics[width=\linewidth]{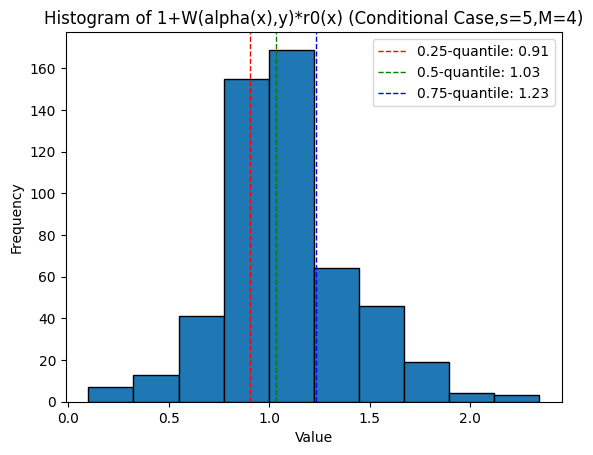}
        \caption{With covariates, $s=5$}
    \end{subfigure}

    \caption{Histograms of $1+W_i'r_0$ without covariates and
    $1+W(\alpha_0(X_i),Y_i)'r_0(X_i)$ with covariates, for
    $s\in\{2,5\}$ and $M=4$.}
    \label{fig:histogram}
\end{figure}

\begin{table}[htp!]
\centering
\small
\setlength{\tabcolsep}{4pt}
\renewcommand{\arraystretch}{0.9}
\begin{threeparttable}
\caption{Simulation results with and without covariates for $M=4$.}
\label{tab:simulation-results}
\begin{tabular}{lcccccc}
\toprule
& \multicolumn{2}{c}{RMSE}
& \multicolumn{2}{c}{$\mathbb{E}\!\left[\max_{y\in\{0,1\}^M}
\left|\hat{p}_y-p_y\right|\right]$}
& \multicolumn{2}{c}{$\mathbb{E}\!\left[\sum_{y\in\{0,1\}^M}
p_y\left|\hat{p}_y-p_y\right|\right]$} \\
\cmidrule(lr){2-3}\cmidrule(lr){4-5}\cmidrule(lr){6-7}
Estimator & $\lambda_{\mathrm{cv}}$ & $\lambda_{\mathrm{theory}}$
& $\lambda_{\mathrm{cv}}$ & $\lambda_{\mathrm{theory}}$
& $\lambda_{\mathrm{cv}}$ & $\lambda_{\mathrm{theory}}$ \\
\midrule
\multicolumn{7}{l}{\quad \textit{Panel A:} \textit{no-covariate} $M=4$, $N=500$, $s=2$} \\
PI ($w_{\mathrm{I}}$)         & 0.1203 & 0.3502 & 0.0138 & 0.0255 & 0.0089 & 0.0140 \\
PI ($w_{\mathrm{II}}$)        & 0.1304 & 0.3399 & 0.0171 & 0.0254 & 0.0061 & 0.0098 \\
FO ($w_{\mathrm{I}}$)         & 0.0913 & 0.1991 & 0.0136 & 0.0217 & 0.0070 & 0.0128 \\
FO ($w_{\mathrm{II}}$)        & 0.0902 & 0.2073 & 0.0141 & 0.0237 & 0.0057 & 0.0094 \\
PI (oracle $w_{\mathrm{I}}$)  & 0.0872 & 0.1079 & 0.0064 & 0.0205 & 0.0040 & 0.0045 \\
PI (oracle $w_{\mathrm{II}}$) & 0.0858 & 0.1394 & 0.0063 & 0.0201 & 0.0037 & 0.0052 \\
\addlinespace
\multicolumn{7}{l}{\quad \textit{Panel B:} \textit{no-covariate} $M=4$, $N=500$, $s=5$} \\
PI ($w_{\mathrm{I}}$)         & 0.0996 & 0.2809 & 0.0202 & 0.0320 & 0.0082 & 0.0143 \\
PI ($w_{\mathrm{II}}$)        & 0.0987 & 0.2840 & 0.0203 & 0.0292 & 0.0083 & 0.0125 \\
FO ($w_{\mathrm{I}}$)         & 0.0277 & 0.1551 & 0.0128 & 0.0271 & 0.0061 & 0.0113 \\
FO ($w_{\mathrm{II}}$)        & 0.0280 & 0.1548 & 0.0140 & 0.0247 & 0.0062 & 0.0116 \\
PI (oracle $w_{\mathrm{I}}$)  & 0.0182 & 0.0982 & 0.0102 & 0.0216 & 0.0054 & 0.0106 \\
PI (oracle $w_{\mathrm{II}}$) & 0.0187 & 0.0995 & 0.0106 & 0.0212 & 0.0053 & 0.0102 \\
\midrule
\multicolumn{7}{l}{\quad \textit{Panel C:} \textit{with-covariate} $M=4$, $s=2$, $N=100$, $x=0.5$, $h=0.1634$} \\
PI ($w_{\mathrm{I}}$)         & 0.0915 & 0.9042 & 0.0498 & 0.0524 & 0.0194 & 0.0195 \\
PI ($w_{\mathrm{II}}$)        & 0.0913 & 0.9058 & 0.0498 & 0.0524 & 0.0194 & 0.0195 \\
FO ($w_{\mathrm{I}}$)         & 0.0913 & 0.1681 & 0.0498 & 0.0498 & 0.0194 & 0.0194 \\
FO ($w_{\mathrm{II}}$)        & 0.0913 & 0.1682 & 0.0498 & 0.0498 & 0.0194 & 0.0194 \\
PI (oracle $w_{\mathrm{I}}$)  & 0.0912 & 0.1657 & 0.0070 & 0.0308 & 0.0056 & 0.0126 \\
PI (oracle $w_{\mathrm{II}}$) & 0.0912 & 0.1657 & 0.0071 & 0.0308 & 0.0056 & 0.0126 \\
\addlinespace

\multicolumn{7}{l}{\quad \textit{Panel D:} \textit{with-covariate} $M=4$, $s=5$, $N=100$, $x=0.5$, $h=0.1634$} \\
PI ($w_{\mathrm{I}}$)         & 0.2552 & 0.2758 & 0.0894 & 0.0895 & 0.0327 & 0.0333 \\
PI ($w_{\mathrm{II}}$)        & 0.2553 & 0.2749 & 0.0893 & 0.0895 & 0.0328 & 0.0332 \\
FO ($w_{\mathrm{I}}$)         & 0.2531 & 0.2539 & 0.0893 & 0.0892 & 0.0326 & 0.0332 \\
FO ($w_{\mathrm{II}}$)        & 0.2532 & 0.2538 & 0.0892 & 0.0890 & 0.0326 & 0.0332 \\
PI (oracle $w_{\mathrm{I}}$)  & 0.2432 & 0.2522 & 0.0327 & 0.0794 & 0.0134 & 0.0330 \\
PI (oracle $w_{\mathrm{II}}$) & 0.2410 & 0.2429 & 0.0324 & 0.0773 & 0.0133 & 0.0332 \\
\bottomrule
\end{tabular}
\begin{tablenotes}
\footnotesize
\item PI and FO denote the plug-in and first-order estimators. $w_{\mathrm{I}}$ and $w_{\mathrm{II}}$ denote weighting schemes (I) and (II). All estimation errors are rounded to four decimal places.
\end{tablenotes}
\end{threeparttable}
\end{table}

\end{APPENDICES}

\end{document}